\documentclass[10pt]{article}

\usepackage{enumerate,xspace}
\usepackage{enumitem}
\usepackage{amsmath,amssymb,wasysym}
\usepackage[all]{xy}
\usepackage{proof}
\usepackage[svgnames]{xcolor}
\usepackage{cmll}
\usepackage{stmaryrd}
\usepackage{array}
\usepackage{rotating}
\usepackage[title]{appendix}%

\usepackage{hyperref} 
\hypersetup{
    colorlinks,
    citecolor=red,
    filecolor=red,
    linkcolor=blue,
    urlcolor=red
}

\newtheorem{observation}{Remark}[section]
\newtheorem{lemma}[observation]{Lemma}  
\newtheorem{theorem}[observation]{Theorem}
\newtheorem{definition}[observation]{Definition}
\newtheorem{example}[observation]{Example}
\newtheorem{remark}[observation]{Remark}

\newtheorem{proposition}[observation]{Proposition} 
\newtheorem{corollary}[observation]{Corollary}

\makeatletter

\newdimen\w@dth

\def\rarr{\rightarrow}
\def\Rarr{\Rightarrow}
\def\op{\oplus}
\def\ot{\otimes}
\newcommand{\os}{\oslash}

\def\c{\,\colon}
\def\cv{\curlyvee}

\newcommand{\bzero}{{\bf 0}}
\newcommand{\bone}{{\bf 1}}

\def\bX{\mbox{${\mathcal X}$}}

\def\bX{\mbox{${\mathbb X}$}}
\def\bbX{\mbox{$\scriptscriptstyle{\mathbb X}$}}
\def\bB{\mbox{${\mathbb B}$}}

\def\bD{\mbox{${\mathbb D}$}}
\def\bY{\mbox{${\mathbb Y}$}}

\def\SZ{\mbox{${\mathrm{SZ}}$}}
\def\mix{\mathrm{mix}}

\def\CLDC{\mbox{$\mathbf{FCLDC}$}}

\begin{document}
\allowdisplaybreaks

\title{Cartesian Linearly Distributive Categories: Revisited}
\author{Rose Kudzman-Blais and Jean-Simon Pacaud Lemay}

\maketitle

\begin{abstract} 
Linearly distributive categories (LDC) were introduced by Cockett and Seely to provide alternative categorical semantics for multiplicative linear logic. In contrast to Barr’s $\ast$-autonomous categories, LDCs take multiplicative conjunction and disjunction as primitive notions. Thus, a LDC is a category with two monoidal products that interact via linear distributors. A cartesian linearly distributive category (CLDC) is a LDC whose two monoidal products coincide with categorical products and coproducts. Initially, it was believed that CLDCs and distributive categories would coincide, but this was later found not to be the case. Consequently, the study of CLDCs was not pursued further at the time. With recent developments for and applications of LDCs, there has been renewed interest in CLDCs. This paper revisits CLDCs, demonstrating strong structural properties they all satisfy and investigating two key classes of examples: posetal distributive categories and semi-additive categories. Additionally, a previously assumed class of CLDCs, the Kleisli categories of exception monads of distributive categories, is re-examined and it is showed that they are not, in fact, CLDCs.
\end{abstract}

\noindent \small \textbf{Acknowledgments.} The authors would first like to thank Richard Garner (who humbly declined to be a coauthor) for the many helpful, insightful, and productive discussions. The authors would also like to thank the Australian Category Theory group at Macquarie University for hosting Kudzman-Blais' visit to work with Lemay on this research. 

\tableofcontents

\section{Introduction}

Most categories of interest, whether from logic, theoretical computer science, algebra, or category theory more broadly, exhibit multiple monoidal products worth investigating. Within monoidal category theory, a significant body of research explores the distributivity of monoidal structures, that is, how two monoidal products interact. In general, this interaction is mediated by a natural transformation, known as a distributivity map, which may or may not be an isomorphism. Frequently, one or both of the monoidal structures is cartesian or cocartesian, in other words induced from finite products or coproducts respectively. 
 
Perhaps the most well-known definition in this area is that of a distributive category \cite{Cockett_1993}, a category with finite products and coproducts, whose products $\times$ distribute over its coproducts $+$ by an isomorphism:
\begin{align}\label{distributive:iso}
(A\times C)+(B\times C) \xrightarrow{\cong}(A+B)\times C
\end{align}
This concept was generalized to distributive monoidal categories and extended further with the introduction of rig categories \cite{Laplaza_1972}, also known as bimonoidal categories  \cite{Johnson_Yau_2024}, which are categories equipped with two arbitrary monoidal structures whose distributivity is mediated by monomorphisms. They are named as such because they categorify rigs, modeling the distributivity of multiplication over sums in a ri(n)g. 

Another important notion of a category with two monoidal structures is that of a duoidal category  \cite{Aguiar_Mahajan_2010}, which arises canonically when a monoidal category has either finite products or coproducts. For a duoidal category, distributivity is a natural transformation given by an interchange map:  
\begin{align}\label{distributive:duiodal} (A\star B)\diamond(C\star D)\rarr(A\diamond C)\star(B\diamond D) 
\end{align}
The theory of duoidal categories emerged during the study of braided monoidal categories, where the interchange law generalizes the canonical flip induced by braidings \cite{Joyal_Street_1993}. 

The development of linearly distributive categories (LDC) fits naturally within this research program. LDCs were first introduced by Cockett and Seely in \cite{Cockett_Seely_1992} as categorical semantics for multiplicative linear logic (MLL), as developed by Girard \cite{Girard_1987}. In this framework, the multiplicative conjunction $\ot$ and disjunction $\parr$ are taken as primitive notions, along with their interactions via linear distributors:
\begin{align}\label{distributive:LDC}
A\ot (B\parr C)\rarr (A\ot B)\parr C 
\end{align}
This contrasts with $\ast$-autonomous categories, introduced by Barr \cite{Barr_1979}, which provide the categorical semantics for MLL with negation \cite{Seely_1989}, where multiplicative conjunction $\otimes$ and linear negation $\ast$ are taken as primitive, with the multiplicative disjunction $\parr$ defined via de Morgan duality. Consequently, linear distributivity is implicit, following from the duality between $\otimes$ and $\parr$. However, linear distributivity is a fundamental concept, as it precisely captures the interaction between two binary operations necessary to model the cut rule. The significance of explicitly considering linear distributivity is evident in the extensive and continuously growing literature on LDCs. 

Initially, LDCs were referred to as weakly distributive categories, as linear distributivity (\ref{distributive:LDC}) was thought to be a weakening of the classical distributivity of products over coproducts found in distributive categories (\ref{distributive:iso}). It was believed that specializing to cartesian linearly distributive categories (CLDC), which are LDCs where $\otimes = \times$ is a product and $\parr = +$ is a coproduct, would recover distributive categories. However, this turned out to be false: linear distributivity and classical distributivity are, in some sense, orthogonal notions \cite{Cockett_Seely_1997_LDC}. Following this realization, while research on LDCs has remained central to categorical linear logic, the study of CLDCs was not pursued further at that time. 

With recent developments for and applications of LDCs, there has also been renewed interest in CLDCs. CLDCs remain an interesting structure worth exploring and deserving to be properly studied. As such, the main objective of this paper is to revisit the theory of CLDCs and investigate various classes of examples. This work follows the first named author's efforts to better understand the structure of CLDCs by developing a linearly distributive version of the Fox theorem \cite{Kudzman-Blais_2025}. 

By investigating CLDCs more closely, we uncovered structural properties that had not been previously explored, revealing distinctive and nuanced behavior within CLDCs. One such key observation is that in a CLDC, its terminal object is always preinitial and, dually, its initial object is always subterminal. This implies that every CLDC is, in fact, a \emph{mix} LDC. Furthermore, these mix maps can be used to show that, remarkably, in a CLDC, an object is preinitial if and only if it is subterminal. 

In our investigation of examples of CLDCs, two key classes emerged: posetal distributive categories and semi-additive categories. Indeed, it appears to often be the case that a CLDC must fall into one of these two categories via collapse theorems. There are already two well-known collapses of LDCs to posetal categories: the first being Joyal's paradox, and the second stemming from the orthogonality between linear and classical distributivity. In addition to these, we discovered two new collapse results: a CLDC must be semi-additive if it either has \emph{invertible} linear distributors or if it is \emph{isomix}. At first glance, these collapses seem to constrain the landscape of possible CLDCs. In an effort to develop new examples, we introduce a Grothendieck construction that generates new CLDCs, some of which are neither posetal nor semi-additive categories.

On the other hand, we also reconsider the Kleisli category of the exception monad of a distributive category, which was thought to be a CLDC according to \cite{Cockett_Seely_1997_LDC}. However, thanks to our better understanding of CLDCs, it soon became apparent that something was amiss and that Kleisli categories of exception monads for distributive categories are not, in fact, examples of CLDCs. Nevertheless, they are still LDCs, where $\parr$ is still the coproduct, but where $\otimes$ is the product on objects but defined differently on maps. As such, it follows that classical distributive restriction categories \cite{Cockett_Lemay_2024} are also LDCs (but not CLDCs). 

The outline of this paper is as follows. In Section \ref{sec:preliminaries}, we carefully review the necessary background theory on LDCs and the morphisms between them. Section \ref{sec:CLDC} revisits CLDCs, where we explore fundamental properties shared by all CLDCs and also develop the appropriate notions of morphisms between CLDCs. In Section \ref{sec:BDL}, the first significant class of CLDC examples is investigated: posetal distributive categories. Section \ref{sec:S-Add} examines the second key class of CLDCs: semi-additive categories. Two new collapse theorems are presented, along with a construction that maps a CLDC to a semi-additive category. Section \ref{sec:Kleisli} revisits the Kleisli category of the exception monad of a distributive category and explains why it is not a CLDC, despite still being a well-defined LDC. Finally, Section \ref{sec:Examples} discusses two simple constructions that provide further examples of CLDCs: products and fibrations. 

{\bf Conventions:} Arbitrary categories will be denoted as $\bX$ or $\bY$. In an arbitrary category $\bX$, objects are denoted by capital letters ($A$, $B$, $C$, etc.), while the maps are denoted by lowercase letters ($f$, $g$, $h$, etc.). Identity maps are denoted by $1_A$, while composition is denoted by $;$ and written in diagrammatic order, that is, the composition of $f: A \to B$ and $g: B \to C$ is $f;g: A \to C$ which first does $f$ then $g$. For an arbitrary (symmetric) monoidal category, we denote it simply as a triple $(\bX, \os, I)$ where $\bX$ is the underlying category, $\os$ is the monoidal product, and $I$ is the monoidal unit, and with $\os$-associator ${\alpha_\os}_{A,B,C}\c (A\os B)\os C\xrightarrow{\cong} A\os (B\os C)$, right $\os$-unitor ${u_\os^R}_{A}\c A \xrightarrow{\cong}  A\os I$, and left $\os$-unitor ${u_\os^L}_{A}\c A \xrightarrow{\cong} I\os A$ (and $\os$-symmetry ${{\sigma_\os}_{A, B}\c A\os B  \xrightarrow{\cong}  B\os A}$). 
 
\section{Background: Linearly Distributive Categories}\label{sec:preliminaries}

Before introducing CLDCs, we first need to provide some important background on LDCs. Beyond their introduction in \cite{Cockett_Seely_1997_LDC}, the theory of LDCs was further developed in a subsequent series of articles by Cockett and Seely, sometimes alongside co-authors Blute and Trimble \cite{Blute_Cockett_Seely_Trimble_1996, Cockett_Seely_1997, Cockett_Seely_1999}. We refer the reader to those papers for an in-depth introduction to LDCs. 

\subsection{Linearly Distributive Categories}

Our notation for LDCs follows more closely the notation introduced by Girard in \cite{Girard_1987}, in contrast to the notation used by Cockett and Seely in \cite{Cockett_Seely_1992}. We have chosen to do so as to not conflict with other standard categorical notation. As such, multiplicative conjunction is represented by $\ot$, known as tensor, with unit $\bone$, and multiplicative disjunction is $\parr$, known as par, with unit $\Bot$. 

\begin{definition}\cite[Sec 2.1]{Cockett_Seely_1997_LDC}\label{def:LDC}
A {\bf linearly distributive category} (LDC) is a category $\bX$ equipped with two monoidal structures: 
\begin{enumerate}[label=(\roman*)]
\item A {\bf tensor} monoidal structure $(\bX, \ot, \bone)$; 
\[ {\alpha_\ot}_{A,B,C}\c (A\ot B)\ot C\rarr A\ot (B\ot C) \quad  {u_\ot^R}_{A}\c A\rarr A\ot \bone \quad {u_\ot^L}_{A}\c A\rarr \bone\ot A \]
\item A {\bf par} monoidal structure $(\bX, \parr, \Bot)$;
\[{\alpha_\parr}_{A,B,C}\c A\parr(B\parr C)\rarr (A\parr B)\parr C \quad {u_\parr^R}_{A}\c A\parr\Bot\rarr A \quad {u_\parr^L}_{A}\c\Bot\parr A\rarr A\] 
\end{enumerate}
which also comes equipped with natural transformations: 
\begin{align*}
\delta^R_{A,B,C}\c (A\parr B)\ot C\rarr A\parr (B\ot C) && \delta^L_{A,B,C}\c A\ot(B\parr C)\rarr (A\ot B)\parr C,
\end{align*}
where $\delta^L$ and $\delta^R$ are called the left and right {\bf linear distributors} respectively, such that:
\begin{enumerate}[label=(\roman*)]	
\item The linear distributors are compatible with the unitors in the sense that the following equalities hold:
\begin{align}\label{cc:unit_lineardist}
\begin{split}
	&{u_\ot^L}_{A\parr B}; \delta^L_{\bone, A, B} = {u_\ot^L}_{A}\parr 1_{B}\\
	&{u_\ot^R}_{A\parr B}; \delta^R_{A, B,\bone} = 1_{A}\parr {u_\ot^R}_{B}\\
    & \delta^R_{\Bot, A, B};{u_\parr^L}_{A\ot B} = {u_\parr^L}_{A}\ot 1_{B} \\
    &\delta^L_{A, B, \Bot};{u_\parr^R}_{A\ot B} = 1_{A}\ot {u_\parr^R}_{B}
\end{split}
\end{align}
\item The linear distributors are compatible with the associators in the sense that the following equalities hold:
\begin{align}\label{cc:assoc_lineardist}
\begin{split}
	&\delta^L_{A\ot B, C, D};({\alpha_\ot}_{A,B,C}\parr 1_{D}) = {\alpha_\ot}_{A,B,C\parr D}; (1_{A}\ot \delta^L_{B,C,D}); \delta^L_{A,B\ot C, D} \\
	&({\alpha_\parr}_{A,B,C}\ot 1_{D});\delta^R_{A\parr B, C, D} =\delta^R_{A,B\parr C, D}; (1_{A}\parr \delta^R_{B,C,D});{\alpha_\parr}_{A,B,C\ot D}\\
    &\delta^L_{A, B, C\parr D};{\alpha_\parr}_{A\ot B, C, D} =  (1_{A}\ot {\alpha_\parr}_{B,C,D}); \delta^L_{A,B\parr C, D};(\delta^L_{A,B,C}\parr 1_{D})\\
	&{\alpha_\ot}_{A\parr B, C, D};\delta^R_{A, B, C\ot D} = (\delta^R_{A,B,C}\ot 1_{D}); \delta^R_{A,B\ot C, D}; (1_{A}\parr {\alpha_\ot}_{B,C,D})
\end{split}
\end{align}
\item The left and right linear distributors are compatible with each other in the sense the following equalities hold: 
\begin{align}\label{cc:leftright_lineardist}
\begin{split}
	& \delta^L_{A\parr B, C, D}; (\delta^R_{A,B,C}\parr 1_{D}) = \delta^R_{A,B,C\parr D}; (1_{A}\parr \delta^L_{B,C,D}); {\alpha_\parr}_{A,B\ot C, D} \\
	&(\delta^L_{A,B,C}\ot 1_{D});\delta^R_{A\ot B, C, D} = {\alpha_\ot}_{A, B\parr C, D}; (1_{A}\ot \delta^R_{B,C,D}); \delta^L_{A,B,C\ot D}
\end{split} 
\end{align}
    \end{enumerate}	
\end{definition}

It may be useful to draw out the commutative diagrams for some of the axioms. For example, here are the first equalities of (\ref{cc:unit_lineardist}), (\ref{cc:assoc_lineardist}), and (\ref{cc:leftright_lineardist}) respectively. 
\begin{equation*}
\begin{gathered}\xymatrixrowsep{1.75pc}\xymatrixcolsep{2.75pc}\xymatrix{
A\parr B\ar[r]^-{{u_\ot^L}_{A\parr B}}\ar[rd]_-{{u_\ot^L}_{A}\parr 1_{B}} & \bone\ot(A\parr B)\ar[d]^-{\delta^L_{\bone, A, B}}\\ 
& (\bone\ot A)\parr B 
}
\end{gathered}\end{equation*}
\begin{equation*}
\begin{gathered}\xymatrixcolsep{3pc}\xymatrix{
(A\ot B) \ot (C\parr D)\ar[dd]_-{\delta^L_{A\ot B, C, D}}\ar[r]^-{{\alpha_\ot}_{A,B,C\parr D}}& A\ot (B\ot (C\parr D))\ar[d]^-{1_{A}\ot \delta^L_{B,C,D}} \\
& A\ot ((B\ot C)\parr D)\ar[d]^-{\delta^L_{A,B\ot C, D}} \\
((A\ot B)\ot C) \parr D\ar[r]_-{{\alpha_\ot}_{A,B,C}\parr 1_{D}} & (A\ot (B\ot C))\parr D
}
\end{gathered}\end{equation*}
\begin{equation*}
\begin{gathered}\xymatrix{
& (A\parr B) \ot (C\parr D)\ar[dl]_-{\delta^R_{A,B,C\parr D}}\ar[dr]^-{\delta^L_{A\parr B, C, D}} \\
A\parr (B\ot (C\parr D))\ar[d]_-{1_{A}\parr \delta^L_{B,C,D}} & & ((A\parr B)\ot C) \parr D\ar[d]^-{\delta^R_{A,B,C}\parr 1_{D}} \\
A \parr ((B\ot C)\parr D)\ar[rr]_-{{\alpha_\parr}_{A,B\ot C, D}} & &(A\parr (B\ot C))\parr D
}
\end{gathered}\end{equation*}
We leave drawing out the other axioms as commutative diagrams as an exercise. 

The standard variant of MLL is commutative and this is modeled by considering \emph{symmetric} LDCs. 

\begin{definition}\cite[Sec 3]{Cockett_Seely_1997_LDC}\label{def:SLDC}
A  {\bf symmetric linearly distributive category} (SLDC) is a LDC $\bX$ such that $(\bX, \ot, \bone)$ and $(\bX,\parr,\Bot)$ are symmetric monoidal categories, and the linear distributors are compatible with the symmetries in the sense that the following diagram commutes: 
\begin{equation}\label{cc:lin_dist_braiding}
\begin{gathered}\xymatrix{
(A\parr B) \ot C\ar[r]^-{\delta^R_{A,B,C}} \ar[d]_-{{\sigma_\ot}_{A\parr B, C}} & A\parr (B\ot C) \\
C\ot (A\parr B)\ar[d]_-{1_{C} \ot {\sigma_\parr}_{A,B}} & (B\ot C) \parr A\ar[u]_-{{\sigma_\parr}_{B\ot C, A}}\\
C\ot (B\parr A)\ar[r]_-{\delta^L_{C,B,A}} & (C\ot B)\parr A \ar[u]_-{{\sigma_\ot}_{C,B}\parr 1_{A}} 
}
\end{gathered}\end{equation}
\end{definition}

It is worth mentioning that for SLDCs, we can in fact suppress one of the linear distributors and simply define one from the other. For example, we can define a SLDC simply in terms of the left linear distributor $\delta^L$, drop the coherence conditions involving $\delta^R$ from (\ref{cc:unit_lineardist}) and (\ref{cc:assoc_lineardist}), and replace (\ref{cc:leftright_lineardist}) by two diagrams involving the left linear distributor, associativities and symmetries. This remains a LDC as we define the right linear distributor $\delta^R$ via \eqref{cc:lin_dist_braiding}. 
\subsection{Shifted Tensors}

The simplest example of a LDC is a monoidal category, where the linear distributors are simply the associativity isomorphisms of the monoidal product. 

\begin{example} Every (symmetric) monoidal category $(\bX, \ot, \bone)$ can be viewed as a (symmetric) LDC $(\bX, \ot, \bone, \ot, \bone)$, when taking the tensor and par structures to be equal to the original monoidal structure, i.e. $\ot=\parr$ and $\bone= \Bot$. In this case, the linear distributors are just the $\ot$-associators:
\begin{align*}
\delta^R_{A,B,C} ={\alpha_\ot}_{A,B,C}\c (A\ot B)\ot C  \xrightarrow{\cong}  A\ot (B\ot C) \\
\delta^L_{A,B,C} ={\alpha^{-1}_\ot}_{A,B,C}\c A\ot (B\ot C) \xrightarrow{\cong}  (A\ot B)\ot C
\end{align*}
These LDCs are known as {\bf degenerate}.
\end{example}

Degenerate LDCs are examples of LDCs where the linear distributors are isomorphisms. In fact, LDCs whose linear distributors are isomorphisms are completely characterized by the categorical analogue of \emph{shift monoids}. These are defined via invertible objects in a monoidal category, with which we can define a \emph{shifted tensor}. 

\begin{definition}\cite[Sec 5.2]{Cockett_Seely_1997_LDC} In a monoidal category $(\bX, \ot, \bone)$, an object $\Bot\in\bX$ is said to have a {\bf $\ot$-inverse} if there is an object $\Bot^{-1}$ equipped with two isomorphisms $s^L\c\Bot \ot \Bot^{-1}\rarr\bone$ and $s^R\c \Bot^{-1}\ot\Bot\rarr\bone$ such that the following diagram commutes: 
\begin{equation}\label{cc:shifted_tensor}
\begin{gathered}\xymatrixrowsep{1pc}\xymatrixcolsep{1.75pc}\xymatrix{
(\Bot^{-1}\ot \Bot) \ot \Bot^{-1} \ar[rr]^-{{\alpha_\ot}_{\Bot^{-1},\Bot, \Bot^{-1}}} \ar[d]_-{s^R\ot 1_{\Bot^{-1}}}& &\Bot^{-1}\ot(\Bot\ot\Bot^{-1}) \ar[d]^-{1_{\Bot^{-1}}\ot s^L}\\
\bone\ot\Bot^{-1}\ar[rd]_-{{u^L_\ot}^{-1}_{\Bot^{-1}}} & & \Bot^{-1}\ot \bone \ar[ld]^-{{u^R_\ot}^{-1}_{\Bot^{-1}}}\\
& \Bot^{-1}
}
\end{gathered}\end{equation}
\end{definition}    
    
If $\Bot$ is an object in a monoidal category $(\bX, \ot, \bone)$ with a $\ot$-inverse, we can define a new monoidal structure $\parr$ on $\bX$ where $A\parr B = A\ot (\Bot^{-1}\ot B)$, called the {\bf $\Bot$-shifted tensor}, with monoidal unit is $\Bot$. 
 
\begin{proposition}\cite[Prop 5.3 \& 5.4]{Cockett_Seely_1997_LDC}\label{prop:shift_tensor} Let $(\bX, \ot, \bone)$ be a monoidal category with an object $\Bot\in\bX$ with a $\otimes$-inverse. Then $(\bX, \ot, \bone, \parr,\Bot)$, where $\parr$ is the $\Bot$-shifted tensor, is a LDC, known as a \textbf{shifted-tensor} LDC, with  invertible linear distributors defined as follows: 
\begin{align*}
&\delta^R_{A,B,C} = {\alpha_\ot}_{A, \Bot^{-1}\ot B,C}; (1_A\ot {\alpha_\ot}_{\Bot^{-1}, B, C})\c (A\ot (\Bot^{-1}\ot B))\ot C \xrightarrow{\cong} A\ot (\Bot^{-1}\ot (B\ot C)) \\
&\delta^L_{A,B,C} = {\alpha_\ot^{-1}}_{A,B,\Bot^{-1}\ot C}\c A\ot (B\ot (\Bot^{-1}\ot C))\xrightarrow{\cong} (A\ot B)\ot (\Bot^{-1}\ot C)
\end{align*}
Conversely, for every LDC with invertible linear distributors, the $\parr$-unit $\Bot$ has a $\otimes$-inverse, $\Bot^{-1} := \bone \parr \bone$, and moreover $\parr$ is naturally equivalent to the $\Bot$-shifted tensor. 
\end{proposition}
    
In particular, for any monoidal category $(\bX, \ot, \bone)$, the monoidal unit $\bone$ is its own $\ot$-inverse and the induced $\bone$-shifted tensor is (up to natural isomorphism) the starting monoidal product $\ot$.  

\subsection{Mix}

The categorical semantics of many variants of MLL can be expressed using (symmetric) LDCs. Among these variants, an important one is MLL with the MIX rule $A\ot B \dashv A\parr B$, which is equivalent to the nullary MIX rule $\Bot\dashv \bone$ in the presence of the cut rule. As such, this amounts to asking that our LDC has a map from the $\parr$-unit to the $\otimes$-unit. 

\begin{definition}\cite[Def 6.2]{Cockett_Seely_1997}
A LDC is {\bf mix} if there is a map $m\c\Bot\rarr\bone$, called the {\bf nullary mix map}, such that the following diagram commutes: 
\begin{equation}\label{cc:mix_LDC}
\begin{gathered}\xymatrixrowsep{1.75pc}\xymatrixcolsep{2.75pc}\xymatrix{
A\ot B\ar[r]^-{1_A\ot {u^L_\parr}_B^{-1}}\ar[d]_-{{u^R_\parr}_A^{-1}\ot 1_B}& A\ot (\Bot \parr B)\ar[r]^-{1_A\ot(m\parr 1_B)} & A\ot (\bone \parr B)\ar[d]^-{\delta^L_{A, \bone, B}} \\
(A\parr \Bot)\ot B\ar[d]_-{(1_A\parr m)\ot 1_B} & & (A\ot \bone)\parr B \ar[d]^-{{u^R_\ot}_A^{-1}\parr 1_B} \\
(A\parr \bone)\ot B \ar[r]_-{\delta^R_{A,\bone,B}} & A\parr (\bone \ot B)\ar[r]_-{1_A\parr {u^L_\ot}_B^{-1}} & A\parr B
}
\end{gathered}\end{equation}
The induced natural transformation $\mix_{A,B}\c A\ot B\rarr A\parr B$, defined by the equivalent composites above, is called the {\bf mix map}. 
\end{definition}

The mix map is of course natural, as well as compatible with the associators and unitors (and symmetry) in the obvious way. For example, here is the compatibility between the mix map and the associators, which we will need later on. 

\begin{lemma}\label{lem:mix_maps_associators}
In a mix LDC \bX, the following diagrams commute:
\begin{equation}\begin{gathered}\label{cc:mix_maps_associators}
\xymatrixrowsep{1.75pc}\xymatrixcolsep{2pc}\xymatrix{
A\ot (B\ot C)\ar[r]^-{{\alpha_\ot}^{-1}_{A,B,C}}\ar[d]_-{1_A\ot \mix_{B,C}} & (A\ot B)\ot C\ar[d]_-{\mix_{A\ot B, C}} & (A\ot B)\ot C\ar[r]^-{{\alpha_\ot}_{A,B,C}}\ar[d]^-{\mix_{A,B}\ot 1_C} & A\ot (B\ot C) \ar[d]^-{\mix_{A,B\ot C}}\\
A\ot (B\parr C)\ar[r]^-{\delta^L_{A,B,C}}\ar[d]_-{\mix_{A,B\parr C}} & (A\ot B)\parr C\ar[d]_{\mix_{A,B}\parr 1_C} & (A\parr B)\ot C\ar[r]^-{\delta^R_{A,B,C}}\ar[d]^-{\mix_{A\parr B, C}} & A\parr (B\ot C)\ar[d]^-{1_A\parr \mix_{B, C}}\\
A\parr (B\parr C)\ar[r]_{{\alpha_\parr}_{A,B,C}} & (A\parr B)\parr C & (A\parr B)\parr C\ar[r]_-{{\alpha_\parr}^{-1}_{A,B,C}} & A\parr (B\parr C)
}
\end{gathered}\end{equation}
\end{lemma}

The following is a useful lemma which greatly simplifies checking if a LDC is mix. 

\begin{lemma}\cite[Lem 6.2]{Cockett_Seely_1997}\label{lem:prove_mix1}
A LDC is mix if and only if there is a map $m\c\Bot\rarr\bone$ for which \eqref{cc:mix_LDC} holds for any one of the following cases: (1) $A= B=\bone$, (2) $A= B= \Bot$, (3) $A= \Bot$ and $B=\bone$, or (4) $A=\bone$ and $B=\Bot$. 
\end{lemma}

There is a stronger version of the nullary MIX rule, which often holds in categorical models, that is $\Bot \dashv\vdash \bone$. In this case, we have that $\Bot\cong\bone$. 

\begin{definition}\cite[Def 6.5]{Cockett_Seely_1997}
A LDC is {\bf isomix} if it is mix and the nullary mix map ${m\c\Bot\rarr\bone}$ is an isomorphism.
\end{definition}

In fact, we do not need to check \eqref{cc:mix_LDC} if the two monoidal units are isomorphic.

\begin{lemma}\cite[Lem 6.6]{Cockett_Seely_1997}\label{lem:isomix}
A LDC for which there is an isomorphism $\bone\cong\Bot$ is isomix.
\end{lemma}

It is important to stress that in an isomix LDC, even if $\bone\cong\Bot$, the induced mix map $\mix_{A,B}\c A\ot B\rarr A\parr B$ is not necessarily an isomorphism. 

\begin{definition}\cite[Sec 2.3]{Cockett_Comfort_Srinivasan_2021}\label{def:compact_LDC}
A LDC is {\bf compact} if it is isomix and the mix maps $\mix_{A,B}\c A\ot B\rarr A\parr B$ are isomorphisms. 
\end{definition}

Of course, every degenerate LDC is a compact LDC, where the mix maps are identities. Moreover, by Lem \ref{lem:mix_maps_associators}, the linear distributors of a compact LDC are essentially the associators (modulo the mix maps) and, in particular, are invertible. Thus we can further equivalently characterize compact LDCs as the isomix LDCs with invertible distributors.  

\subsection{Complements and Negation}

The most important class of LDCs are arguably \emph{$\ast$-autonomous categories}. Unlike LDCs, $\ast$-autonomous categories make the multiplicative conjunction and linear negation primitive, then define the multiplicative disjunctions by de Morgan duality.

\begin{example} Briefly recall that a $\ast$-autonomous category can be defined as a symmetric monoidal category $(\bX, \ot, \bone)$ equipped with a full and faithful contravariant functor $\ast$ such that there are natural isomorphisms $\bX(A\ot B, C^\ast)\cong\bX(A, (B\ot C)^\ast)$. Then every $\ast$-autonomous category is a SLDC, with the tensor structure given by the original monoidal structure and the par structure given by de Morgan duality, $A \parr B = (B^\ast \ot A^\ast)^\ast$ and $\Bot = \bone^\ast$. Some notable examples of $\ast$-autonomous categories include the category of finite-dimensional vector spaces, the category of sup-lattices, the category of coherence spaces and the category of finiteness spaces. 
\end{example}

In fact $\ast$-autonomous categories are precisely the SLDCs with \emph{negation}. However, before reviewing negation, we take a slight detour and talk about complemented objects (which we will also need for the results in Sec \ref{sec:prein-subterm}). 

\begin{definition}\cite[Def A.5]{Cockett_Seely_1999}
In a LDC, a \textbf{complementation pair} is a tuple $(A,A^c, \gamma, \tau)$ consisting of an object $A$, called the \textbf{left complement}, an object $A^c$, called the \textbf{right complement}, and maps $\gamma\c A\ot A^c \rarr \Bot$, called the \textbf{evaluation}, and $\tau\c \bone\rarr A^c\parr A$, called the \textbf{coevaluation}, such that the following diagrams commute: 
\begin{equation}\label{cc:complement_LDC}
\begin{gathered}\xymatrixrowsep{1.75pc}\xymatrixcolsep{2.5pc}\xymatrix{
A\ar[d]_-{{u^R_\ot}_A}\ar[r]^-{1_A} & A & A^c\ar[d]_-{{u^L_\ot}_A^c}\ar[r]^-{1_{A^c}} & A^c\\
A\ot \bone\ar[d]_-{1_A\ot \tau} & \Bot\ot A\ar[u]_-{{u^L_\parr}_A} & \bone\ot A^c\ar[d]_-{\tau\ot 1_{A^c}}  & A^c\parr \Bot\ar[u]_-{{u^R_\parr}_{A^c}} \\
A\ot (A^c\parr A)\ar[r]_-{\delta^L_{A,A^c,A}} & (A\ot A^c)\parr A\ar[u]_-{\gamma\parr 1_A} & (A^c\parr A)\ot A^c\ar[r]_-{\delta^R_{A^c,A,A^c}} & A^c\parr (A\ot A^c)\ar[u]_-{1_{A^c}\parr \gamma}
}
\end{gathered}\end{equation}
\end{definition}

In any LDC, the $\otimes$-unit and the $\parr$-unit always form a complementation pair. 

\begin{proposition}\cite[Prop A.7]{Cockett_Seely_1999}\label{prop:bot_top_complements}
In a LDC, $(\Bot, \bone, {u^R_\ot}_\Bot^{-1}\c\Bot\ot \bone\rarr\Bot,{u^R_\parr}^{-1}_\bone\c \bone \rarr\bone \parr \Bot)$ is a complementation pair. 
\end{proposition}

As one might expect, complementation pairs induce adjunctions.

\begin{lemma}\cite[Lem 4.4]{Cockett_Seely_1997_LDC}\label{lem:complement_pairs_bijection}
Let $(A,A^c, \gamma, \tau)$ be a complementation pair in a LDC \bX. Then there is an adjunction $A\ot - \dashv A^c \parr -$, in other words there is a natural isomorphism between the hom-sets $\bX(A\ot B, C)\cong \bX(B, A^c\parr C)$.
\end{lemma}

Viewed through the lens of linear logic, given a complementation pair $(A,A^c, \gamma, \tau)$, we may see $A^c$ as essentially the ``right negation'' of $A$, and dually that $A$ is the ``left negation'' of $A^c$. Thus we say that a LDC has negation if there is a coherent way of associating to every object both a left negation and a right negation. In standard commutative MLL modeled by a SLDC, we require that the left negation and right negation be the same. 

\begin{definition}\cite[Def 4.1, 4.3]{Cockett_Seely_1997_LDC}\label{def:SLDC_negation}
A SLDC has {\bf negation} if there is an object function $(-)^\perp$, together with the following parametrized family of maps 
\begin{align*}
&\gamma_A^R\c A\ot A^\perp \rarr \Bot & \tau^R_A\c \bone\rarr A\parr A^\perp
\end{align*}
which additionally induce the following families,
\[ \gamma_A^L = A^\perp \ot A \xrightarrow{{\sigma_\ot}_{A^\perp, A}} A\ot A^\perp \xrightarrow{\gamma^R_A} \Bot \qquad\qquad \tau^L_A = \bone \xrightarrow{\tau^R_A} A\parr A^\perp \xrightarrow{{\sigma_\parr}_{A,A^\perp}} A^\perp\parr A\]
such that $(A, A^\perp, \gamma^R_A, \tau^L_A)$ and $(A^\perp, A, \gamma^L_A, \tau^R_A)$ form complementation pairs. 
\end{definition}

\begin{theorem}\cite[Thm 4.5]{Cockett_Seely_1997_LDC}
The notions of SLDCs with negation and $\ast$-autonomous categories coincide.
\end{theorem}

\subsection{Linear Functors}

We conclude this background section by reviewing the appropriate notion of morphisms between LDCs. The definition was determined in order to develop the appropriate version of exponential modalities for LDCs. We do not review the full definition here, as we will not necessarily need to know all the axioms for the story of this paper, and instead invite the curious reader to see \cite{Cockett_Seely_1999} for full details. 

\begin{definition}\cite[Def 1]{Cockett_Seely_1999} A {\bf (bilax) linear functor} $F = (F_\ot, F_\parr)\c\bX\rarr\bY$ between LDCs consists of:
\begin{enumerate}[label=(\roman*)]
\item A (lax) $\ot$-monoidal functor $(F_\ot, m_\bone, m_\ot)\c(\bX,\ot,\bone)\rarr(\bY,\ot,\bone)$, with underlying functor ${F_\ot: \bX \to \bY}$, unit isomorphism (resp. map) $m_\bone\c\bone\rarr F_\ot(\bone)$, and natural monoidal isomorphism (resp. transformation) \\${{m_\ot}_{A, B}\c F_\ot(A)\ot F_\ot(B)\rarr F_\ot(A\ot B)}$; 
\item A (colax) $\parr$-monoidal functor $(F_\parr, n_\Bot, n_\parr)\c(\bX,\parr,\Bot)\rarr(\bY,\parr,\Bot)$, with underlying functor $F_\parr: \bX \to \bY$, unit isomorphism (resp. map) $n_\Bot\c F_\parr(\Bot)\rarr\Bot$, and natural comonoidal isomorphism (resp. transformation) ${n_\parr}_{A,B}\c F_\parr(A\parr B)\rarr F_\parr(A)\parr F_\parr(B)$;
\end{enumerate}
equipped with four natural transformations, known as the {\bf linear strengths}, 
        \begin{align*}
      {v_\ot^R}_{A,B} \c F_\ot(A\parr B)\rarr F_\parr(A)\parr F_\ot(B) &&  {v_\ot^L}_{A,B} \c F_\ot(A\parr B)\rarr F_\ot(A)\parr F_\parr(B)\\
 {v_\parr^R}_{A,B} \c F_\ot(A)\ot F_\parr(B)\rarr F_\parr(A\ot B) &&  {v_\parr^L}_{A,B} \c F_\parr(A)\ot F_\ot(B)\rarr F_\parr(A\ot B)
        \end{align*}
such that various coherence conditions detailed in \cite{Cockett_Seely_1999} hold. 
\end{definition}

We shall also be particularly interested in restricting our attention to linear functors whose component functors are in fact equal. These were first defined in \cite{Blute_Panangaden_Slanov_2012} under the name degenerate linear functors. They were renamed {\em Frobenius} linear functors in \cite{Cockett_Comfort_Srinivasan_2021} as they generalize Frobenius monoidal functors. We shall use this latter terminology.

\begin{definition}\cite[Def 3.1]{Cockett_Comfort_Srinivasan_2021}
A {\bf (bilax) Frobenius} linear functor $F=(F_\ot,F_\parr)\c \bX\rarr\bY$ is a (bilax) linear functor such that $F_\ot = F_\parr$, ${\nu^R_\ot}_{A,B} = {\nu^L_\ot}_{A,B} = {n_\parr}_{A,B}$, and ${\nu^R_\parr}_{A,B} = {\nu^L_\parr}_{A,B} = {m_\ot}_{A,B}$.
\end{definition}

Given the degeneracy, we can give an alternative characterization of Frobenius linear functors.

\begin{proposition}\cite[Lem 3.2]{Cockett_Comfort_Srinivasan_2021} For LDCs $\bX$ and $\bY$, the following notions coincide:
\begin{enumerate}[label=(\roman*)]
    \item A (bilax) Frobenius linear functor $\bX\rarr\bY$;
    \item A functor $F: \bX \to \bY$ which is (lax) $\ot$-monoidal $(F, m_\bone, m_\ot)\c(\bX,\ot,\bone)\rarr(\bY,\ot,\bone)$ and (colax) $\parr$-monoidal $(F, n_\Bot, n_\parr)\c(\bX,\parr,\Bot)\rarr(\bY,\parr,\Bot)$ such that the following diagrams commute: 
    \begin{equation}\label{cc:Frobenius_linear_functor}
\begin{gathered}\xymatrixrowsep{1.75pc}\xymatrixcolsep{1.25pc}\xymatrix@L=0.3pc{
F(A)\ot F(B\parr C)\ar[r]^-{{m_\ot}_{A,B\parr C}}\ar[d]_-{1_{F(A)}\ot {n_\parr}_{B,C}} & F(A\ot (B\parr C))\ar[r]^-{F(\delta^L_{A,B,C})} & F((A\ot B)\parr C)\ar[d]^-{{n_\parr}_{A\ot B, C}} \\
F(A) \ot (F(B)\parr F(C))\ar[r]_-{\delta^L_{F(A), F(B), F(C)}} & (F(A) \ot F(B))\parr F(C)\ar[r]_-{{m_\ot}_{A,B}\parr 1_{F(C)}} & F(A\ot B)\parr F(C) \\
F(A\parr B)\ot F(C)\ar[r]^-{{m_\ot}_{A\parr B, C}}\ar[d]_-{{n_\parr}_{A,B}\ot 1_{F(C)}} & F((A\parr B)\ot C)\ar[r]^-{F(\delta^R_{A, B, C})} & F(A\parr (B\ot C))\ar[d]^-{{n_\parr}_{A, B\ot C}} \\
(F(A)\parr F(B))\ot F(C)\ar[r]_-{\delta^R_{F(A), F(B), F(C)}} & F(A)\parr (F(B)\ot F(C))\ar[r]_-{1_{F(A)}\parr {m_\ot}_{B, C}} & F(A) \parr F(B\ot C) 
}
\end{gathered}\end{equation}
\end{enumerate}
\end{proposition}

So from now on we will denote Frobenius linear functors simply as $F\c\bX \to \bY$. If the LDCs are mix, the definition of Frobenius linear functors can be slightly extended to guarantee that they preserve the mix maps.

\begin{definition}\cite[Def 3.4]{Cockett_Comfort_Srinivasan_2021} For mix LDCs $\bX$ and $\bY$, a (bilax) Frobenius linear functor $F\c \bX\rarr\bY$ is {\em mix} if the following diagram commutes:
\begin{equation}\label{cc:mix_Frobenius_functor}
\begin{gathered}\xymatrixrowsep{1.75pc}\xymatrixcolsep{1.75pc}\xymatrix{
F(\Bot)\ar[r]^-{n_\Bot}\ar[d]_-{F(m)} & \Bot\ar[d]^-{m} \\
F(\bone) & \bone\ar[l]^-{m_\bone}
}
\end{gathered}\end{equation}
\end{definition}

\begin{proposition}\cite[Lem 3.5]{Cockett_Comfort_Srinivasan_2021}
A mix (bilax) Frobenius linear functor $F\c \bX\rarr\bY$ preserves the mix maps, that is, following diagram commutes: 
\begin{equation}
\begin{gathered}\xymatrixrowsep{1.75pc}\xymatrixcolsep{2.75pc}\xymatrix{
F(A)\ot F(B)\ar[r]^-{\mix_{F(A), F(B)}}\ar[d]_-{{m_\ot}_{A,B}} & F(A)\parr F(B) \\
F(A\ot B)\ar[r]_-{F(\mix_{A,B})} & F(A\parr B)\ar[u]_-{{n_\parr}_{A,B}}
}
\end{gathered}\end{equation}
\end{proposition}

With the introduction of linear functors, we can be more precise about the relationship between the different examples and classes of LDCs previously introduced. The following results are implicit in the work of Cockett and Seely, but have yet to explicitly appear in the literature, so we record them here.

\begin{proposition}
Let $(\bX, \ot, \bone, \parr,\Bot)$ be a compact LDC. Then the $\ot$-monoidal structure is isomorphic to the $\parr$-monoidal structure, $(\bX, \ot,\bone)\cong(\bX, \parr, \Bot)$. Moreover, it is isomorphic (as a LDC) to the degenerate LDC induced by the $\ot$-monoidal structure and similarly induced by the $\parr$-monoidal structure, $(\bX, \ot, \bone, \parr,\Bot)\cong (\bX, \ot, \bone, \ot,\bone)\cong (\bX, \parr, \Bot, \parr,\Bot)$.
\end{proposition}
\begin{proof}
Given a compact LDC $(\bX, \ot, \bone, \parr,\Bot)$, the identity functor $1_{\bbX}\c \bX\rarr\bX$ becomes a monoidal functor in four ways. Firstly and secondly, we have the identity monoidal functors $(1_{\bbX}, 1_\bone, 1_{A\ot B})\c (\bX, \ot, \bone)\rarr (\bX, \ot, \bone)$ and $(1_{\bbX}, 1_\Bot, 1_{A\parr B})\c (\bX, \parr, \Bot)\rarr (\bX, \parr, \Bot)$. Thirdly and fourthly, $(1_{\bbX}, m, \mix^{-1}_{A,B})\c(\bX, \ot, \bone)\rarr(\bX, \parr,\Bot)$ and $(1_{\bbX}, m^{-1}, \mix_{A,B})\c(\bX, \parr,\Bot)\rarr(\bX, \ot, \bone)$ are also monoidal functors. These latter two are inverses which mediate the isomorphism $(\bX, \ot,\bone)\cong(\bX, \parr, \Bot)$. 

Furthermore, pairing the first and last monoidal functor structures determines a Frobenius linear functor \\$(1_{\bbX}, 1_\bone,1_{A\ot B}, m^{-1}, \mix_{A,B})\c (\bX, \ot, \bone, \ot, \bone)\rarr (\bX, \ot, \bone, \parr,\Bot)$. Indeed, the Frobenius linear functor coherence diagrams \eqref{cc:Frobenius_linear_functor} simply become the top squares of \eqref{cc:mix_maps_associators}, which commute in any mix LDC. Similarly, pairing the first and third structures determines its inverse $(1_{\bbX}, 1_\bone,1_{A\ot B}, m, \mix^{-1}_{A,B})\c (\bX, \ot, \bone, \parr,\Bot)\rarr (\bX, \ot, \bone, \ot, \bone)$, mediating the isomorphism $(\bX, \ot, \bone, \parr,\Bot)\cong (\bX, \ot, \bone, \ot,\bone)$. Pairing the second monoidal functor with the third and then the fourth provides the isomorphism $(\bX, \ot, \bone, \parr,\Bot)\cong (\bX, \parr, \Bot, \parr,\Bot)$.
\end{proof}

\begin{proposition} Let $(\bX, \ot, \bone, \parr, \Bot)$ be a LDC with invertible linear distributors. Then we have that the $\parr$-monoidal structure is isomorphic to the $\Bot$-shifted tensor structure, $(\bX, \parr, \Bot)\cong (\bX, \cdot\ot((\bone\parr\bone)\ot\cdot), \Bot)$. Moreover, it is isomorphic (as a LDC) to the $\Bot$-shifted LDC, $(\bX, \ot, \bone, \parr, \Bot)\cong (\bX,\ot, \bone, \cdot\ot((\bone\parr\bone)\ot\cdot), \Bot)$.
\end{proposition}
\begin{proof}
By Prop \ref{prop:shift_tensor}, $\Bot$ has a $\ot$-inverse given by $\bone\parr\bone$ and $\cdot\ot((\bone\parr\bone)\ot\cdot)$ is a well-defined monoidal product with unit $\Bot$. Now, define  $\beta_{A,B}$ to be the following natural isomorphism
\[ A\ot ((\bone\parr\bone)\ot B) \xrightarrow{1_A\ot \delta^R_{\bone, \bone, B}} A\ot (\bone\parr(\bone\ot B)) \xrightarrow{\delta^L_{A, \bone, \bone\ot B}} (A\ot \bone)\parr (\bone\ot B)\xrightarrow{{u^R_\ot}^{-1}_A\parr {u^L_\ot}^{-1}_B} A\parr B \]
Then, $(1_{\bbX}, 1_\Bot, \beta_{A,B})\c (\bX, \parr,\Bot)\rarr (\bX, \cdot\ot((\bone\parr\bone)\ot\cdot), \Bot)$ is a monoidal functor with inverse $(1_{\bbX}, 1_\Bot, \beta^{-1}_{A,B})$, mediating the desired isomorphism between monoidal categories. 

Pairing the identity $\ot$-monoidal functor with the above structure determines a Frobenius linear functor \\$(1_{\bbX},  1_\Bot, 1_{A\ot B}, 1_\Bot, \beta^{-1}_{A,B})\c(\bX, \ot, \bone, \parr, \Bot)\rarr (\bX,\ot, \bone, \cdot\ot((\bone\parr\bone)\ot\cdot), \Bot)$. The coherence conditions \eqref{cc:Frobenius_linear_functor} follow from a straightforward diagram chase. Similarly, it has an inverse  $(1_{\bbX},  1_\Bot, 1_{A\ot B}, 1_\Bot, \beta_{A,B})\c (\bX,\ot, \bone, \cdot\ot((\bone\parr\bone)\ot\cdot), \Bot)\rarr (\bX, \ot, \bone, \parr, \Bot)$, providing the desired isomorphism of LDCs.
\end{proof}

\section{Cartesian Linearly Distributive Categories}\label{sec:CLDC}

In this section, we discuss CLDCs, which are of course the principal notion of study in this paper. After reviewing their definition, we will revisit how a CLDC is in fact mix, and provide our first novel observations about CLDCs, notably their duoidal structure and that they have a preinitial-subterminal coincidence. 

By a {\bf cartesian category} we mean a category with finite products, where we denote the binary product by $\times$ and the terminal object by $\top$. Projections maps are denoted as ${\pi^0_{A,B}: A \times B \to A}$ and $\pi^1_{A,B}: A \times B \to B$, with pairing operation $\langle -, - \rangle$, and we denote the unique map to the terminal object by $t_A: A \to \top$. Dually, a {\bf cocartesian category} is a category with finite coproducts, where we denote the binary coproduct by $+$ and the initial object by $\bot$. Injections maps are denoted as ${\iota^0_{A,B}: A \to A + B}$ and $\iota^1_{A,B}: B \to A + B$, with pairing operation $[ -, - ]$, and we denote the unique map from the initial object by $b_A: \bot \to A$. A (co)cartesian category \bX\ canonically becomes a symmetric monoidal category $(\bX, \times, \top)$ (resp. $(\bX, +, \bot)$). We say that a symmetric monoidal category is (co)cartesian if its monoidal product is a (co)product and its monoidal unit is a (resp. initial) terminal object.

\subsection{Cartesian Linearly Distributive Categories}

We can now introduce the main definition of interest: 

\begin{definition}\cite[Sec 2]{Cockett_Seely_1997_LDC}
A {\bf cartesian linearly distributive category} (CLDC) \bX\ is a SLDC whose tensor monoidal structure is cartesian and whose par monoidal structure is cocartesian. 
\end{definition}

So for a CLDC, its tensor is a product $\times$, its tensor unit is a terminal object $\top$, its par is a coproduct $+$, its par unit is an initial object $\bot$, and its linear distributors are of type:
	\begin{align*}
	 \delta^R_{A,B,C}\c (A + B)\times C\rarr A + (B\times C) && \delta^L_{A,B,C}\c A\times (B+ C)\rarr (A \times B) + C.
	\end{align*}
While it is a straightforward and natural concept to consider, we shall see that this gives rise to a highly peculiar form of distributivity, leading to surprisingly strong properties and results. We refrain from giving examples in this section, as that will be a key topic discussed in Sec \ref{sec:BDL}, \ref{sec:S-Add} and  \ref{sec:Examples}. We begin our investigation of CLDCs with a useful result concerning the interaction between the linear distributors and certain projections and injections.

\begin{lemma}\label{lem:delta_pi_iota}
In a CLDC, the following equations hold: 

\begin{align}\label{eqn:linear_dist_projection_injection}
\begin{split}
    &\delta^L_{A,B,C}; (\pi^1_{A,B} + 1_C) = \pi^1_{A, B+C} \qquad\qquad (1_A\times \iota^0_{B,C}); \delta^L_{A,B,C} = \iota^0_{A\times B, C} \\
    &\delta^R_{A,B,C}; (1_A+\pi^0_{B,C}) = \pi^0_{A+B, C} \qquad\qquad (\iota^1_{A,B}\times 1_{C}); \delta^R_{A,B,C} = \iota^1_{A, B\times C}
\end{split}
\end{align}
\end{lemma}

\begin{proof} The first equation follows from the fact that in a cartesian category, $\pi^1 =  (t \times 1); {u_\times^L}^{-1}$, and the unit coherence of the linear distributor \eqref{cc:unit_lineardist}: 
\begin{gather*}
    \delta^L_{A,B,C}; (\pi^1_{A,B} + 1_C) = \delta^L_{A,B,C}; ((t_A\times 1_B)+ 1_C); ({u^L_\times}^{-1}_B + 1_C) \\
  =  (t_A\times 1_{B+C}); \delta^L_{\top, B, C}; ({u^L_\times}^{-1}_B + 1_C) =  (t_A\times 1_{B+C}); {u_\times^L}^{-1}_{B+C} =  \pi^1_{A. B+C}
\end{gather*}
The other equations are proved in similar (dual) fashion. 
\end{proof}

The appropriate definition of morphism between CLDCs is one that preserves products and coproducts, which from a LDC perspective simply amounts to asking for linear functors (since asking that a functor preserves (co)products is equivalent to asking it be monoidal). We will also be particularly interested in restricting our attention to Frobenius cartesian linear functors.

\begin{definition}
A {\bf (Frobenius) cartesian linear functor} between CLDCs is a (Frobenius) linear functor between the CLDCs.  
\end{definition}

We denote the category of CLDCs and (Frobenius) cartesian linear functors between them by $\mathbf{CLDC}$ (resp. \CLDC). 

\subsection{Subterminal Initial and Preinitial Terminal}\label{sec:prein-subterm}

Recall that in a category, an object $A$ is {\bf preinitial} if there is at most one map from $A$ to another object, and dually {\bf subterminal object} if there is at most one map from another object to $A$. It turns out that in a CLDC, the initial object is always \emph{subterminal} and the terminal object is always \emph{preinitial}. This is quite a strong property which highlights the uniqueness of CLDCs and also helps quickly check if a category is not a CLDC. Here are first some useful equivalent characterizations of being subterminal (resp. preinitial) in (co)cartesian categories. 

\begin{lemma}\cite[Lem 3.2]{Cockett_1993}\label{lem:preinitial/subterminal} In a cartesian category, the following are equivalent: 
\begin{enumerate}[label=(\roman*)]
    \item $A$ is subterminal,
    \item $\Delta_A = \langle 1_A, 1_A\rangle\c A\rarr A\times A$ is an isomorphism,
    \item $\pi^0_{A,A} = \pi^1_{A,A}\c A\times A \rarr A$.
\end{enumerate}
Dually, in a cocartesian category, the following are equivalent:
\begin{enumerate}[label=(\roman*)]
    \item $A$ is preinitial,
    \item $\nabla_A = [1_A, 1_A]\c A+A\rarr A$ is an isomorphism,
    \item $\iota^0_{A,A} = \iota^1_{A,A}\c A\rarr A+A$.
\end{enumerate}
\end{lemma}

\begin{proposition}\label{prop:bot_subterminal_top_preinitial} In a CLDC, the initial object $\bot$ is subterminal and the terminal object $\top$ is preinitial. 
\end{proposition}
\begin{proof}
First recall that by Prop \ref{prop:bot_top_complements}, $\bot$ and $\top$ form a complementation pair. Then by Lem \ref{lem:complement_pairs_bijection}, $\bX(\bot\times\bot, \bot)\cong \bX(\bot, \top+\bot)$. As $\bot$ is the initial object, there is a unique map of type $\bot \to \top+\bot$, and therefore there is unique map of type $\bot\times\bot \to \bot$. As such, the projections $\pi^0_{\bot,\bot}$ and $ \pi^1_{\bot,\bot}$ must be equal. Then by Lem \ref{lem:preinitial/subterminal}, $\bot$ is subterminal. Similarly, by Lem \ref{lem:complement_pairs_bijection}, we also have that $\bX(\top, \top+\top)\cong \bX(\bot\times\top,\top)$. As $\top$ is the terminal object, there is a unique map of type $\bot\times\top \to \top$, and thus there is a unique map of type $\top \to \top+\top$. So the injections $\iota^0_{\top,\top}$ and $\iota^1_{\top,\top}$ must be equal. So by Lem \ref{lem:preinitial/subterminal} again, $\top$ is preinitial.
\end{proof}

Since $\bot$ is subterminal and $\top$ is preinitial, we can ask for which objects $A$ is there a map $A \to \bot$ or $\top \to A$. This is precisely the case when $A\cong A\times \bot$ or $A\cong A+\top$ respectively. 

\begin{lemma}\label{lem:A_iso_A_times_bot} Let $A$ be an object in a CLDC. Then: 
\begin{enumerate}[label=(\roman*)]
    \item There exists a map $A\rarr\bot$ if and only if $\pi^0_{A,\bot}\c A\times \bot\rarr A$ is an isomorphism;
    \item There exists a map $\top\rarr A$ if and only if $\iota^0_{A,\top}\c A \rarr A + \top$ is an isomorphism.
\end{enumerate}
\end{lemma}
\begin{proof} We prove $(i)$. For the $\Rightarrow$ direction, suppose that $\pi^0_{A,\bot}\c A\times \bot\rarr A$ is an isomorphism, so $A\cong A\times \bot$. Then we can construct the map $A\cong A\times \bot \xrightarrow{\pi^1_{A,\bot}} \bot$ as desired. For the $\Leftarrow$ direction, suppose there exists a map $a\c A\rarr\bot$. Then consider the map  $\langle 1_A, a\rangle\c A\rarr A\times\bot$. By definition, we have that $\langle 1_A, a\rangle; \pi^0_{A,\bot} = 1_A$. On the other hand, we obviously have that $\pi^0_{A,\bot}; \langle 1_A, a\rangle; \pi^0_{A,\bot} = \pi^0_{A,\bot}$, while we can also compute that: 
\[ \pi^0_{A,\bot}; \langle 1_A, a\rangle; \pi^1_{A,\bot} = \pi^0_{A,\bot}; a = \pi^1_{A,\bot}\]
where the last equality holds since there is a unique map $A\times\bot\rarr\bot$ as $\bot$ is subterminal. Then by the universal property of the product we get that $\pi^0_{A,\bot}; \langle 1_A, a\rangle = 1_{A\times \bot}$. Thus $\pi^0_{A,\bot}$ is an isomorphism with inverse $\langle 1_A, a\rangle$. By similar arguments, one can also show $(ii)$. 
\end{proof}

\subsection{Mix Structure}

In a CLDC, there is a unique map from the initial object to the terminal object, using the universal property of either object. In fact, this unique map is a nullary mix map, meaning that a CLDC is mix. This fact was first proved by Cockett and Seely in \cite{Cockett_Seely_1997}. However the proof they give is somewhat involved as it works with the \emph{nucleus} (the subcategory of complemented objects) and then applies Joyal's paradox (which we review in Sec \ref{sec:posets}) to conclude that a CLDC must be mix. Here we give an alternative streamlined version of their proof, utilizing the same core ideas and making use of the fact that the initial object is subterminal.

\begin{proposition}\cite[Rem 5.4]{Cockett_Seely_1997}
A CLDC is mix with $m = b_\top = t_\bot \c \bot\rarr\top$. 
\end{proposition}
\begin{proof}
By Lem \ref{lem:prove_mix1}, it suffices to prove \eqref{cc:mix_LDC} for $\bot\times\bot\rarr\bot+\bot$, which holds by commutativity of the following diagram: 

\centerline{
\xymatrixrowsep{2.75pc}\xymatrixcolsep{2.75pc}\xymatrix{
\bot\times \bot\ar[r]^-{1_\bot\times {u_+}^{-1}_\bot}\ar[d]_-{{u_+}^{-1}_\bot\times 1_\bot}\ar@{=}[rd] & \bot\times (\bot+\bot)\ar[rr]^-{1_\bot\times(m+1_\bot)}\ar[d]_-{1_\bot \times {u_+}_\bot} & \ar@{}[d]_{({\rm nat})} & \bot\times (\top+\bot)\ar[d]^-{\delta^L_{\bot,\top,\bot}}_*!/_8pt/{\scriptscriptstyle{\eqref{cc:unit_lineardist}}}\ar[ld]_-{1_\bot \times {u_+^R}_\top} \\
(\bot+\bot)\times \bot\ar@{}[rd]|{({\rm nat})}\ar[dd]_-{(1_\bot +m)\times 1_\bot}\ar[r]^-{{u_+}_\bot\times 1_\bot} & \bot\times \bot\ar[r]^-{1_\bot\times m}\ar[d]_-{m\times 1_\bot}\ar[rd]^-{\pi^0_{\bot,\bot}}_-{\pi^1_{\bot,\bot}} & \bot\times \top\ar@{}[rd]|{({\rm nat})}\ar[r]_-{{u^R_+}_{\bot\times\top}^{-1}}\ar[d]^-{{u^R_\times}^{-1}_\bot} & (\bot\times\top)+\bot\ar[dd]^-{{u^R_\times}^{-1}_\bot +1_\bot} \\
&\top\times \bot\ar[d]^-{{u^L_+}^{-1}_{\top\times\bot}}_*!/_8pt/{\scriptscriptstyle{\eqref{cc:unit_lineardist}}}\ar[r]_-{{u^L_\times}^{-1}_\bot} & \bot\ar[rd]^-{{u_+}^{-1}_\bot} &\\
(\bot+\top)\times\bot\ar[ur]^-{{u^L_+}_\top\times 1_\bot}\ar[r]_-{\delta^R_{\bot,\top,\bot}} & \bot+(\top\times\bot)\ar[rr]_-{1_\bot + {u_\times^L}_\bot^{-1}} & \ar@{}[u]^{({\rm nat})}& \bot+\bot
}}

\end{proof}

Therefore, for all pairs of objects $A$ and $B$ in a CLDC, there exists a mix map from their product to their coproduct: $\mix_{A,B}\c A\times B \rarr A+B$. 

\begin{lemma}\label{lem:mix_cartesian_Frobenius_linear_functor}
A cartesian Frobenius linear functor is a mix Frobenius linear functor.
\end{lemma}
\begin{proof}
By definition of a cartesian Frobenius linear functor, $n_\bot$ and $m_\top$ are isomorphisms. Then, as $\bot$ is initial (or $\top$ is terminal), we get $m = n_\bot^{-1}; F(m); m_\top^{-1}$, meaning \eqref{cc:mix_Frobenius_functor} holds. 
\end{proof}

\subsection{Preinitial and Subterminal Coincidence}

We now wish to show that in a CLDC, being preinitial is equivalent to being subterminal. To do so, let us first give an equivalent characterization for preinitial or subterminal objects using the mix maps. 

\begin{lemma}\label{lemma:CLDC_preintial_subterminal}
Let $A$ be an object in a CLDC. Then: 
\begin{enumerate}[label=(\roman*)]
    \item $A$ is preinitial if and only if the following equalities hold:
    \begin{align}\label{eq:Delta-mix-nabla}
    \Delta_A; \mix_{A,A}; \nabla_A = 1_A&&  \nabla_{A}; \Delta_A; \mix_{A,A}= 1_{A+ A}
    \end{align}
    \item $A$ is subterminal if and only if the following equations hold.
        \begin{align}
         \Delta_A; \mix_{A,A}; \nabla_A = 1_A && \mix_{A,A};\nabla_{A}; \Delta_A= 1_{A\times A}
            \end{align}
\end{enumerate}
\end{lemma}
\begin{proof} We prove $(i)$. For the $\Rightarrow$ direction, suppose that $A$ is preinitial. Then there is a unique map of type $A \to A$, and thus $\Delta_A; m_{A,A}; \nabla_A = 1_A$. On the other hand, by Lem \ref{lem:preinitial/subterminal}, $\nabla_A^{-1}$ exists, so we compute: 
\[ \nabla_{A}; \Delta_A; \mix_{A,A} = \nabla_{A}; \Delta_A; \mix_{A,A}; \nabla_A; \nabla_A^{-1} = \nabla_A; 1_A; \nabla_A^{-1} = 1_{A+A}\]
So \eqref{eq:Delta-mix-nabla} holds as desired. For the $\Leftarrow$ direction, suppose that \eqref{eq:Delta-mix-nabla} holds. Observe that this says that $\nabla_A$ is an isomorphism with inverse $\Delta_A; \mix_{A,A}\c A\rarr A+A$. So by Lem \ref{lem:preinitial/subterminal}, $A$ is preinitial. By similar arguments, one can also show $(ii)$. 
\end{proof}

Using the above lemma, we can prove the desired equivalence between being preinitial and being subterminal in a CLDC. In fact, we do so by showing that for a preinitial or subterminal object $A$, the mix map $A \times A \to A + A$ must be of a specific form. 

\begin{proposition}\label{prop:CLDC_preintial_subterminal}
Let $A$ be an object in a CLDC. Then the following are equivalent: 
\begin{enumerate}[label=(\roman*)]
    \item $A$ is preinitial;
    \item $A$ is subterminal; 
    \item The following diagram commutes:
\begin{equation}\begin{gathered}\label{eqn:mix_pi_iota}
\xymatrixrowsep{1pc}\xymatrixcolsep{1.5pc}\xymatrix{
& A\ar[rd]^-{\iota^0_{A,A}} & \\
A\times A\ar[ru]^-{\pi^0_{A,A}}\ar[rd]_-{\pi^1_{A,A}}\ar[rr]^-{\mix_{A,A}} && A+A \\
& A\ar[ru]_-{\iota^1_{A,A}} &
}
\end{gathered}\end{equation}
\end{enumerate}
\end{proposition}
\begin{proof}
We prove $(i) \Leftrightarrow (iii)$. For $(i) \Rightarrow (iii)$, suppose $A$ is preinitial. Then the following diagrams commute: 
\begin{equation*}\begin{gathered}
\xymatrixrowsep{1.75pc}\xymatrixcolsep{3.25pc}\xymatrix{
A\times A\ar[r]^-{1_A\times {u^L_+}_A^{-1}}\ar[dd]_-{\pi^0_{A,A}}\ar[rd]^-{1_A\times t_A} & A\times (\bot+A) \ar@{}[d]|{({\rm preinit})}\ar[r]^-{1_A\times (m+1_A)} & A\times (\top+A)\ar[d]^-{\delta^L_{A,\top,A}}  \\
\ar@{}[r]|{({\rm def})}& A\times\top\ar[ru]^-{1_A\times\iota^0_{\top,A}}\ar[r]_-{\iota^0_{A\times\top,A}}^{\eqref{eqn:linear_dist_projection_injection}}\ar[ld]^-{{u^R_\times}^{-1}_A}\ar@{}[d]|{({\rm nat})} & (A\times\top)+A\ar[d]^-{{u_\times^R}_A^{-1}+1_A} \\ 
A\ar[rr]_-{\iota^0_{A,A}} & & A+A
}
\end{gathered}\end{equation*}
\begin{equation*}\begin{gathered}
\xymatrixrowsep{1.75pc}\xymatrixcolsep{3.25pc}\xymatrix{
A\times A\ar[r]^-{{u^R_+}_A^{-1}\times 1_A}\ar[dd]_-{\pi^1_{A,A}}\ar[rd]^-{t_A\times 1_A} & (A+\bot)\times A \ar@{}[d]|{({\rm preinit})}\ar[r]^-{(1_A+m)\times 1_A}& (A+\top)\times A\ar[d]^-{\delta^R_{A,\top,A}} \\
\ar@{}[r]|{({\rm def})} & \top\times A\ar[ru]^-{\iota^1_{A,\top}\times 1_A}\ar[r]_-{\iota^1_{A,\top\times A}}^{\eqref{eqn:linear_dist_projection_injection}}\ar[ld]^-{{u^L_\times}^{-1}_A}\ar@{}[d]|{({\rm nat})} & A+(\top\times A)\ar[d]^-{{u^R_\times}^{-1}_A+1_A}\\ 
A\ar[rr]_-{\iota^1_{A,A}} & & A+A
}
\end{gathered}\end{equation*}
By definition, the upper composites of the above diagrams are the two equivalent definitions of $\mix_{A,A}$. For $(iii) \Rightarrow (i)$, suppose \eqref{eqn:mix_pi_iota} holds. Then we first compute that: 
\[ \Delta_A; \mix_{A,A}; \nabla_A = \Delta_A; \pi^j_{A,A}; \iota^j_{A,A}; \nabla_A = 1_A; 1_A = 1_A\]
We can also compute that: 
\[ \iota^0_{A,A}; \nabla_A; \Delta_A; \mix_{A,A} =\iota^0_{A,A};  \nabla_A; \Delta_A; \pi^0_{A,A}; \iota^0_{A,A} = 1_A; 1_A; \iota^0_{A,A} = \iota^0_{A,A} \]
\[ \iota^1_{A,A}; \nabla_A; \Delta_A;\mix_{A,A} =\iota^1_{A,A}; \nabla_A;\Delta_A; \pi^1_{A,A}; \iota^1_{A,A} = 1_A; 1_A; \iota^1_{A,A} = \iota^1_{A,A} \] 
Then by the universal property of the coproduct, $\nabla_A; \Delta_A; \mix_{A,A} = 1_{A+A}$. Then by Lem \ref{lemma:CLDC_preintial_subterminal}, $A$ is preinitial. 

We now show $(ii) \Leftrightarrow (iii)$. For $(ii) \Rightarrow (iii)$, suppose that $A$ is subterminal, then the following diagrams commute
\begin{equation*}\begin{gathered}
\xymatrixrowsep{1.75pc}\xymatrixcolsep{3.25pc}\xymatrix{
A\times A\ar[dd]_-{\pi^0_{A,A}}\ar[r]^-{{u^R_+}^{-1}_A\times 1_A}\ar@{}[rdd]|{({\rm nat})} & (A+\bot)\times A\ar[d]_-{\pi^0_{A+\bot,A}}\ar[r]^-{\delta^R_{A,\bot,A}}_{\eqref{eqn:linear_dist_projection_injection}} & A+(\bot\times A)\ar[d]^-{1_A+(m\times 1_A)}\ar[ld]^-{1_A+\pi^0_{\bot,A}}\\
& A+\bot\ar[rd]^-{1_A+b_A}\ar@{}[r]|{({\rm subterm})}\ar@{}[d]|{({\rm def})} & A+(\top\times A)\ar[d]^-{1_A+{u^L_\times}^{-1}_A} \\
A\ar[ru]^-{{u^R_+}_A^{-1}}\ar[rr]_-{\iota^0_{A,A}} & & A+A
}
\end{gathered}\end{equation*}
\begin{equation*}\begin{gathered}
\xymatrixrowsep{1.75pc}\xymatrixcolsep{3.25pc}\xymatrix{
A\times A\ar[dd]_-{\pi^1_{A,A}}\ar[r]^-{1_A\times {u^L_+}^{-1}_A}\ar@{}[rdd]|{({\rm nat})} & A\times (\bot+A)\ar[d]_-{\pi^1_{A+\bot,A}}\ar[r]^-{\delta^L_{A,\bot,A}}_{\eqref{eqn:linear_dist_projection_injection}} & (A\times\bot)+A \ar[d]^-{(1_A\times m)+1_A}\ar[ld]^-{\pi^1_{A,\bot}+1_A}\\
& \bot+A\ar[rd]^-{b_A+1_A}\ar@{}[r]|{({\rm subterm})}\ar@{}[d]|{({\rm def})} & (A\times\top)+A\ar[d]^-{{u^R_\times}^{-1}_A+1_A} \\
A\ar[ru]^-{{u^L_+}_A^{-1}}\ar[rr]_-{\iota^1_{A,A}} & & A+A
}
\end{gathered}\end{equation*}
Observe that the upper composites are two other equivalent definitions of $\mix_{A,A}$ (by naturality of the linear distributors). For $(iii) \Rightarrow (ii)$, suppose that \eqref{eqn:mix_pi_iota} holds. Then above we already showed that $\Delta_A; \mix_{A,A}; \nabla_A  = 1_A$. We can also compute that: 
\[ \mix_{A,A}; \nabla_A; \Delta_A; \pi^0_{A,A} = \pi^0_{A,A};\iota^0_{A,A};\nabla_A; \Delta_A; \pi^0_{A,A} = \pi^0_{A,A}; 1_A; 1_A = \pi^0_{A,A} \]
\[ \mix_{A,A}; \nabla_A; \Delta_A; \pi^1_{A,A} = \pi^1_{A,A};\iota^1_{A,A};\nabla_A; \Delta_A; \pi^1_{A,A} = \pi^1_{A,A}; 1_A; 1_A = \pi^1_{A,A} \]
Then by the universal property of the product, $\mix_{A,A}; \nabla_A; \Delta_A = 1_{A \times A}$. Then by Lem \ref{lemma:CLDC_preintial_subterminal}, $A$ is subterminal. 

Putting all this together we get $(i) \Leftrightarrow (iii) \Leftrightarrow (ii)$. So we conclude that $A$ is preinitial if only if $A$ is subterminal, as desired. 
\end{proof}

\subsection{Duoidal Structure}

As discussed in the introduction, duoidal categories are another important type of category with two monoidal structures with some sort of distributivity. Here we explain how every CLDC is duoidal and study its duoidal structure. For an in-depth introduction to duoidal categories, we invite the reader to see \cite{Aguiar_Mahajan_2010}, where they are referred to as 2-monoidal categories.  

\begin{definition}\cite[Def 6.1]{Aguiar_Mahajan_2010}\label{def:duoidal_cat}
A {\bf duoidal category} $(\bX, \diamond, I,\star, J)$ is category \bX\ with two monoidal structures  $(\bX, \diamond, I)$ and $(\bX, \star, J)$, and equipped with morphisms 
\begin{alignat*}{3}
\Delta_{I}  &\c I\rarr I\star I &\qquad  \nabla_{J}& \c J\diamond J\rarr J  &\qquad  m  & \c I\rarr J
\end{alignat*}
and a natural transformation:
\begin{align*}
 \mu_{A, B, C, D}\c (A\star B)\diamond(C\star D)\rarr(A\diamond C)\star(B\diamond D)
\end{align*}
called the {\bf interchange}, such that various coherence conditions detailed in \cite{Aguiar_Mahajan_2010} hold. 
\end{definition}

Duoidal structures arise canonically whenever monoidal categories have finite products or finite coproducts. Therefore, every category which has both finite products and finite coproducts is a duoidal category, as detailed in \cite[Ex 6.19]{Aguiar_Mahajan_2010}. Explicitly, the structure maps $\Delta_{\bot} \c \bot\rarr \bot\times \bot$, $\nabla_{\top}\c \top+\top\rarr \top$, and $m\c \bot\rarr\top$ are defined respectively as follows: 
\begin{align}\label{eq:duoidal-CLDC-1}
\Delta_{\bot} = b_{\bot\times \bot} = \langle 1_\bot, 1_\bot\rangle && \nabla_{\top} = t_{\top+\top} = [1_\top, 1_\top] &&  m = t_{\bot}=b_{\top}
\end{align}
while the interchange $ \mu_{A,B,C,D} \c (A\times B)+(C\times D)\rarr (A+C)\times (B+D)$ is defined as follows:
\begin{align}\label{eq:duoidal-CLDC-2}
     \mu_{A,B,C,D} &= \langle \pi^0_{A,B} + \pi^0_{C,D}, \pi^1_{A,B} +\pi^1_{C,D}\rangle = [\iota^0_{A,C} \times \iota^0_{B,D}, \iota^1_{A,C} \times \iota^1_{B,D}] 
\end{align}
Therefore, since every category which is cartesian and cocartesian is duoidal, we get that every CLDC is duoidal, and moreover the linear distributors are compatible with the interchange. 

\begin{proposition}\label{prop:CLDC_is_duoidal}
A CLDC is a duoidal category with the duoidal structure given as in \eqref{eq:duoidal-CLDC-1} and \eqref{eq:duoidal-CLDC-2}. Moreover, the linear distributors are compatible with the interchange in the sense that the following equalities hold: 
\begin{align}\label{diag:interchange_lineardist}
\begin{split}
&(1_{X} \times \mu_{A,B,C,D}); {\alpha_\times}^{-1}_{X, A+ C, B+ D} ; (\delta^L_{X,A,C}\times 1_{B+ D})\\
&= \delta^L_{X,A\times B, C\times D} ; ({\alpha_\times}^{-1}_{X, A, B}+ 1_{C\times D}); \mu_{X\times A, B, C, D} \\ \\
&(1_{A\times B}+ \delta^L_{C,D,X}); {\alpha_+}_{A\times B, C\times D, X}; (\mu_{A,B,C,D}+ 1_{X}) \\
&= \mu_{A,B,C,D+ X}; (1_{A+ C}\times {\alpha_+}_{B,D,X}); \delta^L_{A+ C, B+ D, X}\\ \\
&(\mu_{A,B,C,D} \times 1_{X}); {\alpha_\times}_{A+ C, B+ D, X}; (1_{A+ C} \times \delta^R_{B,D,X})\\
&= \delta^R_{A\times B, C\times D, X} ; (1_{A\times B} + {\alpha_\times}_{C, D, X}); \mu_{A, B, C, D\times X} \\ \\ 
&(\delta^R_{X,A,B} + 1_{C\times D}); {\alpha_+}^{-1}_{X, A\times B, C\times D}; (1_{X} + \mu_{A,B,C,D}) \\
&= \mu_{X+ A, B, C, D}; ({\alpha_+}^{-1}_{X, A, C}\times 1_{B+ D}); \delta^R_{X, A+ C, B+ D} \\
\end{split}
\end{align}
\end{proposition}
\begin{proof} As a commutative diagram, the first equality of \eqref{diag:interchange_lineardist} is: 
\begin{equation*}
\xymatrixrowsep{1.75pc}\xymatrixcolsep{4pc}\xymatrix{
X\times ((A\times B)+(C\times D))\ar[r]^-{\delta^L_{X, A\times B, C\times D}}\ar[d]_-{1_X\times \mu_{A,B,C,D}} & (X\times (A\times B))+(C\times D)\ar[d]^-{{\alpha_\times}^{-1}_{X,A,B}+1_{C\times D}} \\
X\times ((A+C)\times (B+D))\ar[d]_-{{\alpha_\times}^{-1}_{X,A+C, B+D}} & ((X\times A)\times B)+(C\times D)\ar[d]^-{\mu_{X\times A, B, C, D}} \\
(X\times (A+C))\times (B+D)\ar[r]_-{\delta^L_{X,A,C}\times 1_{B+D}} & ((X\times A)+C)\times (B+D)
}
\end{equation*}
We leave it as an exercise for the reader to draw out the other commutative diagrams. 

Now the first equality of \eqref{diag:interchange_lineardist} holds by the following computation using the cartesian tensor structure: 
\begin{align*}
    &\delta^L_{X, A\times B, C\times D}; ({\alpha_\times^{-1}}_{X,A,B}+1_{C\times D});\mu_{X\times A, B, C, D}  \\
    &= \langle \delta^L_{X, A\times B, C\times D}; ({\alpha_\times^{-1}}_{X,A,B}+1_{C\times D});(\pi^0_{X\times A,B} + \pi^0_{C,D}), \delta^L_{X, A\times B, C\times D};\\
    &\qquad\qquad({\alpha_\times^{-1}}_{X,A,B}+1_{C\times D}); (\pi^1_{X\times A,B} +\pi^1_{C,D})\rangle \\
    &= \langle \delta^L_{X, A\times B, C\times D}; ((1_X\times \pi^0_{A,B}) + \pi^0_{C,D}), \delta^L_{X, A\times B, C\times D}; (\pi^1_{X, A\times B} + 1_{C\times D}); (\pi^1_{A,B}+\pi^1_{C,D})\rangle \\
    &= \langle (1_X\times (\pi^0_{A,B} + \pi^0_{C,D})); \delta^L_{X,A,C}, \pi^1_{X,(A\times B)+(C\times D)}; (\pi^1_{A,B}+\pi^1_{C,D})\rangle \\
    &= \langle 1_X\times \mu_{A,B,C,D}; \pi^0_{A+C, B+D}); \delta^L_{X,A,C}, \pi^1_{X,(A\times B)+(C\times D)}; \mu_{A,B,C,D}; \pi^1_{A+C, B+D}\rangle \\
    &= \langle 1_X\times \mu_{A,B,C,D}; \pi^0_{A+C, B+D}); \delta^L_{X,A,C}, (1_X\times \mu_{A,B,C,D});\pi^1_{X, (A+C)\times (B+D)}; \pi^1_{A+C, B+D}\rangle \\
    &= (1_X\times \mu_{A,B,C,D}); \langle (1_X\times  \pi^0_{A+C, B+D}); \delta^L_{X,A,C}, \pi^1_{X, (A+C)\times (B+D)}; \pi^1_{A+C, B+D}\rangle \\
    &= (1_X\times \mu_{A,B,C,D}); \langle {\alpha_\times^{-1}}_{X, A+C, B+D}; \pi^0_{X\times (A+C), B+D};\delta^L_{X,A,C} , \\
    &\qquad\qquad {\alpha_\times^{-1}}_{X, A+C, B+D}; \pi^1_{X\times (A+C), B+D}\rangle \\
    &= (1_X\times \mu_{A,B,C,D}); {\alpha_\times^{-1}}_{X, A+C, B+D}; \langle \pi^0_{X\times (A+C), B+D};\delta^L_{X,A,C}, \pi^1_{X\times (A+C), B+D}\rangle \\
    &= (1_X\times \mu_{A,B,C,D}); {\alpha_\times^{-1}}_{X, A+C, B+D}; (\delta^L_{X,A,C} + 1_{B+D})
`\end{align*}
The second equality in \eqref{diag:interchange_lineardist} holds by the cocartesian par structure:
\begin{align*}
    &\mu_{A,B,C,D+X}; (1_{A+C} \times {\alpha_+}_{B,D,X}); \delta^L_{A+C, B+D, X} \\
    &= [(\iota^0_{A,C}\times \iota^0_{B, D+X}); (1_{A+C} \times {\alpha_+}_{B,D,X}); \delta^L_{A+C, B+D, X}, (\iota^1_{A,C}\times \iota^1_{B,D+X}); \\
    &\qquad\qquad (1_{A+C} \times {\alpha_+}_{B,D,X}); \delta^L_{A+C, B+D, X})] \\
    &=[(\iota^0_{A,C}\times \iota^0_{B,D}); (1_{A+C}\times \iota^0_{B+D,X}); \delta^L_{A+C, B+D, X}, (\iota^1_{A,C}\times (\iota^1_{B,D}+ 1_X));\delta^L_{A+C, B+D, X})] \\
    &= [(\iota^0_{A,C}\times \iota^0_{B,D}); \iota^0_{(A+C)\times (B+D), X}, \delta^L_{C,D,X}; ((\iota^1_{A,C}\times \iota^1_{B,D})+1_X)] \\
    &= [\iota^0_{A\times B, C\times D}; \mu_{A,B.C,D}; \iota^0_{(A+C)\times (B+D), X}, \delta^L_{C,D,X}; (\iota^1_{A\times B, C\times D}; \mu_{A,B,C,D} + 1_X)]\\
    &= [\iota^0_{A\times B, C\times D}; \iota^0_{(A\times B)+(C\times D), X}, \delta^L_{C,D,X}; (\iota^1_{A\times B, C\times D}+1_X)]; (\mu_{A,B,C,D}+1_X)\\
     &= [\iota^0_{A\times C, (C\times D)+X}, \delta^L_{C,D,X}; \iota^1_{A\times B, (C\times D)+X}];{\alpha_+}_{A\times B, C\times D, X}; (\mu_{A,B,C,D}+1_X) \\
     &= (1_{A\times B}+\delta^L_{C,D,X}); {\alpha_+}_{A\times B, C\times D, X}; (\mu_{A,B,C,D}+1_X)
\end{align*}
The last two equalities \eqref{diag:interchange_lineardist} hold via similar calculations using the right linear distributor instead. 
\end{proof}

\begin{remark}\label{rem:duoidal}
Given that duoidal structure arises canonically when finite (co)products exist, any LDC with a (co)cartesian (resp. par) tensor structure is a duoidal category. 
\end{remark}

Recall that duoidal categories were first introduced in \cite{Joyal_Street_1993} to generalize the canonical interchange isomorphism in a braided monoidal category. While the duoidal interchange map is not in general an isomorphism, in a CLDC, we can use the mix maps to see that it is essentially the canonical interchange isomorphism for the (co)product: 
\begin{align*}
\tau^\times _{W,X,Y,Z} = \langle \pi^0_{W,X}\times \pi^0_{Y,Z}, \pi^1_{W,X}\times \pi^1_{Y,Z}\rangle 
 \c (W\times X)\times (Y\times Z)\rarr (W\times Y)\times (X\times Z)\\ 
 \tau^+_{W,X,Y,Z} = [\iota^0_{W,Y} + \iota^0_{X, Z}, \iota^1_{W,Y} + \iota^1_{X, Z}]
 \c (W+X)+(Y+Z)\rarr (W+Y)+(X+Z)
\end{align*}

\begin{proposition}
In a CLDC, the following diagrams commute: 
\begin{equation}\label{diag:interchange_canonicalflip_mix}
\begin{gathered}\xymatrixrowsep{1.75pc}\xymatrixcolsep{3pc}\xymatrix{
(A\times B)+(C\times D)\ar[r]^-{\mu_{A,B,C,D}}\ar[d]_-{\mix_{A,B}+\mix_{C,D}} & (A+C)\times (B+D)\ar[d]^-{\mix_{A+C, B+D}}  \\
(A+B)+(C+D)\ar[r]_-{\tau^+_{A, B, C, D}} & (A+C)+(B+D) \\
(A\times B)\times (C\times D)\ar[r]^-{\tau^\times_{A,B,C,D}}\ar[d]_-{\mix_{A\times B, C\times D}} & (A\times C)\times (B\times D)\ar[d]^-{\mix_{A,C}\times \mix_{B\times D}} \\
(A\times B)+(C\times D)\ar[r]_-{\mu_{A,B,C,D}} & (A+ C)\times(B+ D) \\
}
\end{gathered}\end{equation}  
\end{proposition}
\begin{proof}
That the first diagram commutes follows from the naturality of the mix maps and the definitions of the interchanges:
\begin{align*}
    &\mu_{A,B,C,D}; \mix_{A+C, B+D} \\
    &= [(\iota^0_{A,C}\times \iota^0_{B,D});\mix_{A+C, B+D} , (\iota^1_{A,C}\times \iota^1_{B,D});\mix_{A+C, B+D}] \\
    &= [\mix_{A,B};(\iota^0_{A,C}+ \iota^0_{B,D}), \mix_{C,D}; (\iota^0_{A,C}+ \iota^0_{B,D})] \\
    &= [\mix_{A,B};\iota^0_{A+B, C+D}; \tau^+_{A,B,C,D}, \mix_{C,D}; \iota^1_{A+B, C+D}; \tau^+_{A,B,C,D}] \\
    &= [\mix_{A,B};\iota^0_{A+B, C+D}, \mix_{C,D}; \iota^1_{A+B, C+D}]; \tau^+_{A,B,C,D} \\
    &= (\mix_{A,B}+\mix_{C,D}); \tau^+_{A,B,C,D}
\end{align*}
Similarly, we can also show that the second diagram commutes. 
\end{proof}

We conclude this section with the observation that in any category with finite products and finite coproducts, the interchange map is compatible with the symmetries in the sense that the following diagrams commute: 
\begin{equation}\label{diag:interchange_braidings}
\begin{gathered}\xymatrixrowsep{1.75pc}\xymatrixcolsep{3pc}\xymatrix{
(A\times B)+(C\times D)\ar[r]^-{\mu_{A,B,C,D}}\ar[d]_-{{\sigma_\times}_{A,B}+{\sigma_\times}_{C,D}} & (A+C)\times (B+D)\ar[d]^-{{\sigma_\times}_{A+C, B+D}} \\
(B\times A)+(C\times D)\ar[r]_-{\mu_{B,A,D,C}} & (B+D)\times (A+C)\\
(A\times B)+(C\times D)\ar[r]^-{\mu_{A,B,C,D}}\ar[d]_-{{\sigma_+}_{A\times B, C\times D}} & (A+C)\times (B+D)\ar[d]^-{{\sigma_+}_{A,C}\times {\sigma_+}_{B,D}} \\
(C\times D)+(A\times B)\ar[r]_-{\mu_{C,D,A,B}}  & (C+A)\times (D+B)
}
\end{gathered}\end{equation}
which are easily shown using the (co)universal property of the (co)product and the (co)pairing definition of the interchange. 

With these compatibilities and Prop \ref{prop:CLDC_is_duoidal}, it follows that every CLDC is also a symmetric \emph{medial} linearly distributive category (MLDC). MLDCs were introduced by the first named author in order to develop a linearly distributive version of the Fox theorem. For more on MLDCs and how CLDCs can be characterized as special MLDCs, we invite the curious reader to see \cite{Kudzman-Blais_2025}.

\section{Posetal Distributive Categories}\label{sec:BDL}

In this section we discuss our first central class of CLDCs: posetal distributive categories. We also review how if a distributive category is a CLDC, then it must be posetal. Afterwards, we give an adjunction between CLDCs and posetal distributive categories using \emph{semizero} objects. 

\subsection{Posetal and Distributive Categories}

Let us first recall the definition of distributive categories. 

\begin{definition}\cite[Sec 3]{Cockett_1993}\label{def:distributive}
A {\bf distributive category} is a category \bD\ with finite products and finite coproducts such that the product distributes over the coproduct in the sense that the canonical natural transformations
\begin{equation}\label{eq:distributive-cat}
\begin{gathered}
d^L_{A,B,C} = [1_{A}\times \iota^0_{B,C}, 1_{A}\times \iota^1_{B,C}]\c (A\times B)+(A\times C)\rarr A\times(B+C) \\
d^R_{A, B, C} = [\iota^0_{A,B}\times 1_{C}, \iota^1_{A, B}\times 1_{C}]\c(A\times C)+(B\times C)\rarr(A+B)\times C
\end{gathered}
\end{equation}
are isomorphisms. 
\end{definition}
Note of course that if one of the above natural transformations is an isomorphism, then so is the other one via symmetry. 

Now, recall that a \textbf{posetal} category, otherwise known as a thin category, is a category such that for every $A$ and $B$ there is at most one map of type $A \to B$. If a map $A \to B$ exists, then we say that $A \leq B$. We abuse notation and use $A \leq B$ to also represent the map of type $A \to B$, and we write $A = B$ for $A \cong B$, that is, when $A \leq B$ and $B \leq A$.

Therefore, a posetal distributive category is a posetal category with terminal object $\top$, binary products $A\wedge B$ (denoted by a meet in the posetal context), an initial object $\bot$, and binary coproducts $A\vee B$ (denoted by a join in the posetal context), such that the product and coproduct distributive over one another, in the sense that the following equality holds for all objects $A$, $B$, and $C$: 
\begin{align}\label{eq:dist-lattice-eq}
A \wedge (B\vee C) = (A\wedge B)\vee (A\wedge C) && A \vee (B \wedge C) = (A \vee B) \wedge (A\vee C)
\end{align}
Note that a small posetal distributive category is a {\bf bounded distributive lattice} in the usual sense. 

\subsection{Posetal Collapse}\label{sec:posets}

Every posetal distributive category is a CLDC:

\begin{example}\label{prop:BDLC-CLDC} 
A posetal distributive category ${\mathcal L}$ is a CLDC, whose linear distributors are given by: 
\begin{align*}
& \delta^R_{A,B,C} \c (A\vee B)\wedge C = (A\wedge C) \vee (B\wedge C) \leq A\vee (B\wedge C) \\
&\delta^L_{A, B, C} \c  A\wedge (B\vee C) = (A\wedge B)\vee (A\wedge C) \leq (A\wedge B)\vee C \end{align*}
\end{example}

The linear distributor coherence conditions hold trivially as there is at most one map between two objects. In fact, posetal distributive categories correspond precisely to posetal CLDCs. 

\begin{lemma}\label{lem:posetal_CLDC}
A CLDC is posetal if and only if it is a posetal distributive category.
\end{lemma}
\begin{proof} We have already seen the $\Leftarrow$ direction in Ex \ref{prop:BDLC-CLDC}. For the $\Rightarrow$ direction, suppose that we have a CLDC which is also posetal. In particular, the linear distributors tell us that:
\[ (A\vee B)\wedge C \leq A\vee (B\wedge C) \qquad A\wedge (B\vee C) \leq (A\wedge B)\vee C\]
The only axiom of a posetal distributive category which is not immediate is distributivity \eqref{eq:dist-lattice-eq}. To show this, first note that $(A \wedge B) \vee (A \wedge C) \leq A \wedge (B \vee C)$ is true whenever we have products and coproducts. In fact, this is the canonical natural transformation from the definition of a distributive category. On the other hand, using the linear distributors we get that: 
\begin{align*}\label{eq:posetal_distributive_CLDC}
   A\wedge (B\vee C) &= (A\wedge A)\wedge (B \vee C) = A\wedge (A\wedge (B\vee C)) \leq A \wedge ((A\wedge B)\vee C) \\
   &= ((A\wedge B)\vee C) \wedge A \leq (A\wedge B)\vee (C\wedge A) = (A\wedge B)\vee (A\wedge C) 
\end{align*}
Therefore it follows that $ A\wedge (B\vee C) = (A\wedge B)\vee (A\wedge C)$. For the other distributive law, again first note that $A \vee (B \wedge C) \leq (A \vee B)\wedge (A \vee C)$ holds whenever we have products and coproducts. Using the linear distributor we also have that: 
\begin{align*}\begin{split}
    (A\vee B)\wedge (A\vee C) &= (A\vee B)\wedge (C\vee A) \leq ((A\vee B)\wedge C)\vee A\leq (A\vee (B\wedge C))\vee A  \\
    &= A\vee (A\vee (B\wedge C)) = (A\vee A)\vee (B\wedge C) = A\vee (B\wedge C)
\end{split}\end{align*}
So $A\vee (B\wedge C) =     (A\vee B)\wedge (A\vee C)$. So we conclude that a posetal CLDC is a posetal distributive category. 
\end{proof}

There are two well-known theorems describing how certain kinds of CLDCs collapse to being posetal: Joyal's Paradox and ``orthogonality of linear distributivity and standard distributivity''. We will take the time to detail these results here, as they are integral to our understanding of CLDCs. 

We begin with Joyal's Paradox, one of the most famous results in categorical logic. Joyal's Paradox states that a small cartesian closed category with involution is a Boolean algebra (otherwise known as a Boolean lattice), in other words is posetal. This means that the naive definition for categorical semantics of classical logic, generalizing the categorical semantics of intuitionistic logic to the classical case, just provides semantics of provability and not proofs. For an in-depth discussion of Joyal's Paradox in its standard form, see \cite[Appendix B]{Girard_1991}. We shall state the result here using the language of CLDCs. 

As before, let us consider a posetal distributive category, but this time add complements: for every object $A$, there is an object $A^\ast$ such that $A \vee A^\ast = \top$ and $A \wedge A^\ast = \bot$. Note that a small complemented posetal distributive category is precisely a complemented bounded distributive lattice, otherwise known as a {\bf Boolean lattice}.

\begin{theorem}[Joyal's Paradox]
A CLDC has negation if and only if it is a complemented posetal distributive category.   
\end{theorem}
\begin{proof}
For the $\Leftarrow$ direction, a complemented posetal distributive category is by definition a posetal distributive category, and hence a posetal CLDC by Lem \ref{lem:posetal_CLDC}. It is immediate that it has negation by setting $A^\perp = A^\ast$. 

For $\Rightarrow$ direction, consider a  CLDC \bX\ with negation. For every object $A$ we have a complementation pair $(A, A^\perp, \gamma^R_A, \tau^L_A)$. Now for every other object $B$ we have that $\bX(A,B)\cong \bX(A\times \top, B) \cong \bX(\top, A^\perp + B)$. By Prop \ref{prop:bot_subterminal_top_preinitial}, $\top$ is preinitial, meaning that there is at most one map of type $\top \to A^\perp + B$. As such, there is at most one map of type $A \to B$, or in other words, $A$ is preinitial. As this is true for all objects $A$ in \bX, there is at most one morphism between two objects, thus \bX\ is a posetal. Hence by Lem \ref{lem:posetal_CLDC}, \bX\ is a posetal distributive category. Setting $A^\ast = A^\perp$, since we have a complementation pair, we have that  $\top \leq A \vee A^\ast $ and $A \wedge A^\ast \leq \bot$. Now, $A \vee A^\ast \leq \top$ and $\bot\leq A \wedge A^\ast$ as $\top$ is terminal and $\bot$ is initial. Thus, $A \vee A^\ast = \top$ and $A \wedge A^\ast = \bot$. Therefore, \bX\ is a complemented posetal distributive category, as desired. 
\end{proof}
Of course, if we restrict our attention to small CLDCs, we get that: a small CLDC has negation if and only if it is a Boolean lattice. 

We now turn our attention to the other collapse, which involves distributive categories. As previously mentioned, it was initially thought that all distributive categories were CLDCs. However, it later became clear that this is only true if the distributive category is posetal. This now very well-known result was given by Cockett and Seely in the corrected version of ``Weakly Distributive Categories'' \cite{Cockett_Seely_1997_LDC}. We give an alternative proof of this result using a characterization of preinitial objects in a distributive category and our result that the terminal object in a CLDC is preinitial. This gives a nice simple alternative proof of the collapse in question.

\begin{proposition}\cite[Prop 3.3]{Cockett_1993}\label{prop:preintial_distributive}
For an object $A$ in a distributive category, the following are equivalent: 
\begin{enumerate}[label=(\roman*)]
    \item $A$ is preinitial,
    \item There is an object $B$ such that $\iota^0_{A,B}\c A\rarr A+B$ is an isomorphism,
    \item $t_A; \iota^0_{\top,\top} = t_A; \iota^1_{\top,\top} \c A\rarr \top+\top$
\end{enumerate}
\end{proposition}

\begin{proposition}\cite[Prop 3.1]{Cockett_Seely_1997_LDC}\label{prop:CLDC_distributive_poset}
A CLDC is a distributive category if and only if it is posetal. 
\end{proposition}
\begin{proof} The $\Leftarrow$ direction is immediate from Lem \ref{lem:posetal_CLDC}: a posetal CLDC is a posetal distributive category, which is of course a distributive category. For the $\Rightarrow$ direction, suppose that a CLDC is also a distributive category. Then for all objects $A\in \bX$, by Lem \ref{lem:preinitial/subterminal} and Prop \ref{prop:bot_subterminal_top_preinitial}, we get that $t_A; \iota^0_{\top,\top} = t_A; \iota^1_{\top,\top}$ as $\iota^0_{\top,\top} = \iota^1_{\top,\top}$  Therefore, by Prop \ref{prop:preintial_distributive}, every object $A$ is preinitial, implying \bX\ is posetal. 
\end{proof}

In fact, we do not require the inverses of all canonical maps $d^L_{A,B,C}$ to prove the CLDC is posetal. This is due to the fact that the proof of $(iii)\Rarr (i)$ of Prop \ref{prop:preintial_distributive} only requires distributivity over $\top+\top$, as remarked by Cockett in \cite{Cockett_1993}. In other words:

\begin{lemma} Let $A$ be an object in a cartesian and cocartesian category. If the map
\[d^L_{A,\top,\top} = [1_A\times \iota^0_{\top,\top}, 1_A\times \iota^1_{\top,\top}]\]
has a right inverse and $t_A; \iota^0_{\top,\top} = t_A; \iota^1_{\top,\top}$, then $A$ is preinitial. 
\end{lemma}
\begin{proof}
Consider the following commuting diagram:
\begin{equation*}
\begin{gathered}\xymatrixrowsep{1.75pc}\xymatrixcolsep{2pc}\xymatrix{
A\times A\ar[rrrr]^-{\pi^0_{A,A}}\ar[dddd]_-{\pi^0_{A,A}}\ar[rrd]^-{1_A\times t_A}\ar[rdd]_-{1_A\times t_A} & & & & A\ar[dddd]^-{\iota^0_{A,A}} \\
& & A\times \top\ar[rru]^-{\pi^0_{A,\top}}\ar[d]_-{1_A\times \iota^0_{\top,\top}}\ar[rd]^-{\iota^0_{A\times \top, A\times\top}} \\
& A\times \top\ar[ldd]_-{\pi^0_{A,\top}}\ar[r]^-{1_A\times \iota^1_{\top,\top}}\ar[rd]_-{\iota^1_{A\times\top, A\times\top}} & A\times (\top+\top)\ar[rd]^-{{d^L_{A,\top,\top}}^{-1}} & (A\times\top)+(A\times \top)\ar@{=}[d]\ar[l]_-{d^L_{A,\top,\top}} \\
& & (A\times\top)+(A\times \top)\ar@{=}[r]\ar[u]_-{d^L_{A,\top,\top}} & (A\times\top)+(A\times \top)\ar[rd]^-{\pi^0_{A,\top}+\pi^0_{A,\top}}\\
A\ar[rrrr]_-{\iota^1_{A,A}} &&&& A+A
}
\end{gathered}\end{equation*}  
This implies $\Delta_A; \pi^0_{A,A}; \iota^0_{A,A} = \Delta_A; \pi^0_{A,A}; \iota^1_{A,A}$ and thus $\iota^0_{A, A} = \iota^1_{A,A}$. $A$ is then preinitial by Lem \ref{lem:preinitial/subterminal}.
\end{proof}

Then, using the above lemma instead of Prop \ref{prop:preintial_distributive}, we get:

\begin{lemma}\label{lem:CLDC_posetal_distributive_top}
A CLDC is posetal if and only if the following canonical natural transformation:
\[ d^L_{A,\top,\top} = [1_A\times \iota^0_{\top,\top}, 1_A\times \iota^1_{\top,\top}]\c (A\times \top)+(A\times \top)\rarr A\times (\top+\top) \]
has a right inverse.
\end{lemma}

An important remark is that in any CLDC, there is a potential reasonable candidate ${d^L}^\flat: A\times (B+C) \to (A\times B)+(A\times C)$ for an inverse to the canonical natural transformation $d^L_{A,B,C}: (A\times B)+(A\times C) \to A \times (B+C)$, which is defined to be the following composite:
\begin{equation}\label{eq:candidate_distributive_inverse}
\begin{gathered}\xymatrixrowsep{1.75pc}\xymatrixcolsep{2.5pc}\xymatrix{
A\times (B+C)\ar[dd]_-{{d^L}_{A,B,C}^{\flat}} \ar[r]^-{\Delta_A \times 1_{B+C}} & (A\times A)\times (B+C) \ar[r]^-{{\alpha_\times}_{A,A,B+C}} & A\times (A\times (B+C)) \ar[d]^-{1_A\times \delta^L_{A,B,C}} \\
&& A\times ((A\times B)+C)\ar[d]^-{{\sigma_\times}_{A,(A\times B)+C}} \\
(A\times B)+(A\times C) & (A\times B)+(C\times A)\ar[l]^-{1_{A\times B}+{\sigma_\times}_{C,A}}  & ((A\times B)+C)\times A\ar[l]^-{\delta^R_{A\times B, C, A}}
}
\end{gathered}\end{equation}  
By Prop \ref{prop:CLDC_distributive_poset}, we see that this candidate is, in fact, the inverse only when the CLDC is posetal. Cockett and Seely further prove an additional characterization of when said candidate is the inverse: when the initial object is {\bf strict}, which means that all maps to it are isomorphisms (and dually terminal object is {\bf costrict} if all maps from it are isomorphisms). 

\begin{lemma}\cite[Lem 3.2]{Cockett_Seely_1997_LDC} In a CLDC, if the following diagram commutes:  
\begin{equation}\label{eq:strict_initial_lemma}
\begin{gathered}\xymatrixrowsep{1.75pc}\xymatrixcolsep{1.75pc}\xymatrix{
A\times B\ar[r]^-{1_A\times \iota^0_{B,C}}\ar[d]_-{\pi^1_{A,B}} & A\times (B+C)\ar[d]^-{{\sigma_\times}_{A, B+C}}\\
B\ar[rdd]_-{\iota^0_{B,A\times C}} & (B+C)\times A\ar[d]^-{\delta^R_{B,C,A}} \\
& B+(C\times A)\ar[d]^-{1_B+{\sigma_\times}_{C,A}} \\
& B+(A\times C)
}
\end{gathered}\end{equation} 
then the canonical natural transformation $d^L_{A,B,C}$ has a right inverse. 
\end{lemma}

\begin{remark} \cite[Lem 3.2]{Cockett_Seely_1997_LDC} includes an extra diagram in its assumptions, but this diagram is simply the second equation of \eqref{eqn:linear_dist_projection_injection}, which holds in any CLDC. Moreover, the lemma states that a CLDC is distributive if and only if the diagrams commute, but only prove that the candidate morphism \eqref{eq:candidate_distributive_inverse} is a right inverse. Nevertheless, the right inverse is all that is needed to prove the collapse. 
\end{remark}

We can now state:
\begin{proposition}\cite[Thm 3.3]{Cockett_Seely_1997_LDC}
A CLDC is posetal if and only if the initial object is strict.
\end{proposition}
\begin{proof}
For the $\Rightarrow$ direction, a CLDC being posetal implies that the initial object is strict. For the $\Leftarrow$ direction, we can show that if we have a strict initial object, then \eqref{eq:strict_initial_lemma} holds, as detailed in the proof of \cite[Thm 3.3]{Cockett_Seely_1997_LDC}. This means that $d^L_{A,B,C}$ has a right inverse by the above lemma, and so in particular $d^L_{A,\top,\top}$ has a right inverse. Thus by Lem \ref{lem:CLDC_posetal_distributive_top}, we conclude that the CLDC is posetal.
\end{proof}

If we consider the opposite CLDC (the opposite category with the roles of tensor and par exchanged), we get a dual statement about costrict terminal objects. Compiling all these characterizations together gives the following collapse theorem:
\begin{theorem}
For a CLDC \bX, the following are equivalent:
\begin{enumerate}[label=(\roman*)]
    \item \bX\ is a posetal distributive category;
    \item \bX\ is a posetal category;
    \item \bX\ is a distributive category;
    \item \bX\ has a strict initial object;
    \item \bX\ has a costrict terminal object.
\end{enumerate}
\end{theorem}

\subsection{Semizero Objects}

We conclude this section by showing that from any CLDC we can construct a posetal distributive category. We do so by considering the subcategory of {\bf semizero} objects. In an arbitrary category, a semizero object is an object that is both preinitial and subterminal (in the same way that a zero object is an object that is both terminal and initial). Of course, recall that in a CLDC, by Prop \ref{prop:CLDC_preintial_subterminal}, being subterminal is equivalent to being preinitial. So we can immediately extend Prop \ref{prop:CLDC_preintial_subterminal} to include semizero objects.

\begin{lemma} 
In a CLDC, an object $A$ is a semizero object if and only if $A$ is preinitial (or equivalently subterminal). 
\end{lemma}

Now for a category $\bX$, we let $\SZ[\bX]$ be the full subcategory of semizero objects of $\bX$. By the above lemma, for a CLDC this subcategory is equivalent to taking the subcategory of preinitial objects (or equivalently subterminal objects). 

\begin{proposition}
For a CLDC \bX, $\SZ[\bX]$ is a posetal CLDC, so in particular it is a posetal distributive category.
\end{proposition}
\begin{proof} 
Since there is at most one map between semizero objects, for any category $\bX$, we have that $\SZ[\bX]$ is always a posetal category. So suppose that $\bX$ is also a CLDC. To show that $\SZ[\bX]$ is a CLDC, it suffices to show that the (co)product of semizero objects is again semizero, and that both the terminal object and initial object are semizero objects. However, recall from Prop \ref{prop:bot_subterminal_top_preinitial} that $\top$ is preinitial and $\bot$ is subterminal, and therefore it follows that indeed $\top$ and $\bot$ are semizero objects. Next it follows from Lem \ref{lem:preinitial/subterminal}, that in a (co)cartesian category, the (co)product of (resp. preinitial) subterminal objects is again (resp. preinitial) subterminal. Thus from this fact and Prop \ref{prop:CLDC_preintial_subterminal}, it follows that in a CLDC, if $A$ and $B$ are semizero objects, then $A + B$ and $A \times B$ are also semizero objects. Therefore, from here we can conclude that $\SZ[\bX]$ is a sub-CLDC of $\bX$, and hence by Lem \ref{lem:posetal_CLDC}, $\SZ[\bX]$ is a posetal distributive category. 
\end{proof}

We now show that taking the subcategory of semizero objects of a CLDC provides the right adjoint to the inclusion functor of posetal distributive categories into CLDCs. To state this properly, we must first address morphisms between posetal distributive categories. These will correspond to the categorical version of lattice homomorphisms \cite[Sec I.1.]{Johnstone_1982}. So for posetal distributive categories ${\mathcal B}$ and ${\mathcal L}$, a \textbf{posetal distributive functor} is a functor $F\c {\mathcal B} \to {\mathcal L}$ which strictly preserves products, coproducts, the terminal object, and the initial object:
\begin{align*}
 F(A\wedge B) = F(A)\wedge F(B) && F(\top)=\top && F(A\vee B) = F(A)\vee F(B) && F(\bot) = \bot
\end{align*}
It is straightforward to see that when seeing a posetal distributive category as a posetal CLDC, a posetal distributive functor is the same thing as Frobenius cartesian linear functor. So let $\mathbf{PDC}$ denote the category of posetal distributive categories and posetal distributive functors, which can equivalently be described as the category of posetal CLDC and Frobenius cartesian linear functors. 

Now to build our functor from $\CLDC$ to $\mathbf{PDC}$, we first show that, unsurprisingly, cartesian linear functor preserve semizero objects, which then implies that Frobenius cartesian linear functors restrict to the subcategory of semizero objects. 

\begin{lemma}
Cartesian linear functors preserve semizero objects.
\end{lemma}
\begin{proof} It follows from Lem \ref{lem:preinitial/subterminal} that if a functor $F$ preserves (co)products, and $A$ is a (resp. preinitial) subterminal object, then $F(A)$ is also (resp. preinitial) subterminal. Since a cartesian linear functor $F=(F_\times, F_+)\c \bX\rarr \bY$ consists of a functor $F_\times$ and $F_+$ which preserves products and coproduct respectively, if $A$ is a semizero object in $\bX$, then $F_\times(A)$ and $F_+(A)$ are both semizero objects in CLDC \bY. 
\end{proof}

\begin{corollary}
A Frobenius cartesian linear functor $F\c \bX \to \bY$ restricts to a posetal distributive functor $F\c \SZ[\bX] \rarr  \SZ[\bY]$. 
\end{corollary}

As such, we can now define the functor $\SZ[-]\c \CLDC\rarr \mathbf{PDC}$ which maps a CLDC to its subcategory of semizero objects, and a Frobenius cartesian linear functor to its restriction on the subcategory of semizero objects. On the other hand, let $U\c\mathbf{PDC}\rarr \CLDC$ be inclusion functor, or in other words, the forgetful functor. 

\begin{theorem}
$\SZ[-]\c \CLDC\rarr \mathbf{PDC}$ is a right adjoint of $U\c\mathbf{PDC}\rarr \CLDC$. \end{theorem}
\begin{proof} Let $\bX$ be a CLDC and ${\mathcal L}$ a posetal distributive category. Now note that every object in ${\mathcal L}$ is semizero, as such $\SZ[{\mathcal L}] = {\mathcal L}$. Then given a Frobenius cartesian linear functor $F\c {\mathcal L}\rarr \bX$, this clearly gives a posetal distributive functor $F\c \SZ[{\mathcal L}]\rarr \SZ[\bX]$. On the other hand, given a posetal distributive functor $G \c {\mathcal L}\rarr\SZ[\bX]$, this extends to a Frobenius cartesian linear functor ${\mathcal L} \to \bX$ by post-composing $G$ with the inclusion functor $\SZ[\bX] \hookrightarrow \bX$. It is clear that this gives a natural bijection $\mathbf{PDC}({\mathcal L}, \SZ[\bX]) \cong \CLDC(U({\mathcal L}), \bX)$. 
\end{proof}

\section{Semi-Additive Categories}\label{sec:S-Add}

In this section, we discuss our second main class of CLDC examples: \emph{semi-additive} categories, which turn out to be the compact CLDCs. Recall that a semi-additive category is essentially a category with finite \emph{biproducts}. The main result of this section is a collapse result which says that if a CLDC is isomix if and only if it is a semi-additive category. We also provide a construction of a semi-additive category from a CLDC via the (co)slice category over the initial (resp. terminal) object. 

\subsection{Semi-Additive Categories and Biproducts}

If only to the setup notation and terminology, it may be useful to quickly review semi-additive categories. First recall that in an arbitrary category, a {\bf zero object} is an object $\bzero$, which is both a terminal and initial object. A useful observation is that having a zero object is equivalent to having a terminal object and an initial object that are isomorphic. 

\begin{lemma}
A category has a zero object if and only if it has an initial object $\bot$ and a terminal object $\top$, and the unique map $b_\top = t_\bot \c \bot\rarr\top$ is an isomorphism.
\end{lemma}

If a category has a zero object, then it has \textbf{zero morphisms}, which recall are the unique maps between every pair of objects $X$ and $Y$ which factor through the zero object: $0_{X,Y}\c X\xrightarrow{t_X} \bzero \xrightarrow{b_Y} Y$. These zero morphisms are absorbing in the sense that for all maps $f\c X \to Y$, $f; 0_{Y,Y^\prime} = 0_{X,Y^\prime}$ and $ 0_{X^\prime, X}; f = 0_{X^\prime, Y}$. Now if a category with a zero object also has finite products and finite coproducts, then there is a canonical natural transformation $\psi_{X,Y}\c X + Y \to X \times Y$ defined as follows:
\begin{align}\label{eq:psi} \psi_{X,Y} = [\langle 1_X, 0_{X,Y}\rangle, \langle 0_{Y,X}, 1_Y\rangle] = \langle [1_X, 0_{Y,X}], [0_{X,Y}, 1_Y]\rangle \c X+Y\rarr X\times Y
\end{align}

\begin{definition}\cite{Lack_2012}\label{def:semi-additive}
A {\bf semi-additive} category is a category with a zero object $\bzero$, with finite products and finite coproducts, such that $\psi_{X,Y}\c X + Y \to X \times Y$ is a natural isomorphism. 
\end{definition}

As such, we will call $\psi$ the \textbf{semi-additive comparison}. Then a \textbf{semi-additive category} is a category where its semi-additive comparison is an isomorphism. It is worth mentioning that any isomorphisms between binary coproducts and binary products is sufficient for a category to be semi-additive. 

\begin{proposition}\cite[Thm 5]{Lack_2012}\label{prop:semi_add}
If a category has finite products and finite coproducts and a natural isomorphism $A+B \to A\times B$, then it is semi-additive. 
\end{proposition}

Now in a semi-additive category, since $A + B \cong A \times B$, this implies that $+$ is also a product and $\times$ is also a coproduct, and, in fact, both are a biproduct. Of course, one can define biproducts directly. Recall that in a category with a zero object, a \textbf{biproduct} of $A$ and $B$ is an object $A\op B$ equipped with four maps:  
\[ \pi^0_{A,B}\c A\op B \rarr A \qquad \pi^1_{A,B}\c A\op B\rarr B \qquad \iota^0_{A,B}\c A\rarr A\op B \qquad \iota^1_{A,B}\c B\rarr A\op B\]
such that  $A\op B$ equipped with $\pi^0_{A,B}$ and $\pi^1_{A,B}$ is a product,  $A\op B$ equipped with $\iota^0_{A,B}$ and $\iota^1_{A,B}$ is a coproduct, and the following diagrams commute: 
\begin{equation}
\begin{gathered}\xymatrixcolsep{1.5pc}\xymatrix{
A\ar[r]^-{\iota^0_{A,B}}\ar@{=}[rd] & A\op B\ar[d]^-{\pi^0_{A,B}} & B\ar[r]^-{\iota^1_{A,B}}\ar@{=}[rd] & A\op B\ar[d]^-{\pi^1_{A,B}} & A\ar[r]^-{\iota^0_{A,B}}\ar[rd]_-{0_{A,B}} & A\op B\ar[d]^-{\pi^1_{A,B}} & B\ar[r]^-{\iota^1_{A,B}}\ar[rd]_-{0_{B,A}}& A\op B\ar[d]^-{\pi^0_{A,B}}\\
& A & & B & & B & & A
}
\end{gathered}\end{equation}
See \cite{Karvonen_2020} for more details, where the above definition is known as pointed biproducts. Then a bicartesian category is a category with finite biproducts. Every bicartesian category $\bX$ is a symmetric monoidal category $(\bX, \oplus, \bzero)$. Equivalently, a symmetric monoidal category is bicartesian if its monoidal structure is both cartesian and cocartesian. Of course, every bicartesian category is a semi-additive category where $\top = \bzero = \bot$ and $\times = \oplus = +$, and $\psi_{A,B} = 1_{A \oplus B}$. Conversely, every semi-additive category is a bicartesian category by setting $\bzero = \top$ or $\bzero = \bot$, and $\oplus = \times$ or $\oplus = +$. 

We can also ask what should morphisms of semi-additive categories be. These should be functors that preserves both products and coproducts, but are also compatible with the semi-additive comparison map. 

\begin{definition}\label{def:semi-additive_functor}
A functor $F\c \bX\rarr\bY$ between semi-additive categories is {\bf semi-additive} if 
\begin{enumerate}[label=(\roman*)]
    \item $(F, m_\bzero,m_\times)\c (\bX,\times, \bzero)\rarr (\bY,\times,\bzero)$ is a monoidal functor; 
    \item $(F, n_\bzero, n_+)\c (\bX,+,\bzero)\rarr (\bY,+,\bzero)$ is a monoidal functor
\end{enumerate}
and such that the one of the following equivalent diagrams commutes. 
\begin{equation}\begin{gathered}\label{cc:semi-add_functor}
\xymatrix{
F(A+B)\ar[r]^-{{n_+}_{A,B}}\ar[d]_-{F(\psi_{A,B})} & F(A)+F(B)\ar[d]^-{\psi_{F(A), F(B)}} && F(A)\times F(B)\ar[r]^-{{m_\times}_{A,B}}\ar[d]_-{\psi^{-1}_{F(A), F(B)}} & F(A\times B)\ar[d]^-{F(\psi^{-1}_{A,B})} \\
F(A\times B) & F(A)\times F(B)\ar[l]^-{{m_\times}_{A,B}}&& F(A)+F(B) & F(A+B)\ar[l]^-{{n_+}_{A,B}}
} 
\end{gathered}\end{equation}
\end{definition}

We denote the category of semi-additive categories and semi-additive functors by $\mathbf{SAdd}$. 

\subsection{Semi-Additive Collapse}

Every semi-additive category is a CLDC:

\begin{example}\label{ex:semi-additive-CLDC}
A semi-additive category, with natural isomorphism $\psi_{A, B}\c A+B\cong A\times B$, is a CLDC whose linear distributors are defined as follows: 
\begin{align*}
    \delta^R_{A, B, C} &= \xymatrix{(A+B)\times C \ar[r]^-{\psi_{A,B}\times 1_C} & (A\times B)\times C \ar[r]^-{{\alpha_\times}_{A,B,C}} & A\times (B\times C) \ar[r]^-{\psi^{-1}_{A,B\times C}} & A\times (B+C)}\\
    &= \xymatrix@L=0.5pc{(A+B)\times C \ar[r]^-{\psi^{-1}_{A+B, C}} & (A+B)+C \ar[r]^-{{\alpha_+}^{-1}_{A,B,C}} & A+(B+C) \ar[r]^-{1_A+\psi_{B,C}} & A+(B\times C)} \\
    \delta^L_{A, B, C} &=  \xymatrix@L=0.5pc{A\times(B+ C) \ar[r]^-{1_A\times \psi_{B,C}} & A\times (B\times C) \ar[r]^-{{\alpha_\times}_{A,B,C}^{-1}} & (A\times B)\times C \ar[r]^-{\psi^{-1}_{A\times B, C}} &(A\times B)+C} \\
    &= \xymatrix@L=0.5pc{A\times(B+ C) \ar[r]^-{\psi^{-1}_{A,B+C}}  & A+ (B+ C) \ar[r]^-{{\alpha_+}_{A,B,C}}& (A+ B)+ C \ar[r]^-{\psi_{A,B}+1_C} & (A\times B)+C}
\end{align*}
It is straightforward (though very tedious) to check that the linear distributor axioms hold. These essentially amount to the cartesian monoidal structure and cocartesian monoidal structure being associative, unital, and symmetric, and also isomorphic to each other. Moreover, a semi-additive category is a compact CLDC, where the nullary mix map is the zero morphism $m \c= 0_{\bot,\top} \c \bot \to \top$, and the induced mix map $\mix_{A,B} \c A \times B \to A+B$ is $\mix_{A,B} = \psi^{-1}_{A,B}$.
\end{example}

This implies that every bicartesian category is a degenerate CLDC. 

\begin{example} 
A bicartesian category is a degenerate CLDC, so in particular its linear distributors are given by the associators of the biproduct, $A \oplus (B \oplus C) \cong (A \oplus B) \oplus C$. 
\end{example}

In fact, semi-additive (resp. bicartesian) categories correspond precisely to the compact (resp. degenerate) CLDCs. 

\begin{lemma}\label{lem:compact_CLDC}
A CLDC is compact if and only if it is a semi-additive category.
\end{lemma}
\begin{proof} The $\Leftarrow$ direction is Ex \ref{ex:semi-additive-CLDC}. For the $\Rightarrow$ direction, suppose that $\bX$ is a compact CLDC, so $\bot \cong \top$ and $A \times B \cong A + B$. However by Prop \ref{prop:semi_add}, since the mix maps give a natural isomorphism from the coproduct to the product, this implies that $\bX$ is a semi-additive category. It is easy to check that the linear distributors must be of the form as given in Ex \ref{ex:semi-additive-CLDC} and that $\psi_{A,B} = \mix^{-1}_{A,B}$.
\end{proof}

\begin{corollary} A CLDC is degenerate if and only if it is a bicartesian category.
\end{corollary}

Moreover, it is straightforward to see that from the point of view of CLDCs, semi-additive functors correspond precisely to Frobenius cartesian linear functors.

\begin{lemma}
The notion of semi-additive functors is equivalent to Frobenius cartesian linear functors between semi-additive categories (seen as a compact CLDC). 
\end{lemma} 
\begin{proof} It is clear that a Frobenius cartesian linear functor is a semi-additive functor. Indeed, the second diagram of \eqref{cc:semi-add_functor} is diagram \eqref{cc:mix_Frobenius_functor} in this context, which holds as cartesian Frobenius linear functors are mix, as stated in Lem \ref{lem:mix_cartesian_Frobenius_linear_functor}. 

Conversely, consider a semi-additive functor $F\c \bB\rarr \bB'$ between semi-additive categories. To prove that $F$ is a Frobenius cartesian linear functor, it remains only to show that \eqref{cc:Frobenius_linear_functor} holds. The first condition holds by commutativity of the following diagram: 
\begin{equation*}
\begin{gathered}\xymatrixrowsep{2.75pc}\xymatrixcolsep{3pc}\xymatrix{
F(A)\times F(B+C)\ar[r]^-{{m_\times}_{A, B+C}}\ar[d]_-{1_{F(A)}\times {n_+}_{B,C}}\ar[rd]^-{1_{F(A)}\times F(\psi_{B,C})} & F(A\times (B+C))\ar[r]^-{F(1_A\times \psi_{B,C})}\ar@{}[d]|{({\rm nat})} & F(A\times (B\times C))\ar[d]^-{F({\alpha_\times^{-1}}_{A,B,C})}\\
F(A)\times (F(B)+F(C))\ar[d]_-{1_{F(A)}\times \psi_{F(B), F(C)}}\ar@{}[r]|{\eqref{cc:semi-add_functor}} & F(A)\times F(B\times C)\ar[ru]^-{{m_\times}_{A,B\times C}}\ar@{}[d]|{({\rm mon})} & F((A\times B)\times C)\ar[d]^-{F({\psi^{-1}_{A\times B, C}})} \\
F(A)\times (F(B)\times F(C))\ar[d]_-{{\alpha_\times^{-1}}_{F(A), F(B), F(C)}}\ar[ru]^-{1_{F(A)}\times {m_\times}_{B,C}} & F(A\times B)\times F(C)\ar[ru]^-{{m_\times}_{A\times B, C}}\ar[rd]^-{\psi^{-1}_{F(A\times B), F(C)}}\ar@{}[d]|{({\rm nat})} \ar@{}[r]|{\eqref{cc:semi-add_functor}} & F((A\times B)+C)\ar[d]^-{{n_+}_{A\times B, C}}\\
(F(A)\times F(B))\times F(C)\ar[r]_-{\psi^{-1}_{F(A)\times F(B), F(C)}}\ar[ru]^-{{m_\times}_{A,B}\times 1_{F(C)}} & (F(A)\times F(B))+ F(C)\ar[r]_-{{m_\times}_{A,B}+1_{F(C)}} & F(A\times B)+F(C)
}
\end{gathered}\end{equation*}
The second condition follows similarly. Thus $F\c \bB\rarr \bB'$ is a Frobenius cartesian linear functor. 
\end{proof}

As such, this gives us a forgetful functor $U\c \mathbf{SAdd}\rarr\CLDC$.

The objective for the remainder of this section is to show that a CLDC is isomix if and only if it is semi-additive. Equivalently, this amounts to saying that a CLDC has a zero object if and only if it its products are isomorphic to its coproducts. In particular, from a LDC point of view, this says that in a CLDC, if its tensor unit and par unit are isomorphic, then tensor and par are also isomorphic. This is quite a strong collapse since, in general, a LDC can be isomix $\top \cong \bot$ without the mix map $A \parr B \to A \otimes B$ being an isomorphism. 

To show this, we first show that if the linear distributors are isomorphisms, then a CLDC is a semi-additive category (and hence compact). 

\begin{proposition}\label{prop:CLDC_invertible_lindist}
A CLDC has invertible linear distributors if and only if it is a semi-additive category. 
\end{proposition}
\begin{proof} The $\Leftarrow$ direction is immediate since the linear distributors as defined in Ex \ref{ex:semi-additive-CLDC} are the composites of isomorphisms. For the $\Rightarrow$ direction, let $\bX$ be a CLDC whose linear distributors are isomorphisms. By Prop \ref{prop:shift_tensor}, this implies that $\bot$ has $\times$-inverse. However, note that in a cartesian category, an object has a $\times$-inverse if and only if it is terminal, that is, isomorphic to $\top$. Indeed, suppose $A$ had a $\times$-inverse $A^{-1}$, which means that $t_{A \times A^{-1}}: A \times A^{-1} \to \top$ is an isomorphism. Then consider the composite $t^{-1}_{A \times A^{-1}}; \pi^0_{A,A^{-1}}: \top \to A$. By the universal property of the terminal object, we of course have that $t^{-1}_{A \times A^{-1}}; \pi^0_{A,A^{-1}}; t_A = t_\top = 1_\top$. On the other hand, note that by the universal property of the terminal object once again, we get that $t_A = \langle 1_A, t^{-1}_{A \times A^{-1}}; \pi^1_{A,A^{-1}}  \rangle; t_{A \times A^{-1}}$. So we can compute that: 
\begin{gather*}
 t_A; t^{-1}_{A \times A^{-1}}; \pi^0_{A,A^{-1}} = \langle 1_A, t^{-1}_{A \times A^{-1}}; \pi^1_{A,A^{-1}}  \rangle; t_{A \times A^{-1}};t^{-1}_{A \times A^{-1}}; \pi^0_{A,A^{-1}} \\
 =  \langle 1_A, t^{-1}_{A \times A^{-1}}; \pi^1_{A,A^{-1}}  \rangle; \pi^0_{A,A^{-1}} = 1_A
\end{gather*}
\[ \]
So we conclude that $t_A: A \to \top$ is an isomorphism. Therefore, since $\bot$ has a $\times$-inverse, it is terminal and so $\bot \cong \top$. As \bX\ is isomix with invertible linear distributors, \bX\ is a compact CLDC and, by Lem \ref{lem:compact_CLDC}, \bX\ is a semi-additive category. 
\end{proof}

We can now show that a CLDC being isomix implies that it is compact. 

\begin{proposition}\label{thm:CLDC_isomix}
A CLDC is isomix if and only if it is a semi-additive category.    
\end{proposition}
\begin{proof} The $\Leftarrow$ direction is immediate as a semi-additive category has a zero object. For the $\Rightarrow$ direction, suppose that $\bX$ is an isomix CLDC, so the nullary mix map $m\c \bot\rarr\top$ is invertible. We shall show that the left linear distributor $\delta^L$ is invertible. Define the natural transformation $\partial^R_{A,B,C} \c (A\times B)+C\rarr A\times (B+C)$ as follows: 
\begin{align*}
    \partial^R_{A,B,C} = \langle[\pi^0_{A,B}, t_C;m^{-1};b_A], \pi^1_{A,B}+ 1_C\rangle =  [1_A\times \iota^0_{B, C}, \langle t_C;m^{-1};b_A, \iota^1_{B,C}\rangle]
\end{align*}
First observe that the following diagram commutes: 
\begin{equation*}
\begin{gathered}\xymatrixrowsep{1.75pc}\xymatrixcolsep{2.75pc}\xymatrix{
A\times (B+C)\ar[dddd]_-{1_A \times t_{B+C}}\ar[rrr]^-{\delta^L_{A,B,C}}\ar[dr]_-{1_A\times (t_B+t_C)} &\ar@{}[rd]|{({\rm nat})}&& (A\times B)+C\ar[ld]^-{(1_A\times t_B)+t_C}\ar[dddd]^-{[\pi^0_{A,B}, t_C;m^{-1};b_A]} \\
& A\times (\top+\top)\ar[r]^-{\delta^L_{A,\top,\top}}\ar@{}[rd]|{({\rm nat})}\ar[d]_-{1_A\times (1_\top+m^{-1})} & (A\times\top)+\top \ar[d]^-{1_{A\times\top} + m^{-1}}& \\
\ar@{}[rd]|{({\rm term})}& A\times (\top+\bot)\ar[r]^-{\delta^L_{A,\top,\bot}} \ar[rd]_-{1_A\times {u^R_+}_\top}^-{\eqref{cc:unit_lineardist}}& (A\times\top)+\bot\ar[d]^-{{u^R_+}_{A\times\bot}}\ar@{}[r]|{(*)}& \\
& & A\times \top\ar[rd]^-{{u^R_\times}^{-1}_A}  & \\
A\times \top\ar[rrr]_-{{u^R_\times}^{-1}_A}\ar@{=}[rru] &&& A
}
\end{gathered}\end{equation*}
where $(*)$ commutes by
\begin{align*}
    &((1_A\times t_B)+t_C); (1_{A\times \top}+ m^{-1}); {u^R_+}_{A\times\top}; {u^R_\times}^{-1}_A \\
    &= [(1_A\times t_B);\iota^0_{A\times\top, \bot}, t_C; m^{-1};\iota^1_{A\times\top, \bot}]; [1_{A\times\top}, b_{A\times \top}]; {u^R_\times}^{-1}_A \\
    &= [(1_A\times t_B);{u^R_\times}^{-1}_A  , t_C; m^{-1}; b_{A\times\top}; {u^R_\times}^{-1}_A ] = [\pi^0_{A,B}, t_C;m^{-1};b_A]
\end{align*}
As such, we can compute that: 
\begin{align*}
    &\delta^L_{A,B,C}; \partial^R_{A,B,C} = \delta^L_{A,B,C}; \langle [\pi^0_{A,B}, t_C;m^{-1};b_A], \pi^1_{A,B}+ 1_C\rangle \\
    & = \langle \delta^L_{A,B,C};[\pi^0_{A,B}, t_C;m^{-1}; b_A], \delta^L_{A,B,C};(\pi^1_{A,B}+ 1_C)\rangle = \langle \pi^0_{A,B+C}, \pi^1_{A,B+C}\rangle  = 1_{A\times(B+C)}
\end{align*}
where the equality in the first component follows from the above diagram, and the equality in the second component is by Lem \ref{lem:delta_pi_iota}. Similarly, we can also compute that: 
\begin{align*}
    &\partial^R_{A,B,C}; \delta^L_{A,B,C} = [1_A\times \iota^0_{B,C}, \langle t_C;m^{-1};b_A, \iota^1_{B,C}\rangle]; \delta^L_{A,B,C}  \\
    & = [(1_A\times\iota^0_{B,C});\delta^L_{A,B,C} , \langle t_C;m^{-1};b_A, \iota^1_{B,C}\rangle; \delta^L_{A,B,C} ] = [\iota^0_{A\times B, C}, \iota^1_{A\times B,C}] = 1_{A\times (B+C)}
\end{align*}
Thus $\delta^L$ is an isomorphism with inverse $\partial^R$. Consequently, via the symmetries, the right linear distributor $\delta^R$ is invertible as well. Then, as \bX\ is isomix with invertible linear distributors, \bX\ is compact and, by Lem \ref{lem:compact_CLDC}, \bX\ is a semi-additive category. Note that the inverse of the mix map $\mix_{A,B}: A \times B \to A +B$ is:
\[ \mix_{A,B}^{-1} = \langle [1_A, t_B; m^{-1}; b_A], [t_A; m^{-1}; b_B, 1_B]\rangle = [ \langle 1_A, t_A; m^{-1}; b_B \rangle, \langle t_B; m^{-1}; b_A, 1_B\rangle ]  \]
which is precisely $\psi_{A,B}$. 
\end{proof}

Combining these characterizations together, we record that: 

\begin{theorem}
For a CLDC \bX, the following are equivalent:
\begin{enumerate}[label=(\roman*)]
    \item \bX\ is a semi-additive category;
    \item \bX\ is compact;
    \item \bX\ has invertible linear distributors;
    \item \bX\ is isomix.
\end{enumerate}
\end{theorem}

It is worth mentioning that some of the constructions used in the previous proof connect back to duoidal categories. Indeed, a \emph{normal} duiodal category, that is, a duoidal category $(\bX, \diamond, I, \star, J)$ such that $\iota\c I\rarr J$ is an isomorphism, is an isomix LDC \cite{Spivak_Srinivasan_2024}, whose linear distributors are defined as follows: 
\begin{align*}
  &\resizebox{\hsize}{!}{$\partial^L_{A,B,C}= A\diamond (B\star C) \cong (A\star J)\diamond (B\star C) \xrightarrow{\mu_{A,J,B,C}} (A\diamond B)\star (J\diamond C) \cong (A\diamond B)\star (I \diamond C) \cong (A\diamond B)\star C$} \\
  &\resizebox{\hsize}{!}{$\partial^R_{A,B,C} = (A\star B)\diamond C \cong (A\star B)\diamond (J\star C) \xrightarrow{\mu_{A,B,J,C}} (A\diamond J) \star (B\diamond C) \cong (A\diamond I)\star (B\diamond C) \cong A\star (B\diamond C)$}
\end{align*}
Now, for any isomix CLDC $\bX$ (so a semi-additive category), $(\bX, +,\bot,\times,\top)$ is normal duoidal. The induced linear distributors from this normal duoidal category are precisely the inverses of $\delta^R$ and $\delta^L$ constructed in the proof of Proposition \ref{thm:CLDC_isomix}. As such this makes $\bX$ also an isomix ``cocartesian'' LDC\footnote{Note that the opposite category of a CLDC is not a cocartesian LDC, and vice-versa.}, where by this we mean a LDC where the coproduct is the tensor and product is par. 

Finally, it is interesting to notice that our two main classes of CLDCs discussed so far are ``orthogonal'' to one another. 

\begin{proposition}
A CLDC is a semi-additive category and a posetal distributive category if and only if it is trivial. 
\end{proposition}
\begin{proof}
The $\Leftarrow$ direction is trivially true. On the other hand, the $\Rightarrow$ direction follows from the fact that for a posetal distributive category its initial object $\bot$ is strict, while for a semi-additive category its initial object $\bot$ is a zero objet. Therefore, for any object $A$, the unique zero map $0_{A,\bot}: A \to \bot$ must be an isomorphism. So every object is a zero object, and so the CLDC is trivial.  
\end{proof}

Viewing CLDCs as models for some logic with two connectives, two constants and the cut rule, the posetal distributive categories correspond to all proofs between two formulas being identified, while the semi-additive categories correspond to the connectives being identified. 

\subsection{(Co)Slice Construction}

In this section, we show that for a CLDC, the (co)slice category over the initial (resp. terminal) object is an isomix CLDC, and so in particular a semi-additive category. In fact, this (co)slice category is equivalent to a certain subcategory. As such, there are at least two full subcategories of a CLDC which are semi-additive categories. However, these subcategories will not necessarily be sub-CLDCs of the original CLDC, since the subcategories may not have the same initial or terminal object as the CLDC. 

While we will focus on the slice category over the initial object, the same dual arguments of course hold for the coslice category over the terminal object. So recall that for a category $\bX$ and object $X$, the {\bf slice category} $\bX/X$ is the category whose objects are maps $f\c A\rarr X$ with codomain $X$, where a map $g\c (f\c A\rarr X)\rarr (f^\prime\c A^\prime\rarr X)$ is a map $g\c A\rarr A^\prime$ in \bX\ such that $g; f^\prime =f$. 

Now suppose that $\bX$ is a CLDC and consider the slice category over its initial object $\bX/\bot$. Recall that by Prop \ref{prop:bot_subterminal_top_preinitial}, $\bot$ is subterminal. Therefore, for every object $A$, there is at most one map into $\bot$. So having such a map is a property of an object $A$ rather than structure. Moreover, if $A$ and $A^\prime$ both have maps into $\bot$, then any map $g: A \to A^\prime$ will also preserve said maps. Thus $\bX/\bot$ is isomorphic to the full subcategory of objects for which there exists (a necessarily) unique map into $\bot$. Abusing notation, we will simply associate $\bX/\bot$ to this subcategory. Therefore, from now on, for a CLDC $\bX$, $\bX/\bot$ will be the full subcategory of object $A$ of $\bX$ for which there exists (a necessarily) unique map into $\bot$. 

\begin{proposition}
Let \bX\ be a CLDC. Then $\bX/\bot$ is a compact CLDC, and therefore a semi-additive category.   
\end{proposition}
\begin{proof} We will in fact explain why $\bX/\bot$ is semi-additive. First we will explain why it has the same products, coproducts, and initial object as $\bX$. We only need to show this is true for objects, since being a full subcategory will immediately imply their respective (co)universal property. Starting with the latter, obviously $\bot \in \bX/\bot$. Now if $A, A^\prime \in \bX/\bot$ with respective maps $a\c A\rarr\bot$ and $a^\prime \c A^\prime\rarr\bot$, then so are $A \times A^\prime$ and $A + A^\prime$ via maps $(a \times a^\prime); \pi^0_{\bot,\bot}: A \times A^\prime \to \bot \times \bot \to \bot$ and $[a,a^\prime]: A + A^\prime \to \bot$ respectively. 

Now note that it is easy to see that $\bot$ is in fact a zero object in $\bX/\bot$ (in fact this, along with the coproducts, follow from basic facts about slice categories). Moreover, if $A, A^\prime \in \bX/\bot$, then the zero morphism from $A$ to $A^\prime$ in this category is the composite $a; b_{A^\prime}$. As such since $\bX/\bot$ has a zero object, products, and coproducts, we get our canonical map $\psi_{A,A^\prime} \c A + A^\prime \to A \times A^\prime$, which is worked out to be:
\begin{align*}
\psi_{A,A^\prime} = \langle [1_A, a';b_{A}], [a;b_{A^\prime}, 1_{A^\prime}]\rangle 
\end{align*}
We can also show that this in fact an isomorphism, whose inverse is precisely the mix map, $\psi^{-1}_{A,B} = \mix_{A,B}$. To show this,  we compute the following: 
\begin{align*}
& \iota^0_{A,A^\prime}; \psi_{A,A^\prime}; \psi^{-1}_{A,A^\prime} =~ \langle 1_A, a; b_{A^\prime} \rangle; \mix_{A,A^\prime} \\
&=~ \Delta_A; (1_A\times a;b_{A^\prime}); (1_A\times {u^L_+}^{-1}_{A^\prime}); \delta^L_{A,\bot,A^\prime}; ((1_A\times m)+1_{A^\prime}); ({u^R_\times}^{-1}_A+1_{A^\prime}) \\
&=~ \Delta_A; (1_A\times a); (1_A\times \iota^0_{\bot, {A^\prime}}); \delta^L_{A,\bot, {A^\prime}}; ((1_A\times m)+1_{A^\prime}); ({u^R_\times}^{-1}_A+1_{A^\prime})  \\
&=~ \Delta_A; (1_A\times a); \iota^0_{A\times \bot, {A^\prime}}; ((1_A\times m)+1_{A^\prime}); ({u^R_\times}^{-1}_A+1_{A^\prime}) \tag{by Lem \ref{lem:delta_pi_iota}}  \\
&=~ \Delta_A; (1_A\times a); (1_A\times m); {u^R_\times}^{-1}_A; \iota^0_{A,{A^\prime}}\\
&=~ \Delta_A; (1_A\times t_A); {u^R_\times}^{-1}_A; \iota^0_{A,{A^\prime}}\\
&=~ 1_A;  \iota^0_{A,{A^\prime}}\\
&=~ \iota^0_{A,{A^\prime}} 
\end{align*}
and similarly $ \iota^1_{A,A^\prime}; \psi_{A,A^\prime}; \psi^{-1}_{A,A^\prime} =  \iota^1_{A,A^\prime}$. Therefore, by the universal property of the coproduct, we get that $\psi_{a,b}; \psi^{-1}_{a,b}= 1_{A+B}$. Similarly, we can show $\psi^{-1}_{a,b}; \psi_{a,b} = 1_{A\times B}$ (instead using the universal property of the product). As such, $\bX/\bot$ is a semi-additive category. Then $\bX/\bot$ is a compact CLDC. Moreover, it not difficult to check that the induced linear distributors from the semi-additive structure end up being precisely the same linear distributors of $\bX$ (this is because the linear distributors were used to define the mix map). 
\end{proof}

Now for any CLDC $\bX$, there is another equivalent way to describe objects in $\bX/\bot$. Indeed, recall from Lem \ref{lem:A_iso_A_times_bot} that saying that $A \in \bX/\bot$ is equivalent to saying $A \cong A \times \bot$. Moreover, for every $B \in \bX$, consider $B\times\bot$, then $(B \times \bot) \times \bot \cong B \times \bot$, meaning $B\times \bot\in \bX/\bot$. Therefore, every object in $\bX/\bot$ is isomorphic to one of the form $- \times \bot$. So let $\bX\times\bot$ be the full sub-category of objects of the form $A\times\bot$ for some $A$ in \bX, then:

\begin{lemma} Let $\bX$ be a CLDC. Then $\bX/\bot \simeq \bX\times\bot$, and therefore $\bX\times\bot$ is a semi-additive category. 
\end{lemma}
\begin{proof} Define the functor $F\c \bX/\bot\rarr\bX\times\bot$ on objects as $F(A) = A \times \bot$ and on maps as $F(f) = f \times 1_\bot$. It is clear that $F$ is full, faithful, and essentially surjective. Therefore, $F$ induces an equivalence $\bX/\bot \simeq \bX\times\bot$ as desired. Explicitly, the functor $G\c \bX\times\bot\rarr \bX/\bot$ is defined on objects as $G(A \times \bot) = A \times \bot$ and $G(f) = f$. The natural isomorphism $\alpha_A: A \to GF(A)$ is defined as $\alpha_{A} = \langle 1_A, a\rangle$, while the natural isomorphism $\beta_{A \times \bot}: FG(A \times \bot) \to A \times \bot$ is defined as $\beta_{A \times \bot} = \pi^0_{A \times \bot,\bot}$. 
\end{proof}

So from a CLDC, we obtain a subcategory that is semi-additive. Naturally, one can ask what happens if we apply the slice construction to a semi-additive category, which we now know is a compact CLDC. Well in this case, note that the initial object is a zero object and hence a terminal object. It is well-known that taking the slice category over the terminal is isomorphic to the original base category. Therefore, we have that:

\begin{corollary}
Let $\bB$ be a semi-additive category. Then $\bB$ is isomorphic to $\bB/\bot$. 
\end{corollary}

Let us now explain why Frobenius cartesian linear functors preserve this construction. So let $F \c \bX \to \bY$ be a Frobenius cartesian linear functor. Recall in particular this implies that we have an isomorphism $n_\bot: F(\bot) \to \bot$. Therefore, for any $A \in  \bX/\bot$, with associated map $a: A \to \bot$, we get the composite $F(a); n_\bot: F(A) \to F(\bot) \to \bot$, which implies that $F(A) \in \bY/\bot$. As such, this allows us to restrict $F$ to this subcategory, so we get a functor $F/\bot\c \bX/\bot \rarr \bY/\bot$ which is simply defined as $F/\bot(-) = F(-)$. We now explain how this is in fact a semi-additive functor. 

\begin{proposition}\label{prop:FCLF-SAF}
Let $F \c \bX \to \bY$ be a Frobenius cartesian linear functor. Then $F/\bot\c \bX/\bot \rarr \bY/\bot$ is a semi-additive functor. 
\end{proposition}
\begin{proof} By the explanation given above, $F/\bot$ is well-defined. Moreover, since $F$ preserves (up to isomorphism) products, coproducts, and the initial object, $F/\bot$ will as well. Since the initial object in $\bX/\bot$ and $\bY/\bot$ is also the terminal object, it follows that $F/\bot$ preserves the terminal object as well. It remains now to show that $F/\bot$ satisfies \eqref{cc:semi-add_functor}. This is equivalent to checking one of the following equalities: 
\begin{align*}
& {n_+}_{A,B};\langle [1_{F(A)}, F(b);n_\bot;b_{F(A)}], [F(a);n_\bot;b_{F(B)}, 1_{F(B)}]\rangle; {m_\times}_{A,B}\\
&= F(\langle [1_A, a;b_A], [f;b_{B}, 1_{B}]\rangle) \\ \\
& {m_\times}_{A,B}; F(\mix_{A,B}); {n_+}_{A,B} = \mix_{F(A), F(B)}
\end{align*}
The former is equivalent to 
\begin{align*}
& \langle [1_{F(A)}, F(b);n_\bot;b_{F(A)}], [F(a);n_\bot;b_{F(B)}, 1_{F(B)}]\rangle \\
& = {n^{-1}_+}_{A,B};F(\langle [1_A, b;b_A], [a;b_{B}, 1_{B}]\rangle); {m_\times}^{-1}_{A,B}
\end{align*}
which holds since 
\begin{align*}
&\iota^i_{F(A), F(B)}; {n^{-1}_+}_{B};F(\langle [1_A, b;b_A], [a;b_{B}, 1_{B}]\rangle); {m_\times}^{-1}_{A,B}; \pi^j_{F(A), F(B)} \\
& = F(\iota^i_{A,B}; \langle [1_A, b;b_A], [a;b_{B}, 1_{B}]\rangle; \pi^j_{A, B})
\end{align*}
and $F(c); b_B = F(c); n_\bot; b_{F(B)}$ for any $c\c C\rarr\bot$.

\end{proof}

As such, we obtain a functor $[-]/\bot\c \CLDC\rarr \mathbf{SAdd}$ which maps a CLDC to its slice category over its initial object, and restricts a Frobenius cartesian linear functor to this semi-additive subcategory. It is important to stress that $[-]/\bot\c \CLDC\rarr \mathbf{SAdd}$ does not provide either a right or left adjoint to the forgetful functor $U\c\mathbf{SAdd}\rarr \CLDC$. The problem is that neither the functor $F\c \bX\rarr\bX/\bot$, mapping $F(A) = A \times \bot$, or $G\c \bX/\bot\rarr \bX$, mapping $G(A) = A$, is a Frobenius cartesian linear functor. In the first case, $F\c \bX\rarr\bX/\bot$ is not $+$-monoidal as that would imply that $F(A + B)= (A + B) \times \bot$ is isomorphic to $F(A) + F(B) \cong (A \times \bot) + (B \times \bot)$, which is not necessarily the case in an arbitrary CLDC. In the second case, $G\c \bX/\bot\rarr \bX$ is not $\times$-monoidal as that would imply that $G(\bot)=\bot$ is isomorphic to $\top$, which again is not necessarily true in an arbitrary CLDC. Instead, if we were to restrict our attention to $\times$-monoidal functors or $+$-comonoidal functors between CLDCs, then this slice construction would provide a left or right adjoint respectively. 


There is likely still a construction which provides an adjunction to the inclusion functor  $U\c\mathbf{SAdd}\rarr \CLDC$, although it is not equivalent to a sub-category. A possibility would be to consider the localization of a CLDC with respect to isomorphisms and its mix maps. Although this construction would be difficult to describe, so we leave it for future work. 

Of course, we can alternatively redo this section by taking the coslice category under the terminal object instead. Indeed, in this case we are restricting to the full subcategory of objects with a (necessarily unique) map from the terminal object to them. As such, the terminal object becomes initial.

\begin{proposition}
Let \bX\ be a CLDC. Then $\top/\bX$, the coslice category of \bX\ under the terminal object $\top$, is a compact CLDC, so in particular it is a semi-additive category. Moreover, every Frobenius cartesian linear functor $F\c \bX\rarr\bY$ induces a semi-additive functor $\top/F\c \top/\bX\rarr\top/\bY$ defined as $\top/F(-) = F(-)$. Furthermore, there is an equivalence of categories $\top/\bX \simeq \bX+\top$, where the $\bX+\top$ is the full sub-category of objects of the form $A + \top$ for some $A$ in \bX. So in particular, $\bX + \top$ is a semi-additive category. 
\end{proposition}

\begin{corollary}
Let $\bB$ be a semi-additive category. Then $\bB$ is isomorphic to $\top/\bB$. 
\end{corollary}

\section{Revisiting Kleisli Category of Exception Monad}\label{sec:Kleisli}

As previously discussed, it was initially thought that the notions of CLDCs and distributive categories would coincide. It was then demonstrated not to be the case. However, in an effort to relate distributive categories to CLDCs, Cockett and Seely looked towards the \emph{exception} monad (sometimes known as the \emph{maybe} monad) to provide a possible source of examples of CLDCs. 

The {\bf exception monad} on a distributive category \bD\ is the monad $(\cdot+\top, \eta, \mu)$ whose unit is $\eta_{A} = \iota^0_{A, \top}\c A\rarr A+\top$ and whose multiplication is $\mu_{A} = [1_{A+\top}, \iota^1_{A, \top}]\c (A+\top)+\top\rarr A+\top$. The exception monad is named as such because, from a computer science perspective, the monad models a programming exception, in other words a computational failure handled in a controlled way. Given computations from one data type $A$ to another $B$, represented by maps $A\rarr B$, the effect of possibly throwing an exception is handled by applying the exception monad to $B$ and considering $B+\top$, where the terminal object $\top$ represents the controlled failure.

While a distributive category may not be (C)LDC, it was hoped that the \emph{Kleisli} category of the exception monad would be a (C)LDC. For clarity, when discussing the Kleisli category of the exception monad $\bD_{\cdot+\top}$, we shall use the interpretation brackets $\llbracket - \rrbracket$. So map $f \in \bD_{\cdot+\top}(A,B)$ will be denoted by its underlying map $\llbracket f \rrbracket\c A\rarr B+\top$. For example, the identity map in the Kleisli category is $\llbracket 1_A \rrbracket = \eta_A$, and composition is given as follows $\llbracket f;g \rrbracket = \llbracket f \rrbracket ; \left( \llbracket g \rrbracket + 1_\top \right) ; \mu_C$. Now it is well-known that $\bD_{\cdot+\top}$ has both products and coproducts. That it has coproducts follows from basic theory about Kleisli categories. As such, the initial object in $\bD_{\cdot+\top}$ is still $\bot$, and:
\begin{align}\llbracket b_A\rrbracket = b_{A+\top}\c \bot \rarr A+\top
\end{align}
The coproduct on objects is the same in \bD, so $A+B$, while the injections are given by: 
\begin{equation}\begin{gathered} \llbracket \iota^0_{A,B}\rrbracket = \iota^0_{A, B};\iota^0_{A+B, \top}\c A\rarr (A+B)+\top \\ 
\llbracket \iota^1_{A,B}\rrbracket = \iota^1_{A, B};\iota^0_{A+B, \top}\c B\rarr (A+B)+\top
\end{gathered}
\end{equation}
The copairing of Kleisli maps is the same as in $\bD$, that is: 
\begin{align}\llbracket [f,g] \rrbracket = \left[ \llbracket f \rrbracket, \llbracket g \rrbracket \right]\c A+B\rarr C+\top
\end{align}
The Kleisli category $\bD_{\cdot+\top}$ also has finite products, although they are rather unique. The terminal object in $\bD_{\cdot+\top}$ is $\bot$, the initial object of \bD, and:
\begin{align} \llbracket t_A \rrbracket= t_{A}; \iota^1_{\bot,\top}\c A\rarr \bot+\top \end{align}
The product in $\bD_{\cdot+\top}$ is given by $A\,\&\,B = (A+B)+(A\times B)$, with projections defined as follows:
\begin{equation}\begin{gathered}  \llbracket \pi^0_{A,B} \rrbracket =  [[\iota^0_{A,\top}, t_B; \iota^1_{A,\top}], \pi^0_{A,B}; \iota^0_{A,\top}]\c (A+B)+(A\times B)\rarr A+\top \\
\llbracket \pi^1_{A,B} \rrbracket = [[t_A; \iota^1_{B,\top}, \iota^0_{B,\top}], \pi^1_{A,B}; \iota^0_{B,\top}]\c (A+B)+(A\times B)\rarr B+\top 
\end{gathered}
\end{equation}
The pairing of Kleisli maps is defined as follows: 	
\begin{align} \llbracket \langle f, g\rangle \rrbracket = C \xrightarrow{\left \langle \llbracket f \rrbracket, \llbracket g \rrbracket \right \rangle} (A+\top)\times(B+\top) \cong ((A+B)+(A\times B))+\top \end{align}

Now in \cite{Cockett_Seely_1997_LDC}, it is claimed that $\bD_{\cdot+\top}$ is a CLDC. However, this cannot be the case. Indeed, note that $\bot$ is both the initial object and the terminal object in $\bD_{\cdot+\top}$, and thus $\bot$ is a zero object in $\bD_{\cdot+\top}$. If $\bD_{\cdot+\top}$ was a CLDC, then it would be isomix. As such, by Proposition \ref{thm:CLDC_isomix}, this would imply that $\bD_{\cdot+\top}$ was semi-additive. Now this is an issue since clearly the coproduct $A+B$ and the product $A \& B$ in $\bD_{\cdot+\top}$ are definitely not isomorphic, unless $\bD$ is trivial (i.e. every object in $\bD$ is a terminal/initial/zero object). So we have that: 

\begin{lemma}  
The Kleisli category of the exception monad of a distributive category $\bD$ is a CLDC if and only if $\bD$ is trivial. 
\end{lemma}

That said, Cockett and Seely were indeed correct on some level: the Kleisli category of the exception monad is an isomix SLDC, but with another tensor product. In fact, this other tensor product is the same on objects as the product, but is different on maps! First, note that coproduct of Kleisli maps is given as follows: 
\begin{align}
\llbracket f+g \rrbracket  =\left( \llbracket f \rrbracket + \llbracket g \rrbracket \right); \tau^+_{A^\prime, \top, B', \top}; \left( 1_{A^\prime+B'}+t_{\top+\top} \right) 
\end{align} 
While $\times$ is not a product in $\bD_{\cdot+\top}$, it is still a monoidal product (in fact a \emph{restriction} product \cite{Cockett_Lemay_2024}), which is defined as follows on Kleisli maps: 
\begin{equation}\begin{gathered}
\llbracket f\times g \rrbracket = (\llbracket f\rrbracket \times \llbracket g\rrbracket);{d^R}^{-1}_{A^\prime,\top,B'+\top}; {(d^L}^{-1}_{A^\prime,B',\top}+1_{\top\times(B'+\top)}); \\
{\alpha_+^{-1}}_{A^\prime\times B', A^\prime\times\top,\top\times(B'+\top)}; (1_{A^\prime\times B'} + t_{(A^\prime\times \top)+(\top\times (B'+\top))})
\end{gathered}
\end{equation}
At first glance, one might hope that the product of Kleisli maps $\llbracket f\,\&\,g\rrbracket$ should be equal to $ \llbracket (f+g)+(f\times g)\rrbracket$. However, this is not the case unfortunately (see \cite[Ex 3.11]{Cockett_Lemay_2024} for a detailed discussion of this fact). That said, $ \llbracket (f+g)+(f\times g)\rrbracket$ still gives us a monoidal product. 

So define a new monoidal product $\cv$ on $\bD_{\cdot+\top}$ on objects as $A \cv B = (A+B)+(A\times B)$, and on Kleisli maps as follows:
\begin{align}
\llbracket f \cv g \rrbracket = \llbracket (f+g)+(f\times g)\rrbracket 
\end{align}
Moreover, $\bot$ is also a monoidal unit for this monoidal product. We will explain in Appendix \ref{app:proof_Kleisli_LDC} why this indeed a monoidal structure on $\bD_{\cdot+\top}$. Now with this monoidal product playing the role of the tensor, we get that $\bD_{\cdot+\top}$ is a LDC with the linear distributors described by Cockett and Seely in \cite{Cockett_Seely_1997_LDC}. 

\begin{proposition}\label{thm:Kleisli_LDC}
For a distributive category \bD, the Kleisli category of its exception monad $\bD_{\cdot+\top}$ is an isomix  SLDC, with tensor structure $(\bD_{\cdot+\top}, \cv, \bot)$ and par structure $(\bD_{\cdot+\top}, +, \bot)$, where the left linear distributor $  \llbracket \delta^L_{A,B,C}\rrbracket \c A\cv (B+C)\rarr ((A\cv B)+C)+\top$, which is explicitly of type:
\begin{align*}
(A+(B+C))+(A\times(B+C)) \rarr (((A+B)+(A\times B))+C)+\top\
\end{align*}
is defined as the unique map induced by the following maps:
\[ \xymatrix@L=0.5pc{A \ar[r]^-{\iota^0_{A, B}} & A+B \ar[r]^-{\iota^0_{A+B, A\times B}} & A\cv B \ar[r]^-{\iota^0_{A\cv B, C}} & (A\cv B)+C \ar[r]^-{\iota^0_{(A\cv B)+C, \top}} & ((A\cv B)+C)+\top} \]
\[ \xymatrix@L=0.5pc{B \ar[r]^-{\iota^1_{A, B}} & A+B \ar[r]^-{\iota^0_{A+B, A\times B}} & A \cv B \ar[r]^-{\iota^0_{A\cv B, C}} & (A\cv B)+C \ar[r]^-{\iota^0_{(A\cv B)+C, \top}} & ((A\cv B)+C)+\top}\]
\[\xymatrix@L=0.5pc{C \ar[r]^-{\iota^1_{A\cv B,C}} & (A\cv B)+C \ar[r]^-{\iota^0_{(A\cv B)+C, \top}} & ((A\cv B)+C)+\top}\]
\[ \xymatrixcolsep{1pc}\xymatrix@L=0.5pc{A\times (B+C) \ar[r]^-{{d^L}^{-1}_{A, B, C}} & (A\times B)+(A\times C) \ar[r]^-{ 1_{A\times B} + t_{ A\times C}} & (A\times B)+\top \ar[r]^-{\iota^1_{A+B, A\times B}+1_{\top}} & (A\cv B)+\top \ar[r]^-{\iota^0_{A\cv B, C} +1_{\top}} & ((A\cv B)+C)+\top}\]
\end{proposition}
\begin{proof} Since the linear distributor axioms were not proven in \cite{Cockett_Seely_1997_LDC}, we give a full detailed proof in Appendix \ref{app:proof_Kleisli_LDC}.  
\end{proof}

In particular, another way of describing Kleisli categories of exception monads is as \emph{classical distributive restriction categories} \cite{Cockett_Lemay_2024}. As such, we may also state that: 

\begin{corollary} A classical distributive restriction category is an isomix SLDC. 
\end{corollary}

The $\cv$ product for a classical distributive restriction category can be described in a more streamlined fashion in terms of combining the restriction product and the restriction coproduct. To help understand the subtleties, let us look at potentially the most well-known example of a Kleisli category of an exception monad, the category of sets and partial functions. 

\begin{example} Let ${\rm Set}$ be the category of sets and functions. Then ${\rm Set}$ is a distributive category, where recall that the product $\times$ is the cartesian product, the coproduct $\sqcup$ is disjoint union, the terminal object is a chosen single $\top = \lbrace \ast \rbrace$, and the initial object is the empty set $\emptyset$. Then ${\rm Set}_{\cdot\sqcup \top}$ is isomorphic to ${\rm Par}$, the category of sets and partial functions. Now let us compare $\&$ and $\cv$ in ${\rm Par}$. On objects, recall that $X \& Y = X \cv Y = X \sqcup Y \sqcup (X \times Y)$. Now for partial functions $f: X \to Z$ and $g: Y \to W$, we have $f \& g$ and $f \cv g$, both of which are of type $X \sqcup Y \sqcup (X \times Y) \to Z \sqcup W \sqcup (Z \times W)$. However on the one hand, for $u \in X \sqcup Y \sqcup (X \times Y)$, we have that: 
\[ \left( f \cv g \right) (u) = \begin{cases} f(u) & \text{if } u\in X  \text{ and } f(u) \downarrow \\
g(u) & \text{if } u\in Y \text{ and } g(u) \downarrow \\
(f(u_0), g(u_1)) & \text{if } u = (u_0, u_1) \in X \times Y \text{ and } f(u_0) \downarrow \text{ and } g(u_1) \downarrow \\
\uparrow & \text{o.w.}  \end{cases} \]
(where $\downarrow$ means defined and $\uparrow$ means undefined), while on the other hand, we have that: 
\[ \left(f \& g \right) (u) = \begin{cases} f(u) & \text{if } u\in X  \text{ and } f(u) \downarrow \\
g(u) & \text{if } u\in Y \text{ and } g(u) \downarrow \\
(f(u_0), g(u_1)) & \text{if } u = (u_0, u_1) \in X \times Y \text{ and } f(x) \downarrow \text{ and } g(x) \downarrow \\
f(u_0) & \text{if } u = (u_0, u_1) \in X \times Y \text{ and } f(u_0) \downarrow \text{ and } g(u_1) \uparrow \\
g(u_1) & \text{if } u = (u_0, u_1) \in X \times Y \text{ and } f(u_0) \uparrow \text{ and } g(u_1) \downarrow \\
\uparrow & \text{o.w. } \end{cases} \]
As such, we clearly see that $f \& g$ and $f \sqcup g \sqcup (f \times g)$ are indeed different. Now ${\rm Par}$ is an LDC with respect to $\cv$ and $\sqcup$, where the linear distributor $\delta^L_{X,Y,Z}\c X\cv (Y\sqcup Z) \rarr (X\cv Y)\sqcup Z$, which is explicitly of type $X\sqcup Y\sqcup Z\sqcup (X\times Y)\sqcup (X\times Z)\rarr X\sqcup Y\sqcup Z \sqcup (X\times Y)$, is the partial function defined as follows:  
\[ \delta^L_{X,Y,Z}(u) = \begin{cases} u & \text{if $u \in X$, or $u \in Y$, or $u \in Z$, or $u \in X \times Y$} \\
 \uparrow & \text{if $u \in X \times Z$} \end{cases}\]
 So the left linear distributor kills off the $X \times Z$ part, since it does not appear in the codomain. 
\end{example}

\section{Further Examples}\label{sec:Examples}

Given all the results of the previous sections, it is reasonable to question whether the only possible CLDCs are posetal distributive categories or semi-additive categories. We can quickly see that is not the case as CLDCs are closed under products, that is, given CLDCs \bX\ and \bY, their product category $\bX\times\bY$ is a CLDC \cite[Sec 3]{Cockett_Seely_1997_LDC}. Therefore, by taking the product of a posetal distributive category ${\mathcal L}$ with a semi-additive category $\bB$, we obtain a CLDC which is neither posetal nor compact. 

We will expand upon this idea, using bounded distributive lattices and semi-additive categories as building blocks to construct more examples of CLDCs. In particular, we will consider the Grothendieck construction of a contravariant functor $F\c \bB^{op}\rarr \mathsf{BDL}$ from a semi-additive category $\bB$ to the category of bounded distributive lattices $\mathsf{BDL}$\footnote{This example is due to valuable discussions between the authors and Richard Garner whose idea it was to look at fibrations for additional examples.}. Recall that the Grothendieck construction for $F$ is the category $\int F$ whose objects are pairs $(A, a)$, where $A\in\bB$ and $a\in F(A)$, and where a map $f\c (A,a)\rarr (B,b)$ is a map $f\c A\rarr B$ in \bB\ such that $a\leq F(f)(b)$. Composition is inherited directly from $\bB$, that is, given $f\c (A,a)\rarr (B,b)$ and $g\c (B,b)\rarr (C,c)$ in $\int F$, their composition is simply $f;g \c (A,a)\rarr (C,c)$. This is well-defined since $a \leq F(f)(b)\leq F(f)(F(g)(c)) = F(f;g)(c)$, which follows from the fact that $F(f)\c F(B)\rarr F(A)$ is a bounded distributive lattice homomorphism.

Now, as we will see below, while $\int F$ obtains products, a terminal object, and an initial object, $\int F$ does not inherit coproducts automatically. As such our functor $F\c \bB^{op}\rarr \mathsf{BDL}$ must satisfy an additional condition: given $A, B\in\bB$, $a\in F(A)$, $b\in F(B)$ and $c\in F(A+B)$,
\begin{align}
\begin{split}\label{eqn:adjoint_condition}
& a\leq  F(\iota^0_{A,B})(c) \Rarr F(\psi_{A,B};\pi^0_{A,B})(a) \leq c \\
& b\leq  F(\iota^1_{A,B})(c)\Rarr F(\psi_{A,B};\pi^1_{A,B})(b) \leq c 
\end{split}
\end{align}
Observe that the converse of (\ref{eqn:adjoint_condition}) is always true: since $\iota^0_{A,B}; \psi_{A,B}; \pi^0_{A,B} = 1_A$, we compute that
\begin{align*}
&F(\psi_{A,B};\pi^0_{A,B})(a) \leq c \Rarr F(\iota^0_{A,B})(F(\psi_{A,B};\pi^0_{A,B})(a)) \leq F(\iota^0_{A,B})(c) \Rarr\\
&F(\iota^0_{A,B}; \psi_{A,B}; \pi^j_{A,B})(a) \leq  F(\iota^0_{A,B})(c)\Rarr F(1_A)(a) \leq F(\iota^0_{A,B})(c) \Rarr a\leq F(\iota^0_{A,B})(c)
\end{align*}
In particular, this means that the image of the projections and injections are adjoints of each other, $F(\psi_{A,B}; \pi^j_{A,B})\dashv F(\iota^j_{A,B})$, which recall means that: 
\begin{align}
\begin{split}\label{eqn:adjoint_condition2}
& F(\psi_{A,B};\pi^0_{A,B})(a) \leq c  \Leftrightarrow a\leq  F(\iota^0_{A,B})(c) \\
& F(\psi_{A,B};\pi^1_{A,B})(b) \leq c \Leftrightarrow  b\leq  F(\iota^1_{A,B})(c)
\end{split}
\end{align}
This condition guarantees the category $\int F$ inherits coproducts. 

\begin{theorem} $\int F$ is a CLDC where: 
\begin{enumerate}[label=(\roman*)]
    \item Terminal object: $(\bzero, \top_{F(\bzero)})$, where $\top_{F(\bzero)}$ is the top element of lattice $F(\bzero)$; 
    \item Products: $(A,a)\times (B,b) \c = \left(A\times B, F(\pi^0_{A,B})(a)\wedge F(\pi^1_{A,B})(b) \right)$; 
    \item Initial object: $(\bzero, \bot_{F(\bzero)})$, where $\bot_{F(\bzero)}$ is the bottom element of lattice $F(\bzero)$; 
    \item Coproducts: $(A,a)+ (B,b) \c = \left(A+ B, F(\psi_{A,B}; \pi^0_{A,B})(a)\vee F(\psi_{A,B}; \pi^1_{A,B})(b) \right)$; 
    \item The linear distributors:
    \begin{align*}
     \delta^L_{(A,a), (B,b), (C,c)}\c (A,a)\times ((B,b)+(C,c))\rarr ((A,a)\times (B,b))+(C,c) \\
      \delta^R_{(A,a), (B,b), (C,c)}\c ((A,a) + ((B,b))  \times (C,c) \rarr (A,a) + ((B,b) \times (C,c))
    \end{align*}
are defined respectively as the linear distributors in $\bB$, that is
\begin{align*}
  \delta^L_{(A,a), (B,b), (C,c)} \c= \delta^L_{A,B,C} &&  \delta^R_{(A,a), (B,b), (C,c)} \c= \delta^R_{A,B,C} 
\end{align*}
\end{enumerate}
\end{theorem}
\begin{proof} Since $\bzero$ is an initial object in $\bB$, it is clear that $(\bzero, \bot_{F(\bzero)})$ is an initial object in $\int F$. Explicitly, for an object $(A,a)\in \int F$, the unique map $b_{(A,a)}\c (\bzero, \bot_{F(\bzero)})\rarr (A,a)$ is given by $b_A\c \bzero\rarr A$ in \bB, which is well-defined in $\int F$ as $\bot_{F(\bzero)}\leq F(b_A)(a) $. Similarly, since $\bzero$ is terminal in $\bB$, then $(\bzero, \top_{F(\bzero)})$ is a terminal object. The unique map $t_{(A,a)}\c (A,a)\rarr (\bzero, \top_{F(\bzero)})$ is given by $t_A\c A\rarr\bzero$ in \bB, which is well-defined as $a\leq \top_{F(A)} = F(t_A)(\top_{F(\bzero)})$ recalling that $F(t_A)$ is a bounded lattice homomorphism. 

Now let's explain why $\int F$ has products. Given objects $(A,a), (B,b)\in \int F$, their product $(A,a)\times(B,b)$ is equipped with projections $\pi^0_{(A,a), (B,b)}\c (A,a)\times (B,b) \rarr (A,a)$ and $\pi^1_{(A,a), (B,b)}\c (A,a)\times (B,b) \rarr (B,b)$ given by the projections $\pi^0_{A,B}\c A\times B\rarr A$ and $\pi^1_{A,B}\c A\times B\rarr B$ in $\bB$. These are well-defined since $F(\pi^0_{A,B})(a)\wedge F(\pi^1_{A,B})(b) \leq F(\pi^0_{A,B})(a)$ and similarly for the other projection. Given maps $f\c (C,c)\rarr (A,a)$ and $g\c (C,c)\rarr (B,b)$, there is a unique map $\langle f,g\rangle\c(C,c)\rarr (A,a)\times (B,b)$ given by their pairing $\langle f,g\rangle\c C\rarr A\times B$ in $\bB$. This is well defined since:
\begin{align*}
c&\leq F(f)(a) \wedge F(g)(b) = F(\langle f, g\rangle; \pi^0_{A,B})(a) \wedge  F(\langle f, g\rangle; \pi^1_{A,B})(b) \\
&= F(\langle f, g\rangle)(F(\pi^0_{A,B})(a)\wedge F(\pi^1_{A,B})(b)))  
\end{align*}

Next we explain why $\int F$ has coproducts. For object $(A,a), (B,b)\in \int F$, their coproduct $(A,a)+(B,b)$ is equipped with injections $\iota^0_{(A,a), (B,b)}\c (A,a)\rarr (A,a)+(B,b)$ and $\iota^1_{(A,a), (B,b)}\c (B,b)\rarr (A,a)+(B,b)$ given by the injections $\iota^0_{A,B}\c A\rarr A+B$ and $\iota^1_{A,B}\c B\rarr A+B$ in $\bB$. These injections are well-defined by (\ref{eqn:adjoint_condition2}):
\begin{align*}
& F(\psi_{A,B}; \pi^0_{A,B})(a)\leq F(\psi_{A,B}; \pi^0_{A,B})(a)\vee F(\psi_{A,B}; \pi^1_{A,B})(b) \\
&\Rarr a \leq  F(\iota^0_{A,B})(F(\psi_{A,B}; \pi^0_{A,B})(a)\vee F(\psi_{A,B}; \pi^1_{A,B})(b))
\end{align*}
and similarly for the other injection. Now given maps $h\c (A,a)\rarr (C,c)$ and $k\c (B,b)\rarr (C,c)$, there is a unique map $[h,k]\c(A,a)+(B,b)\rarr (C,c)$ given by their copairing $[h,k]\c (A,a)+(B,b)\rarr (C,c)$ in $\bB$. This is well-defined by (\ref{eqn:adjoint_condition}):
\[ a\leq F(h)(c) = F(\iota^0_{A,B})(F([h,k])(c)) \Rarr F(\psi_{A,B}; \pi^0_{A,B})(a) \leq F([h,k])(c)\]
\[ b\leq F(k)(c) = F(\iota^1_{A,B})(F([h,k])(c)) \Rarr F(\psi_{A,B}; \pi^1_{A,B})(b) \leq F([h,k])(c)\]
and therefore, we get that $F(\psi_{A,B}; \pi^0_{A,B})(a)\vee F(\psi_{A,B}; \pi^1_{A,B})(b) \leq F([h,k])(c)$. 


So $\int F$ has both finite products and finite coproducts. It remains to show that it is also a LDC. Let us explain why the linear distributors are well-defined. To do so, we observe that in a bounded distributive lattice, we always have the  linear distributivity inequality $x\wedge (y\vee z)\leq (x\wedge y)\vee z$. Moreover, we also have the following equalities: 
\begin{align*}
&F(\delta^L_{A,B,C}; \psi_{A\times B, C};\pi^0_{A\times B, C}; \pi^0_{A,B})(a) \\
&= F((1_A\times\psi_{B,C}); \alpha_{A,B,C}^{-1}; \pi^0_{A\times B, C}; \pi^0_{A,B})(a) = F(\pi^0_{A,B+C})(a) 
\end{align*}
\begin{align*}
& F(\delta^L_{A,B,C}; \psi_{A\times B, C};\pi^0_{A\times B, C}; \pi^1_{A,B})(b) \\
&= F((1_A\times\psi_{B,C}); \alpha_{A,B,C}^{-1}; \pi^0_{A\times B, C}; \pi^1_{A,B})(b) = F(\pi^1_{A,B+C};\psi_{B,C}; \pi^0_{B,C})(b) \end{align*}
\begin{align*}
& F(\delta^L_{A,B,C}; \psi_{A\times B, C};\pi^1_{A\times B, C})(c) \\
&= F((1_A\times\psi_{B,C}); \alpha_{A,B,C}^{-1}; \pi^1_{A\times B, C})(a) = F(\pi^1_{A,B+C};\psi_{B,C}; \pi^1_{B,C})(c) 
\end{align*}
As such, for the left linear distributor, we get that: 
\begin{align*}
&F(\pi^0_{A,B+C})(a) \wedge (F(\pi^1_{A,B+C}; \psi_{B,C}; \pi^0_{B,C})(b)\vee F(\pi^1_{A,B+C};\psi_{B,C};\pi^1_{B,C})(c)) \\
&\leq F(\delta^L_{A,B,C})((F(\psi_{A\times B, C};\pi^0_{A\times B, C}; \pi^0_{A,B})(a) \wedge F(\psi_{A\times B, C};\pi^0_{A\times B, C}; \pi^1_{A,B})(b))\\
&\qquad\vee F(\psi_{A\times B, C}; \pi^1_{A\times B, C})(c))
\end{align*}
So $\delta^L$ in $\int F$ is well-defined, and similarly for the right linear distributor $\delta^R$. Since composition in $\int F$ is the same as in $\bB$, it follows that all the linear distributor axioms are satisfied. So we conclude that $\int F$ is a CLDC, as desired. 
\end{proof}

\begin{remark}
The additional ``adjoint'' condition \eqref{eqn:adjoint_condition} follows from more general results about the Grothendieck construction. In fact, given a functor $F\c {\mathbb I}^{op}\rarr \mathsf{Cat}$, the conditions $F$ should satisfy to ensure $\int F$ inherits particular limits and colimits are studied in \cite[Sec 3.1 \& 3.2]{Tarlecki_Brustall_Goguen_1991}. In our context, ${\mathbb I}$ is a semi-additive category and, for $A$, $B$ and $f\c A\rarr B \in {\mathbb I}$, $F(A)$ and  $F(B)$ are bounded distributive lattices and $F(f)$ is a bounded lattice homomorphism. This implies all the necessary conditions such that $\int F$ has a terminal object, binary products and an initial object are satisfied. The existence of coproducts requires slightly more however. We need $F(\iota^j_{A,B})$ to have left adjoints. These adjoints are further required in our construction to be precisely given by $F(\psi_{A,B}; \pi^j_{A,B})$ in order for the linear distributors to be inherited by $\int F$ as well.
\end{remark}

We conclude this paper by presenting two examples of the above construction, both of which result in CLDCs which are neither semi-additive or posetal. The first recovers the product of a semi-additive category and a bounded distributive lattice. While the second is constructed using ideals of commutative monoids.  

\begin{example} 
Let $\bB$ be a semi-additive category and ${\mathcal L}$ be a bounded distributive lattice. Then, define $F\c \bB^{op}\rarr \mathsf{BDL}$ to be the constant functor mapping every object in $\bB$ to the lattice ${\mathcal L}$ and every map in $\bB$ to the identity lattice homomorphism $1_{\mathcal L}$. This functor trivially satisfies condition  \eqref{eqn:adjoint_condition} and it is immediate that the CLDC $\int F$ is precisely $\bB\times{\mathcal L}$. If neither $\bB$ and $\mathcal{L}$ are trivial, then $\bB\times{\mathcal L}$ is a CLDC which is neither a bounded distributive lattice or a semi-additive category. 
\end{example}

\begin{example} 
Recall that the category of commutative monoids and monoid homomorphisms $\mathsf{CMON}$ has finite biproducts, with the zero object given by the trivial monoid $\bzero = \lbrace 0 \rbrace$ with one element and the binary biproduct of commutative monoids $M$ and $N$ given by the cartesian product $M\times N$ with the pointwise monoid structure. Now, given a monoid $M$, a subset $I\subseteq M$ is an {\bf ideal} if it is closed under the monoid operation, in the sense that:
\[ x\in I \quad\text{and}\quad m\in M\Rarr x+m\in I\]
For a commutative monoid $M$, let $\mathsf{Idl}(M)$ be the set of ideals of $M$. Now $\mathsf{Idl}(M)$ is a bounded distributive lattice, where the order is given by inclusion $\subseteq$, the meet is the intersection $\cap$, the join is the union $\cup$, the top element is $M$ itself, and the bottom element is the empty subset $\emptyset$. Then we define our functor $F\c \mathsf{CMON}^{op}\rarr \mathsf{BDL}$ to send a commutative monoid to its lattice of ideals, $F(M) = \mathsf{Idl}(M)$, and to send a monoid homomorphism $f: M \to N$ to its inverse image $F(f) = f^{-1}: \mathsf{Idl}(N) \to \mathsf{Idl}(M)$, which maps an ideal $J \in \mathsf{Idl}(N)$ to the ideal $f^{-1}(J) = \{m\in M~|~f(m)\in J\}$. Now since it is well-known that the inverse image is functorial and a lattice homomorphism, our functor $F\c \mathsf{CMON}^{op}\rarr \mathsf{BDL}$ is indeed well-defined. It additionally satisfies \eqref{eqn:adjoint_condition} as follows. Consider commutative monoids $M$ and $N$, with ideals $I\in \mathsf{Idl}(M)$ and $K\in\mathsf{Idl}(M\times N)$. Then:
\begin{align*}
    {{\iota^{0}}^{-1}_{M,N}}(K) = \lbrace x\in M ~\vert~ (x,0) \in K \rbrace && {{\pi^{0}}^{-1}_{M,N}}(I) = \lbrace (x,y)\in M\times N ~\vert~ x \in I \text{ and } y \in N \rbrace
\end{align*}
Now suppose that $I\subseteq  {{\iota^{0}}^{-1}_{M,N}}(K)$. We need to show that ${{\pi^{0}}^{-1}_{M,N}}(I)\subseteq K$. Let $(x,y) \in {{\pi^{0}}^{-1}_{M,N}}(I)$. This means that $x \in I$. By assumption, this means that $x \in  {{\iota^{0}}^{-1}_{M,N}}(K)$, which tells us that $(x,0) \in K$. Since $K$ is an ideal, we have that $(x,y) = (x,0) + (0,y) \in K$. As such, we have ${{\pi^{0}}^{-1}_{M,N}}(I)\subseteq K$ as desired. Similarly for the other injection and projection and we conclude $F$ satisfies \eqref{eqn:adjoint_condition}. Therefore, $\int F$ is a CLDC. The objects $(M,I)$ are pairs of a commutative monoid $M$ and an ideal $I\in\mathsf{Idl}(M)$, while a map $f\c (M,I)\rarr (N,J)$ is a monoid homomorphisms $f\c M\rarr N$ such that $I\subseteq f^{-1}(J)$, in other words $f$ restricts to a function $f|_{I}\c I \rarr J$. The terminal object is $(\bzero, \bzero)$, the initial object is $(\bzero, \emptyset)$, and the binary products and coproducts are given as follows:
\[ (M, I) \times (N,J) = (M\times N, {{\pi^0}^{-1}_{M,N}}(I)\cap {{\pi^1}^{-1}_{M,N}}(J)) = (M\times N, I\times J)\]
\[ (M, I) + (N,J) = (M\times N, {{\pi^0}^{-1}_{M,N}}(I)\cup {{\pi^1}^{-1}_{M,N}}(J)) = (M\times N, (I\times N) \cup (M\times J))\]
$\int F$ is a new example of a CLDC which is neither a bounded distributive lattice or a semi-additive category. 
\end{example}

\noindent \small \textbf{Funding.} Kudzman-Blais acknowledges the financial support of the Natural Sciences and Engineering Research Council of Canada (NSERC), under the grant
awarded to Richard Blute, the Australian Category Theory group at Macquarie University, with the support of a Scott Russell Johnson Postgraduate Scholarship, and the Japan Society for the Promotion of Science (JSPS), under a JSPS Postdoctoral Fellowship for Research in Japan (ID P25730). Lemay acknowledges funding by an ARC DECRA award (\# DE230100303) and this material is based upon work supported by the AFOSR under award number FA9550-24-1-0008.

\bibliographystyle{plain} 
\bibliography{bibliography}

\begin{appendices}

\section{Distributive Symmetric Monoidal Categories with Zero Objects}\label{app:proof_Kleisli_LDC}

In this appendix we give full details of why the Kleisli category of the exception monad for a distributive category is an isomix SLDC. However, it turns out that this follows from a more general result about \emph{distributive} symmetric monoidal categories with zero objects. Distributive symmetric monoidal categories are a generalization of distributive categories to a non-cartesian monoidal product. They describe the behavior of the Kleisli category of the exception monad for a distributive category, since $\times$ is a well-defined monoidal product on the Kleisli category, but not the binary product. 

So recall that a symmetric monoidal category $(\bX, \os, I)$ is {\bf distributive} \cite[Def 2.3.1]{Johnson_Yau_2024} if \bX\ has finite coproducts and the following canonical transformations 
    \[d^L_{A,B,C} = [1_{A}\os \iota^0_{B,C}, 1_{A}\os \iota^1_{B,C}]\c (A\os B)+(A\os C)\rarr A\os(B+C) \]
    \[d^R_{A, B, C} = [\iota^0_{A,B}\os 1_{C}, \iota^1_{A, B}\os 1_{C}]\c (A\os C)+(B\os C)\rarr(A+B)\os C \]
    \[ u^L_A = b_{\bot\os A}\c \bot\rarr\bot\os A \qquad\qquad u^R_A= b_{A\os \bot}\c \bot\rarr A\os \bot\]
are isomorphisms. There is a well-known folklore result that given such categories, we can build another symmetric monoidal product, which behaves as an ``either-or-both'' product. Indeed, given a distributive symmetric monoidal category $\bX$, we get a functor  $\cv = (\cdot + \cdot) + (\cdot\os\cdot)\c \bX\times\bX\rarr \bX$ equipped with the following natural isomorphisms 
\begin{align*}
    &{u^R_\cv}_A = {u^R_+}^{-1}_A; {u^R_+}^{-1}_{A+\bot}; (1_{A+\bot}+u^R_A) = \iota^0_{A,\bot}; \iota^0_{A+\bot, A\os \bot}\c A\rarr A\cv \bot \\
    & {u^L_\cv}_A ={u^L_+}^{-1}_A; {u^R_+}^{-1}_{\bot+A}; (1_{\bot+A}+u^L_A) = \iota^1_{\bot,A}; \iota^0_{\bot+A, \bot\os A}\c A\rarr \bot\cv A \\
    & {\sigma_\cv}_{A,B} = {\sigma_+}_{A,B} + {\sigma_\os}_{A,B}\c A\cv B \rarr B\cv A \\
    & {\alpha_\cv}_{A,B,C}\c (A\cv B)\cv C\rarr A\cv (B\cv C)
\end{align*}
where ${\alpha_\cv}_{A,B,C}$ is defined as the unique map induced by the following maps: 
\begin{align*}
    & A \xrightarrow{\iota^0_{A, B\cv C}} A+(B\cv C) \xrightarrow{\iota^0_{A+(B\cv C), A\os (B\cv C)}} A\cv (B\cv C) \\
    &B \xrightarrow{\iota^0_{B, C}} B+C \xrightarrow{\iota^0_{B+C, B\os C}} B\cv C \xrightarrow{\iota^1_{A, B\cv C}}  A+ (B\cv C) \xrightarrow{\iota^0_{A+(B\cv C), A\os (B\cv C)}} A\cv (B\cv C) \\
    & A\os B \xrightarrow{1_A\os \iota^0_{B,C}} A\os (B+C) \xrightarrow{1_A\os \iota^0_{B+C, B\os C}} A\os (B\cv C) \xrightarrow{\iota^1_{A+(B\cv C), A\os (B\cv C)}} A\cv (B\cv C) \\
    & C \xrightarrow{\iota^1_{B,C}} B+C \xrightarrow{\iota^0_{B+C, B\os C}} B\cv C \xrightarrow{\iota^1_{A, B\cv C}}A+(B\cv C)\xrightarrow{\iota^0_{A+(B\cv C), A\os (B\cv C)}} A\cv (B\cv C) \\
    &((A+B)+(A\os B))\os C \xrightarrow{{d^R_{A+B, A\os B, C}}^{-1}} ((A+B)\os C)+((A\os B)\os C) \\
    &\qquad  \xrightarrow{{d^R_{A,B,C}}^{-1}+1_{(A\os B)\os C}} ((A\os C)+(B\os C))+((A\os B)\os C) \xrightarrow{\chi_{A,B,C}}  A\cv (B\cv C)
\end{align*}
where $\chi_{A,B,C}$ is given as the unique morphisms induced by the following maps
\begin{align*}
    & A\os C \xrightarrow{1_A\os \iota^1_{B,C}} A\os (B+C) \xrightarrow{1_A\os \iota^0_{B+C, B\os C}} A\os (B\cv C) \xrightarrow{\iota^1_{A+(B\cv C), A\os (B\cv C)}}  A\cv (B\cv C) \\
    & B\os C \xrightarrow{\iota^1_{B+C, B\os C}} B\cv C \xrightarrow{\iota^1_{A, B\cv C}} A+(B\cv C) \xrightarrow{\iota^0_{A+(B\cv C), A\os (B\cv C)}}  A\cv (B\cv C) \\   
    &(A\os B)\os C \xrightarrow{{\alpha_\os}_{A,B,C}} A\os (B\os C) \xrightarrow{1_A\os \iota^1_{B+C, B\os C}} A\os (B\cv C)\\
    &\qquad\xrightarrow{\iota^1_{A+(B\cv C), A\os (B\cv C)}}  A\cv (B\cv C)
\end{align*}

\begin{lemma}
For a distributive symmetric monoidal category $\bX$, $(\bX, \cv, \bot)$ is a symmetric monoidal category.
\end{lemma}

We now show that for a distributive symmetric monoidal category whose initial object is in fact a zero object, then it is an isomix SLDC. This in particular captures the Kleisli category of an exception monad.

\begin{theorem}
Let $\bX$ be a distributive symmetric monoidal category with a  zero object $\bzero$. Then $\bX$ is an isomix SLDC, with tensor structure $(\bX, \cv, \bzero)$ and par structure $(\bX, +, \bzero)$, where the left linear distributor $\delta^L_{A,B,C}\c A\cv (B+C)\rarr ((A\cv B)+C)+\top$, which is explicitly of type:
\begin{align*}
(A+(B+C))+(A\os(B+C)) \rarr ((A+B)+(A\os B))+C
\end{align*}
is defined as the unique map induced by the following maps: 
\[A \xrightarrow{\iota^0_{A, B}} A+B \xrightarrow{\iota^0_{A+B, A\os B}} A\cv B \xrightarrow{\iota^0_{A\cv B, C}} (A\cv B)+C\]
\[B \xrightarrow{\iota^1_{A, B}} A+B \xrightarrow{\iota^0_{A+B, A\os B}} A\cv B \xrightarrow{\iota^0_{A\cv B, C}} (A\cv B)+C\]
\[C \xrightarrow{\iota^1_{A\cv B,C}} (A\cv B)+C\]
\[ A\os (B+C) \xrightarrow{1_A\os (1_B+t_C)} A\os (B+\bzero) \xrightarrow{1_A\os {u^R_+}_B} A\os B \xrightarrow{\iota^1_{A+B, A\os B}} A\cv B \xrightarrow{\iota^0_{A\cv B, C}} (A\cv B)+C\] 
and the right linear distributor $\delta^R_{A,B,C}\c(A+B)\cv C\rarr A+(B\cv C)$ , which is explicitly of type
\begin{align*}
((A+B) +C) + ((A+B) \os C)  \rarr A + ((B+C) + (B \os C))
\end{align*}
is defined as the unique map induced by the following maps
\[A \xrightarrow{\iota^0_{A, B\cv C}} A+(B\cv C)\]
\[B \xrightarrow{\iota^0_{B, C}} B+C \xrightarrow{\iota^0_{B+C, B\os C}} B\cv C \xrightarrow{\iota^1_{A, B\cv C}} A+(B\cv C)\]
\[C \xrightarrow{\iota^1_{B,C}} B+C\xrightarrow{\iota^0_{B+C, B\os C}} B\cv C \xrightarrow{\iota^1_{A, B\cv C}}  A+(B\cv C)\]
\[ (A+B)\os C \xrightarrow{(t_A+1_B)\os 1_C} (\bzero+B)\os C  \xrightarrow{{{u^L_+}_{B}}\os 1_C} B\os C \xrightarrow{\iota^1_{B+C, B\os C}} B\cv C\xrightarrow{\iota^1_{A, B\cv C}} A+(B\cv C) \] 
\end{theorem}
\begin{proof}
The proofs of the relevant equalities will mostly proceed by the universal properties of coproducts, in particular pre-composing both sides of the equalities by the relevant injections and comparing the results. 

First, we show the left linear distributor is a natural transformation. Consider a triple of maps $f\c A\rarr A^\prime, g\c B\rarr B'$ and $h\c C\rarr C'$ in \bX. Then 
\[ \delta^L_{A,B,C}; ((f\cv g)+h) = (f\cv (g+h)); \delta^L_{A^\prime,B',C'}\]
holds by the following computations.
\begin{align*}
    &\iota^0_{A,B+C}; \iota^0_{A+(B+C), A\os (B+C)}; \delta^L_{A,B,C}; ((f\cv g)+h) \\
    &=\iota^0_{A,B}; \iota^0_{A+B, A\os B}; \iota^0_{A\cv B, C}; ((f\cv g)+h) \\
    &= f; \iota^0_{A^\prime,B'}; \iota^0_{A^\prime+B', A^\prime\os B'}; \iota^0_{A^\prime\cv B', C'} \\ \\
    &\iota^0_{B,C}; \iota^1_{A,B+C}; \iota^0_{A+(B+C), A\os (B+C)}; \delta^L_{A,B,C}; ((f\cv g)+h) \\
    &=\iota^1_{A,B}; \iota^0_{A+B, A\os B}; \iota^0_{A\cv B, C}; ((f\cv g)+h) \\
    &= g; \iota^1_{A^\prime,B'}; \iota^0_{A^\prime+B', A^\prime\os B'}; \iota^0_{A^\prime\cv B', C'} \\ \\
    &\iota^1_{B,C}; \iota^1_{A,B+C}; \iota^0_{A+(B+C), A\os (B+C)}; \delta^L_{A,B,C}; ((f\cv g)+h) \\
    &=\iota^1_{A\cv B, C}; ((f\cv g)+h) \\
    &= h; \iota^1_{A^\prime\cv B', C'} \\ \\
    & \iota^1_{A+(B+C), A\os (B+C)}; \delta^L_{A,B,C}; ((f\cv g)+h) \\
    &=(1_A\os (1_B+t_C);{u^R_+}_B);\iota^1_{A+B, A\os B}; \iota^0_{A\cv B,C}; ((f\cv g)+h)\\
    &=(1_A\os (1_B+t_C);{u^R_+}_B);(f\os g); \iota^1_{A^\prime+B', A^\prime\os B'}; \iota^0_{A^\prime\cv B',C'}\\
    &=(f\os (g+t_C);{u^R_+}_{B'});\iota^1_{A^\prime+B', A^\prime\os B'}; \iota^0_{A^\prime\cv B',C'}\\ \\
    &\iota^0_{A,B+C}; \iota^0_{A+(B+C), A\os (B+C)}; (f\cv (g+h)); \delta^L_{A^\prime,B',C'} \\
    &= f; \iota^0_{A^\prime,B'+C'}; \iota^0_{A^\prime+(B'+C'), A^\prime\os (B'+C')}; \delta^L_{A^\prime,B',C'} \\
    &= f; \iota^0_{A^\prime,B'}; \iota^0_{A^\prime+B', A^\prime\os B'}; \iota^0_{A^\prime\cv B', C'} \\ \\
    &\iota^0_{B,C}; \iota^1_{A,B+C}; \iota^0_{A+(B+C), A\os (B+C)}; (f\cv (g+h)); \delta^L_{A^\prime,B',C'} \\
    &=g; \iota^0_{B',C'};\iota^1_{A^\prime,B'+C'}; \iota^0_{A^\prime+(B'+C'), A^\prime\os (B'+C')}; \delta^L_{A^\prime,B',C'} \\
    &= g; \iota^1_{A^\prime,B'}; \iota^0_{A^\prime+B', A^\prime\os B'}; \iota^0_{A^\prime\cv B', C'} \\ \\
    &\iota^1_{B,C}; \iota^1_{A,B+C}; \iota^0_{A+(B+C), A\os (B+C)}; (f\cv (g+h)); \delta^L_{A^\prime,B',C'} \\
    &= h; \iota^1_{B',C'};\iota^1_{A^\prime,B'+C'}; \iota^0_{A^\prime+(B'+C'), A^\prime\os (B'+C')}; \delta^L_{A^\prime,B',C'} \\
    &= h; \iota^1_{A^\prime\cv B', C'} \\ \\
    & \iota^1_{A+(B+C), A\os (B+C)}; (f\cv (g+h)); \delta^L_{A^\prime,B',C'} \\
    &= (f\os (g+h)); \iota^1_{A^\prime+(B'+C'), A^\prime\os (B'+C')};\delta^L_{A^\prime,B',C'} \\
    &=(f\os (g+h)); (1_{A^\prime}\os (1_{B'}+t_{C'});{u^R_+}_{B'});\iota^1_{A^\prime+B', A^\prime\os B'}; \iota^0_{A^\prime\cv B',C'} \\
    &=(f\os (g+t_C);{u^R_+}_{B'});\iota^1_{A^\prime+B', A^\prime\os B'}; \iota^0_{A^\prime\cv B',C'}    
\end{align*}

Second, we show that the right linear distributor is natural and that this construction will be a symmetric LDC at the same time, by proving \eqref{cc:lin_dist_braiding} holds:
\[ \delta^R_{A,B,C} = {\sigma_\cv}_{A+B, C}; (1_C \cv {\sigma_+}_{A,B}); \delta^L_{C,B,A}; ({\sigma_\cv}_{C,B}+1_A); {\sigma_+}_{B\cv C, A}\]
We consider the right-hand side pre-composed with the appropriate injections.
\begin{align*}
    &\iota^0_{A,B}; \iota^0_{A+B,C}; \iota^0_{(A+B)+C, (A+B)\os C}; {\sigma_\cv}_{A+B, C}; (1_C \cv {\sigma_+}_{A,B}); \delta^L_{C,B,A}; ({\sigma_\cv}_{C,B}+1_A); {\sigma_+}_{B\cv C, A} \\
    &= \iota^0_{A,B}; \iota^1_{C, A+B}; \iota^0_{C+(A+B), C\os (A+B)}; (1_C \cv {\sigma_+}_{A,B}); \delta^L_{C,B,A}; ({\sigma_\cv}_{C,B}+1_A); {\sigma_+}_{B\cv C, A}  \\
    &= \iota^0_{A,B}; {\sigma_+}_{A,B}; \iota^1_{C, B+A}; \iota^0_{C+(B+A), C\os(B+A)}; \delta^L_{C,B,A}; ({\sigma_\cv}_{C,B}+1_A); {\sigma_+}_{B\cv C, A}  \\
    &= \iota^1_{B,A}; \iota^1_{C, B+A}; \iota^0_{C+(B+A), C\os(B+A)}; \delta^L_{C,B,A}; ({\sigma_\cv}_{C,B}+1_A); {\sigma_+}_{B\cv C, A} \\
    &= \iota^1_{C\cv B, A}; ({\sigma_\cv}_{C,B}+1_A); {\sigma_+}_{B\cv C, A} \\
    &= \iota^1_{B\cv C,A};  {\sigma_+}_{B\cv C, A} \\
    &= \iota^0_{A, B\cv C} \\ \\
    &\iota^1_{A,B}; \iota^0_{A+B,C}; \iota^0_{(A+B)+C, (A+B)\os C}; {\sigma_\cv}_{A+B, C}; (1_C \cv {\sigma_+}_{A,B}); \delta^L_{C,B,A}; ({\sigma_\cv}_{C,B}+1_A); {\sigma_+}_{B\cv C, A} \\  
    &= \iota^1_{A,B}; \iota^1_{C, A+B}; \iota^0_{C+(A+B), C\os (A+B)}; (1_C \cv {\sigma_+}_{A,B}); \delta^L_{C,B,A}; ({\sigma_\cv}_{C,B}+1_A); {\sigma_+}_{B\cv C, A}  \\ 
    &= \iota^1_{A,B}; {\sigma_+}_{A,B}; \iota^1_{C, B+A}; \iota^0_{C+(B+A), C\os(B+A)}; \delta^L_{C,B,A}; ({\sigma_\cv}_{C,B}+1_A); {\sigma_+}_{B\cv C, A}  \\
    &= \iota^1_{C,B}; \iota^1_{C, B+A}; \iota^0_{C+(B+A), C\os(B+A)}; \delta^L_{C,B,A}; ({\sigma_\cv}_{C,B}+1_A); {\sigma_+}_{B\cv C, A}  \\
    &= \iota^1_{C,B}; \iota^0_{C+B, C\os B}; \iota^0_{C\cv B,A}; ({\sigma_\cv}_{C,B}+1_A); {\sigma_+}_{B\cv C, A} \\
    &= \iota^1_{C,B}; \iota^0_{C+B, C\os B}; {\sigma_\cv}_{C,B}; \iota^0_{B\cv C,A}; {\sigma_+}_{B\cv C, A}\\
    &= \iota^0_{B,C}; \iota^0_{B+C, B\os C}; \iota^1_{A,B\cv C} \\ \\
    &\iota^1_{A+B,C}; \iota^0_{(A+B)+C, (A+B)\os C}; {\sigma_\cv}_{A+B, C}; (1_C \cv {\sigma_+}_{A,B}); \delta^L_{C,B,A}; ({\sigma_\cv}_{C,B}+1_A); {\sigma_+}_{B\cv C, A} \\  
    &= \iota^0_{C,A+B}; \iota^0_{C+(A+B), C\os (A+B)}; (1_C \cv {\sigma_+}_{A,B}); \delta^L_{C,B,A}; ({\sigma_\cv}_{C,B}+1_A); {\sigma_+}_{B\cv C, A} \\  
    &= \iota^0_{C,B}; \iota^0_{C+B, C\os B}; \iota^0_{C\cv B, A}; ({\sigma_\cv}_{C,B}+1_A); {\sigma_+}_{B\cv C, A} \\  
    &= \iota^0_{C,B}; \iota^0_{C+B, C\os B}; {\sigma_\cv}_{C,B}; \iota^0_{B\cv C, A}; {\sigma_+}_{B\cv C, A} \\  
    &= \iota^1_{B,C}; \iota^0_{B+C, B\os C}; \iota^1_{A, B\cv C} \\ \\
    & \iota^1_{(A+B)+C, (A+B)\os C}; {\sigma_\cv}_{A+B, C}; (1_C \cv {\sigma_+}_{A,B}); \delta^L_{C,B,A}; ({\sigma_\cv}_{C,B}+1_A); {\sigma_+}_{B\cv C, A} \\ 
    &= {\sigma_\os}_{A+B, C}; \iota^1_{C+(A+B), C\os (A+B)}; (1_C \cv {\sigma_+}_{A,B}); \delta^L_{C,B,A}; ({\sigma_\cv}_{C,B}+1_A); {\sigma_+}_{B\cv C, A} \\ 
    &= {\sigma_\os}_{A+B, C}; (1_C\os {\sigma_+}_{A,B}); \iota^1_{C+(B+A), C\os (B+A)}; \delta^L_{C,B,A}; ({\sigma_\cv}_{C,B}+1_A); {\sigma_+}_{B\cv C, A} \\ 
    &= {\sigma_\os}_{A+B, C}; (1_C\os {\sigma_+}_{A,B}); (1_C\os (1_B+t_A); {u^R_+}_B); \iota^1_{C+B, C\os B}; \iota^0_{C\cv B, A}; ({\sigma_\cv}_{C,B}+1_A); {\sigma_+}_{B\cv C, A} \\
    &= {\sigma_\os}_{A+B, C}; (1_C\os (t_A+1_B); {\sigma_+}_{\top,B}; {u^R_+}_B); \iota^1_{C+B, C\os B}; {\sigma_\cv}_{C,B}; \iota^0_{B\cv C, A}; {\sigma_+}_{B\cv C, A}  \\
    &= {\sigma_\os}_{A+B, C}; (1_C\os (t_A+1_B); {u^L_+}_B); {\sigma_\os}_{C,B}; \iota^1_{B+C, B\os C}; \iota^1_{A, B\cv C} \\
    &=((t_A+1_B); {u^L_+}_B \os 1_C); {\sigma_\os}_{B, C};  {\sigma_\os}_{C,B}; \iota^1_{B+C, B\os C}; \iota^1_{A, B\cv C} \\
    &=((t_A+1_B); {u^L_+}_B \os 1_C); \iota^1_{B+C, B\os C}; \iota^1_{A, B\cv C} \\
\end{align*}
Comparing the above computations with the definition of $\delta^R$, we see \eqref{cc:lin_dist_braiding} holds. 

It remains now to show the coherence conditions between units and linear distributivities \eqref{cc:unit_lineardist}, associativities and linear distributivities \eqref{cc:assoc_lineardist}, and left and right linear distributivities \eqref{cc:leftright_lineardist}. The first condition from \eqref{cc:unit_lineardist} is ${u^L_\cv}_{A+B}; \delta^L_{\bzero, A,B} = {u^L_\cv}_A + 1_B$ and holds as follows.
\begin{align*}
&{u^L_\cv}_{A+B}; \delta^L_{\bzero, A,B} \\
&= \iota^1_{\bzero, A+B}; \iota^0_{\bzero+(A+B), \bzero\os(A+B)}; \delta^L_{\bzero, A,B} \\
&= [ \iota^1_{\bzero, A}; \iota^0_{\bzero+A, \bzero\os A}; \iota^0_{\bzero\cv A, B}, \iota^1_{\bzero\cv, B}] \\
&= \iota^1_{\bzero, A}; \iota^0_{\bzero+A, \bzero\os A} + 1_B \\
&= {u^L_\cv}_A + 1_B
\end{align*}

The second condition from \eqref{cc:unit_lineardist} is $\delta^L_{A,B,\bzero}; {u^R_+}_{A\cv B} = 1_A\cv {u^R_+}_A$. Consider pre-composing the left hand side by the relevant injections. 
\begin{align*}
    &\iota^0_{A, B+\bzero}; \iota^0_{A+(B+\bzero); A\cv (B+\bzero)}; \delta^L_{A,B,\bzero}; {u^R_+}_{A\cv B} \\
    &= \iota^0_{A,B}; \iota^0_{A+B, A\os B}; \iota^0_{A\cv B, \bzero}; {u^R_+}_{A\cv B} \\
    &= \iota^0_{A,B}; \iota^0_{A+B, A\os B} \\ \\
    &\iota^0_{B,\bzero}; \iota^1_{A, B+\bzero}; \iota^0_{A+(B+\bzero); A\cv (B+\bzero)}; \delta^L_{A,B,\bzero}; {u^R_+}_{A\cv B} \\
    &= \iota^1_{A,B}; \iota^0_{A+B, A\os B}; \iota^0_{A\cv B, \bzero}; {u^R_+}_{A\cv B} \\
    &= \iota^1_{A,B}; \iota^0_{A+B, A\os B} \\ \\    
    &\iota^1_{B,\bzero}; \iota^1_{A, B+\bzero}; \iota^0_{A+(B+\bzero); A\cv (B+\bzero)}; \delta^L_{A,B,\bzero}; {u^R_+}_{A\cv B} \\
    &= \iota^1_{A\cv B, \bzero}; {u^R_+}_{A\cv B} \\
    &= b_{A\cv C} \\ \\  
    &\iota^1_{A+(B+\bzero); A\cv (B+\bzero)}; \delta^L_{A,B,\bzero}; {u^R_+}_{A\cv B} \\
    &=(1_A\os (1_B+t_\bzero)); (1_A\os {u^R_+}_B); \iota^1_{A+B, A\os B}; \iota^0_{A\cv B, \bzero}; {u^R_+}_{A\cv B} \\
    &=(1_A\os {u^R_+}_B); \iota^1_{A+B, A\os B}   
\end{align*}

Then, consider the right-hand side 
\begin{align*}
    &1_A\cv {u^R_+}_B \\
    &= (1_A+{u^R_+}_B)+(1_A\os {u^R_+}_B) \\
    &= [[\iota^0_{A,B}, {u^R_+}_B; \iota^1_{A,B}]; \iota^0_{A+B, A\os B}, (1_A\os {u^R_+}_B); \iota^1_{A+B, A\os B}]\\
    &= [[\iota^0_{A,B}, [1_B, b_B]; \iota^1_{A,B}]; \iota^0_{A+B, A\os B}, (1_A\os {u^R_+}_B); \iota^1_{A+B, A\os B}] \\
    &= [[\iota^0_{A,B};\iota^0_{A+B, A\os B}, [\iota^1_{A,B};\iota^0_{A+B, A\os B}, b_{A\cv C}]];, (1_A\os {u^R_+}_B); \iota^1_{A+B, A\os B}] \\
    &= \delta^L_{A,B,\bzero}; {u^R_+}_{A\cv B}
\end{align*}

The third and fourth conditions from \eqref{cc:unit_lineardist} then follow by the definition of $\delta^R$ as a composite of braidings and $\delta^L$.

The first coherence condition from \eqref{cc:assoc_lineardist} is \[ \delta^L_{A\cv B, C, D}; ({\alpha_\cv}_{A,B,C}+ 1_D) = {\alpha_\cv}_{A,B,C+D}; (1_A\cv \delta^L_{B,C,D}); \delta^L_{A, B\cv C, D} \] Once more, we pre-compose with the relevant injections. Firstly, the left-hand side:
\begin{align*}
    &\iota^0_{A,B}; \iota^0_{A+B, A\os B}; \iota^0_{A\cv B, C+D}; \iota^0_{(A\cv B)+(C+D), (A\cv B)\os (C+D)}; \delta^L_{A\cv B, C, D}; ({\alpha_\cv}_{A,B,C}+ 1_D) \\
    &= \iota^0_{A,B}; \iota^0_{A+B, A\os B}; \iota^0_{A\cv B, C}; \iota^0_{(A\cv B)+C, (A\cv B)\os C}; \iota^0_{(A\cv B)\cv C, D}; ({\alpha_\cv}_{A,B,C}+ 1_D)\\
    &= \iota^0_{A,B}; \iota^0_{A+B, A\os B}; \iota^0_{A\cv B, C}; \iota^0_{(A\cv B)+C, (A\cv B)\os C}; {\alpha_\cv}_{A,B,C}; \iota^0_{A\cv (B\cv C), D} \\
    &= \iota^0_{A, B\cv C}; \iota^0_{A+(B\cv C), A\os (B\cv C)}; \iota^0_{A\cv (B\cv C), D} \\ \\
    &\iota^1_{A,B}; \iota^0_{A+B, A\os B}; \iota^0_{A\cv B, C+D}; \iota^0_{(A\cv B)+(C+D), (A\cv B)\os (C+D)}; \delta^L_{A\cv B, C, D}; ({\alpha_\cv}_{A,B,C}+ 1_D) \\
    &= \iota^1_{A,B}; \iota^0_{A+B, A\os B}; \iota^0_{A\cv B, C}; \iota^0_{(A\cv B)+C, (A\cv B)\os C}; {\alpha_\cv}_{A,B,C}; \iota^0_{A\cv (B\cv C), D} \\
    &= \iota^0_{B,C}; \iota^0_{B+C, B\os C}; \iota^1_{A, B\cv C}; \iota^0_{A+(B\cv C), A\os (B\cv C)};\iota^0_{A\cv (B\cv C), D} \\ \\
    & \iota^1_{A+B, A\os B}; \iota^0_{A\cv B, C+D}; \iota^0_{(A\cv B)+(C+D), (A\cv B)\os (C+D)}; \delta^L_{A\cv B, C, D}; ({\alpha_\cv}_{A,B,C}+ 1_D) \\
    &= \iota^1_{A+B, A\os B}; \iota^0_{A\cv B, C}; \iota^0_{(A\cv B)+C, (A\cv B)\os C}; \iota^0_{(A\cv B)\cv C, D};({\alpha_\cv}_{A,B,C}+ 1_D) \\
    &= \iota^1_{A+B, A\os B}; \iota^0_{A\cv B, C}; \iota^0_{(A\cv B)+C, (A\cv B)\os C};{\alpha_\cv}_{A,B,C}; \iota^0_{A\cv (B\cv C), D} \\
    &= (1_A\os \iota^0_{B,C}); (1_A\os \iota^0_{B+C, B\os C}); \iota^1_{A+(B\cv C), A\os (B\cv C)}; \iota^0_{A\cv (B\cv C), D} \\ \\
    & \iota^0_{C,D}; \iota^1_{A\cv B, C+D}; \iota^0_{(A\cv B)+(C+D), (A\cv B)\os (C+D)}; \delta^L_{A\cv B, C, D}; ({\alpha_\cv}_{A,B,C}+ 1_D) \\
    &= \iota^1_{A\cv B, C}; \iota^0_{(A\cv B)+C, (A\cv B)\os C}; \iota^0_{(A\cv B)\cv C, D}; ({\alpha_\cv}_{A,B,C}+ 1_D) \\
    &= \iota^1_{A\cv B, C}; \iota^0_{(A\cv B)+C, (A\cv B)\os C};{\alpha_\cv}_{A,B,C};\iota^0_{A\cv (B\cv C), D} \\
    &= \iota^1_{B,C}; \iota^0_{B+C, B\os C}; \iota^1_{A, B\cv C}; \iota^0_{A+(B\cv C), A\os (B\cv C)}; \iota^0_{A\cv (B\cv C), D} \\ \\
    & \iota^1_{C,D}; \iota^1_{A\cv B, C+D}; \iota^0_{(A\cv B)+(C+D), (A\cv B)\os (C+D)}; \delta^L_{A\cv B, C, D}; ({\alpha_\cv}_{A,B,C}+ 1_D) \\
    &= \iota^1_{(A\cv B)\cv C, D}; ({\alpha_\cv}_{A,B,C}+ 1_D) \\
    &= \iota^1_{A\cv (B\cv C), D} \\ \\
    & \iota^1_{(A\cv B)+C, (A\cv B)\os C}; \delta^L_{A\cv B, C, D}; ({\alpha_\cv}_{A,B,C}+ 1_D) \\
    &= (1_{A\cv B}\os (1_C+ t_D)); (1_{A\cv B}\os {u^R_+}_C); \iota^1_{(A\cv B)+C, (A\cv B)\os C}; \iota^0_{(A\cv B)\cv C, D}; ({\alpha_\cv}_{A,B,C}+ 1_D) \\
    &= (1_{A\cv B}\os (1_C+ t_D)); (1_{A\cv B}\os {u^R_+}_C); \iota^1_{(A\cv B)+C, (A\cv B)\os C}; {\alpha_\cv}_{A,B,C}; \iota^0_{A\cv (B\cv C), D} \\ 
    &= (1_{A\cv B}\os (1_C+ t_D)); (1_{A\cv B}\os {u^R_+}_C); {d^R}^{-1}_{A+B, A\os B, C}; ({d^R}^{-1}_{A,B,D}+ 1_{(A\os B)\os C}); \chi_{A,B,C}; \iota^0_{A\cv (B\cv C), D} \\ 
    &= {d^R}^{-1}_{A+B, A\os B, C+D}; ({d^R}^{-1}_{A,B,C+D}+1_{(A\os B)\os (C+D)}); \\
    &\quad (((1_A \os (1_C+ t_D); {u^R_+}_C) + (1_B \os (1_C+ t_D); {u^R_+}_C))+ (1_{A\os B}\os (1_C+ t_D); {u^R_+}_C)); \chi_{A,B,C}; \iota^0_{A\cv (B\cv C), D} \\
\end{align*}
Secondly, the right-hand side:
\begin{align*}
    &\iota^0_{A,B}; \iota^0_{A+B, A\os B}; \iota^0_{A\cv B, C+D}; \iota^0_{(A\cv B)+(C+D), (A\cv B)\os (C+D)};  {\alpha_\cv}_{A,B,C+D}; (1_A\cv \delta^L_{B,C,D}); \delta^L_{A, B\cv C, D} \\
    &= \iota^0_{A, B\cv (C+D)}; \iota^0_{A+(B\cv (C+D)), A\os(B\cv (C+D))};  (1_A\cv \delta^L_{B,C,D}); \delta^L_{A, B\cv C, D} \\
    &= \iota^0_{A, (B\cv C)+D}; \iota^0_{A+((B\cv C)+D), A\os((B\cv C)+D)}; \delta^L_{A, B\cv C, D} \\
    &= \iota^0_{A, B\cv C}; \iota^0_{A+(B\cv C), A\os (B\cv C)}; \iota^0_{A\cv (B\cv C), D} \\ \\
    &\iota^1_{A,B}; \iota^0_{A+B, A\os B}; \iota^0_{A\cv B, C+D}; \iota^0_{(A\cv B)+(C+D), (A\cv B)\os (C+D)};  {\alpha_\cv}_{A,B,C+D}; (1_A\cv \delta^L_{B,C,D}); \delta^L_{A, B\cv C, D} \\
    &= \iota^0_{B, C+D}; \iota^0_{B+(C+D), B\os (C+D)}; \iota^1_{A, B\cv (C+D)}; \iota^0_{A+(B\cv (C+D)), A\os (B\cv (C+D))}; (1_A\cv \delta^L_{B,C,D}); \delta^L_{A, B\cv C, D} \\
    &=\iota^0_{B, C+D}; \iota^0_{B+(C+D), B\os (C+D)}; \delta^L_{B, C, D}; \iota^1_{A, (B\cv C)+D};\iota^0_{A+((B\cv C)+D), A\os ((B\cv C)+D)}; \delta^L_{A, B\cv C, D} \\
    &= \iota^0_{B,C}; \iota^0_{B+C, B\os C}; \iota^0_{B\cv C, D}; \iota^1_{A, (B\cv C)+D};\iota^0_{A+((B\cv C)+D), A\os ((B\cv C)+D)}; \delta^L_{A, B\cv C, D} \\
    &= \iota^0_{B,C}; \iota^0_{B+C, B\os C};\iota^1_{A, B\cv C}; \iota^1_{A, B\cv C}; \iota^0_{A+(B\cv C), A\os (B\cv C)}; \iota^0_{A\cv (B\cv C), D} \\ \\
    & \iota^1_{A+B, A\os B}; \iota^0_{A\cv B, C+D}; \iota^0_{(A\cv B)+(C+D), (A\cv B)\os (C+D)}; {\alpha_\cv}_{A,B,C+D}; (1_A\cv \delta^L_{B,C,D}); \delta^L_{A, B\cv C, D}\\
    &= (1_A\os \iota^0_{B, C+D}); (1_A\os \iota^0_{B+(C+D), B\os (C+D)}); \iota^1_{A+(B\cv (C+D)), A\os (B\cv (C+D))}; (1_A \cv\delta^L_{B,C,D}); \delta^L_{A, B\cv C, D} \\
    &= (1_A\os \iota^0_{B, C+D}; \iota^0_{B+(C+D), B\os (C+D)}; \delta^L_{B,C D});\iota^1_{A+((B\cv C)+D), A\os ((B\cv C)+D)}; \delta^L_{A, B\cv C, D} \\
    &= (1_A\os \iota^0_{B,C}; \iota^0_{B+C, B\os C}; \iota^0_{B\cv C, D}); (1_A\os (1_{B\cv C}+t_D); {u^R_+}_{B\cv C}); \iota^1_{A+(B\cv C), A\os (B\cv C)}; \iota^0_{A\cv (B\cv C), D} \\
    &= (1_A\os \iota^0_{B,C}; \iota^0_{B+C, B\os C}); (1_A\os \iota^0_{B\cv C, \bzero}); (1_A\os {u^R_+}_{B\cv C}); \iota^1_{A+(B\cv C), A\os (B\cv C)}; \iota^0_{A\cv (B\cv C), D} \\
    &= (1_A\os \iota^0_{B,C});(1_A\os \iota^0_{B+C, B\os C}); \iota^1_{A+(B\cv C), A\os (B\cv C)}; \iota^0_{A\cv (B\cv C), D} \\ \\
    &\iota^0_{C, D}; \iota^1_{A\cv B, C+D}; \iota^0_{(A\cv B)+(C+D), (A\cv B)\os (C+D)}; {\alpha_\cv}_{A,B,C+D}; (1_A\cv \delta^L_{B,C,D}); \delta^L_{A, B\cv C, D}\\
    &= \iota^0_{C,D}; \iota^1_{B, C+D}; \iota^0_{B+(C+D), B\os(C+D)}; \iota^1_{A, B\cv (C+D)}; \iota^0_{A+(B\cv (C+D)), A\os (B\cv (C+D))}; (1_A\cv \delta^L_{B,C,D}); \delta^L_{A, B\cv C, D}\\
    &= \iota^0_{C,D}; \iota^1_{B, C+D}; \iota^0_{B+(C+D), B\os(C+D)}; \delta^L_{B,C,D}; \iota^1_{A, (B\cv C)+D};\iota^0_{A+((B\cv C)+D), A\os ((B\cv C)+D)}; \delta^L_{A, B\cv C, D}\\
    &= \iota^1_{B,C}; \iota^0_{B+C, B\os C}; \iota^0_{B\cv C, D}; \iota^1_{A, (B\cv C)+D}; \iota^0_{A+((B\cv C)+D), A\os ((B\cv C)+D)}; \delta^L_{A, B\cv C, D}\\
    &= \iota^1_{B,C}; \iota^0_{B+C, B\os C}; \iota^1_{A, B\cv C}; \iota^0_{A+(B\cv C), A\os (B\cv C)}; \iota^0_{A\cv (B\cv C), D} \\ \\
    &\iota^1_{C, D}; \iota^1_{A\cv B, C+D}; \iota^0_{(A\cv B)+(C+D), (A\cv B)\os (C+D)}; {\alpha_\cv}_{A,B,C+D}; (1_A\cv \delta^L_{B,C,D}); \delta^L_{A, B\cv C, D}\\
    &= \iota^1_{C,D}; \iota^1_{B, C+D}; \iota^0_{B+(C+D), B\os (C+D)}; \iota^1_{A, B\cv (C+D)}; \iota^0_{A+(B\cv (C+D)), A\os (B\cv (C+D))}; (1_A\cv \delta^L_{B,C,D}); \delta^L_{A, B\cv C, D}\\
    &= \iota^1_{C,D}; \iota^1_{B, C+D}; \iota^0_{B+(C+D), B\os (C+D)}; \delta^L_{B, C, D}; \iota^1_{A, (B\cv C)+D}; \iota^0_{A+((B\cv C)+D), A\os ((B\cv C)+D)}; \delta^L_{A, B\cv C, D}\\
    &= \iota^1_{B\cv C, D}; \iota^1_{A, (B\cv C)+D}; \iota^0_{A+((B\cv C)+D), A\os ((B\cv C)+D)}; \delta^L_{A, B\cv C, D}\\
    &= \iota^1_{A\cv(B\cv C), D} \\ \\
    & \iota^1_{(A\cv B)+(C+D), (A\cv B)\os (C+D)}; {\alpha_\cv}_{A,B,C+D}; (1_A\cv \delta^L_{B,C,D}); \delta^L_{A, B\cv C, D}\\
    &= {d^R}^{-1}_{A+B, A\os B, C+D}; ({d^R}^{-1}_{A,B,C+D}+1_{(A\os B)\os (C+D)}); \chi_{A, B, C+D}; (1_A\cv \delta^L_{B,C,D}); \delta^L_{A, B\cv C, D} \\
\end{align*}
It remains to show that 
\begin{align*}
&(((1_A \os (1_C+ t_D); {u^R_+}_C) + (1_B \os (1_C+ t_D); {u^R_+}_C))+ (1_{A\os B}\os (1_C+ t_D); {u^R_+}_C)); \chi_{A,B,C}; \iota^0_{A\cv (B\cv C), D} \\
&=\chi_{A, B, C+D}; (1_A\cv \delta^L_{B,C,D}); \delta^L_{A, B\cv C, D} 
\end{align*}
which holds by the following computations.

\begin{align*}
    & \iota^0_{A,\os(C+D), B\os (C+D)}; \iota^0_{(A\os (C+D))+(B\os (C+D)), (A\os B)+(C+D)}; {\small LHS} \\
    &= (1_A \os (1_C+ t_D); {u^R_+}_C); \iota^0_{A\os C, B\os C}; \iota^0_{(A\os C)+(B\os C), (A\os B)+C}; \chi_{A,B,C}; \iota^0_{A\cv (B\cv C), D}\\
    &= (1_A \os (1_C+ t_D); {u^R_+}_C); (1_A\os \iota^1_{B,C}); (1_A\os \iota^0_{B+C, B\os C}); \iota^1_{A+(B\cv C), A\os (B\cv C)}; \iota^0_{A\cv (B\cv C), D}\\
    &= (1_A \os [\iota^0_{C, \bzero}, t_D; \iota^1_{C, \bzero}]; {u^R_+}_C); (1_A\os \iota^1_{B,C}); (1_A\os \iota^0_{B+C, B\os C}); \iota^1_{A+(B\cv C), A\os (B\cv C)}; \iota^0_{A\cv (B\cv C), D}\\
    &= (1_A\os [\iota^1_{B,C}; \iota^0_{B+C, B\os C}, t_D]); \iota^1_{A+(B\cv C), A\os (B\cv C)}; \iota^0_{A\cv (B\cv C), D} \\ \\
    & \iota^1_{A,\os(C+D), B\os (C+D)}; \iota^0_{(A\os (C+D))+(B\os (C+D)), (A\os B)+(C+D)}; {\small LHS} \\
    &= (1_B \os (1_C+ t_D); {u^R_+}_C); \iota^1_{A\os C, B\os C}; \iota^0_{(A\os C)+(B\os C), (A\os B)+C}; \chi_{A,B,C}; \iota^0_{A\cv (B\cv C), D} \\
    &= (1_B \os (1_C+ t_D); {u^R_+}_C); \iota^1_{B+C, B\os C}; \iota^1_{A, B\cv C}; \iota^0_{A+(B\cv C), A\os (B\cv C)}; \iota^0_{A\cv (B\cv C), D} \\ \\
    &\iota^1_{(A\os (C+D))+(B\os (C+D)), (A\os B)+(C+D)}; {\small LHS} \\
    &= (1_{A\os B}\os (1_C+ t_D); {u^R_+}_C); \iota^1_{(A\os C)+(B\os C), (A\os B)+C}; \chi_{A,B,C}; \iota^0_{A\cv (B\cv C), D} \\
    &= (1_{A\os B}\os (1_C+ t_D); {u^R_+}_C); {\alpha_\os}_{A,B,C}; (1_A\os \iota^1_{B+C, B\os C}); \iota^1_{A+(B\cv C), A\os (B\cv C)};\iota^0_{A\cv (B\cv C), D} \\ \\
    & \iota^0_{A,\os(C+D), B\os (C+D)}; \iota^0_{(A\os (C+D))+(B\os (C+D)), (A\os B)+(C+D)}; {\small RHS} \\
    &= (1_A\os \iota^1_{B, C+D}; \iota^0_{B+(C+D), B\os (C+D)});\iota^1_{A+(B\cv (C+D)), A\os (B\cv (C+D))}; (1_A\cv \delta^L_{B,C,D}); \delta^L_{A, B\cv C, D} \\
    &= (1_A\os \iota^1_{B, C+D}; \iota^0_{B+(C+D), B\os (C+D)}; \delta^L_{B, C, D}); \iota^1_{A+((B\cv C)+D), A\os ((B\cv C)+D)}; \delta^L_{A, B\cv C, D} \\
    &= (1_A\os [\iota^1_{B,C}; \iota^0_{B+C, B\os C}; \iota^0_{B\cv C, D}, \iota^1_{B\cv C, D}]); (1_A\os (1_{B\cv C}+t_D); {u^R_+}_{B\cv C}); \iota^1_{A+(B\cv C), A\os (B\cv C)}; \iota^0_{A\cv (B\cv C), D} \\ 
    &= (1_A\os [\iota^1_{B,C}; \iota^0_{B+C, B\os C}, t_D]); \iota^1_{A+(B\cv C), A\os (B\cv C)}; \iota^0_{A\cv (B\cv C), D} \\ \\
    & \iota^1_{A,\os(C+D), B\os (C+D)}; \iota^0_{(A\os (C+D))+(B\os (C+D)), (A\os B)+(C+D)}; {\small RHS} \\
    &= \iota^1_{B+(C+D), B\os (C+D)}; \iota^1_{A, B\cv (C+D)}; \iota^0_{A+(B\cv (C+D)), A\os (B\cv (C+D))}; (1_A\cv \delta^L_{B,C,D}); \delta^L_{A, B\cv C, D} \\
    &= \iota^1_{B+(C+D), B\os (C+D)};\delta^L_{B,C,D}; \iota^1_{A, (B\cv C)+D}; \iota^0_{A+((B\cv C)+D), A\os ((B\cv C)+D)}; \delta^L_{A, B\cv C, D} \\
    &= (1_B\os (1_C+t_D); {u^R_+}_C); \iota^1_{B+C, B\os C}; \iota^0_{B\cv C, D}; \iota^1_{A, (B\cv C)+D}; \iota^0_{A+((B\cv C)+D), A\os ((B\cv C)+D)}; \delta^L_{A, B\cv C, D} \\
    &= (1_B\os (1_C+t_D); {u^R_+}_C); \iota^1_{B+C, B\os C}; \iota^1_{A, B\cv C}; \iota^0_{A+(B\cv C), A\os (B\cv C)}; \iota^0_{A\cv(B\cv C), D} \\ \\
    &\iota^1_{(A\os (C+D))+(B\os (C+D)), (A\os B)+(C+D)}; {\small RHS} \\
    &= {\alpha_\os}_{A,B,C+D}; (1_A\os \iota^1_{B+(C+D), B\os (C+D)}); \iota^1_{A+(B\cv (C+D)), A\os (B\cv (C+D))}; (1_A\cv \delta^L_{B,C,D}); \delta^L_{A, B\cv C, D}\\
    &= {\alpha_\os}_{A,B,C+D}; (1_A\os \iota^1_{B+(C+D), B\os (C+D)}; \delta^L_{B,C,D}); \iota^1_{A+((B\cv C)+D), A\os ((B\cv C)+D)}; \delta^L_{A, B\cv C, D}\\
    &= {\alpha_\os}_{A,B,C+D}; (1_A\os (1_B\os (1_C+t_D); {u^R_+}_C)); (1_A\os \iota^1_{B+C, B\os C}; \iota^0_{B\cv C, D}); \\
    &\quad (1_A\os (1_{B\cv C}+t_D);{u^R_+}_{B\cv C});\iota^1_{A+(B\cv C), A\os (B\cv C)}; \iota^0_{A\cv (B\cv C), D}\\
    &= (1_{A\os B}\os (1_C+t_D); {u^R_+}_C); {\alpha_\os}_{A,B,C}; (1_A\os \iota^1_{B+C, B\os C}; \iota^0_{B\cv C, D}); \\
    &\quad (1_A\os (1_{B\cv C}+t_D);{u^R_+}_{B\cv C}); \iota^1_{A+(B\cv C), A\os (B\cv C)}; \iota^0_{A\cv (B\cv C), D}\\
    &= (1_{A\os B}\os (1_C+t_D); {u^R_+}_C); {\alpha_\os}_{A,B,C}; (1_A\os \iota^1_{B+C, B\os C});\iota^1_{A+(B\cv C), A\os (B\cv C)}; \iota^0_{A\cv (B\cv C), D}
\end{align*}
This completes the proof of the first coherence condition from  \eqref{cc:assoc_lineardist}.

The remaining coherence conditions of \eqref{cc:assoc_lineardist} and \eqref{cc:leftright_lineardist} hold by similar calculations (which we exclude simply because of their length and repetitive nature). 
\end{proof}

\end{appendices}

\end{document}